\documentclass[11pt]{amsart}
\usepackage{amsfonts}
\usepackage{amsmath}
\usepackage{amssymb}
\usepackage{amsthm} 
\usepackage{bm}
\usepackage{enumerate}
\usepackage{color}
\usepackage{mathdots}
\usepackage[hidelinks]{hyperref}
\usepackage{tikz}
\usetikzlibrary{shapes.geometric, arrows}
\usepackage{caption}
\usepackage{adjustbox}
\usepackage{bbold}
\usepackage{pgfplots}
\usetikzlibrary{arrows.meta}
\usepackage{mathtools}

\usepackage{hyperref}
\newtheorem{theorem}{Theorem}
\newtheorem{lemma}{Lemma}

\newtheorem{definition}{Definition}

\newtheorem{thm}{Theorem}[section]
\newtheorem{cor}{Corollary}

\newtheorem{prob}[thm]{Problem}

\newtheorem*{rmk*}{Remark}

\newtheorem*{claim*}{Claim}
\DeclareMathOperator{\diam}{diam}

\usepackage{fancyhdr}

\title{On a problem of Erd\H{o}s and S\'ark\"ozy about sequences with no term dividing the sum of two larger terms}
\author{Benjamin Bedert\\
\tiny Mathematical Institute, University of Oxford}
\begin{document}
\maketitle
\begin{abstract}
In 1970, Erd\H{o}s and S\'ark\"ozy wrote a joint paper studying sequences of integers $a_1<a_2<\dots$ having what they called property P, meaning that no $a_i$ divides the sum of two larger $a_j,a_k$. In the paper, it was stated that the authors believed, but could not prove, that a subset $A\subset[n]$ with property P has cardinality at most $|A|\leqslant \left\lfloor \frac{n}{3}\right\rfloor+1$. In 1997, Erd\H{o}s offered \$100 for a proof or disproof of the claim that $|A|\leqslant \frac{n}{3}+C$, for some absolute constant $C$. We resolve this problem, and in fact prove that $|A|\leqslant\left\lfloor \frac{n}{3}\right\rfloor+1$ for $n$ sufficiently large.
\end{abstract}

\tableofcontents
\section{Introduction}
Let us begin with the following definition from a 1970 paper \cite{MR265312} of Erd\H{o}s and S\'ark\"ozy.
\begin{definition}
\normalfont Let $A\subset\mathbf{N}$. We say that $A$ has \textit{property P} if there are no three numbers $x,y,z\in A$ with $z<x,y$ and $z|x+y$.
\end{definition}
So a sequence has property P precisely if it contains no term which divides the sum of two larger terms.
The main result of \cite{MR265312} states that an infinite sequence $A$ of positive integers having property P must have density 0. The density of infinite sequences with property P has been studied in greater detail by various authors, see \cite{MR2115998},\cite{MR3607502},\cite{MR1816258}. In the original paper \cite{MR265312}, Erd\H{o}s and S\'ark\"ozy mention the following finite version of the problem.

\begin{prob}[Erd\H{o}s - S\'ark\"ozy \cite{MR265312}]
Let $n\geqslant 1$ be an integer and $A\subset\{1,2,\dots,n\}$ have property P. Must it be the case that $$|A|\leqslant\bigg\lfloor\frac{n}{3}\bigg\rfloor + 1 ?$$
\label{prob1}
\end{prob}
The paper mentions that Szemer\'edi proved that if a set $A\subset[n]$ has size $|A|>\left\lfloor \frac{n}{3}\right\rfloor+1$, then it contains distinct $x,y,z\in A$ with $z|x+y$ and $\frac{x+y}{z}\neq 2$. This conclusion is significantly weaker than what would be needed to contradict $A$ having property P however, and the author is not aware of any results in the literature improving on this partial result.
In fact, Erd\H{o}s mentioned this particular problem in a large number of his open problem papers \cite{MR0360509, MR532690,MR1018622,MR0374075,MR593525,MR1050324,MR1110810,MR1601631,MR1439273,MR1215590} written between 1970 and 1996. In the 1973 paper \cite{MR0360509} and several later papers, Erd\H{o}s asks for a proof of the slightly weaker claim that $|A|\leqslant \frac{n}{3}+O(1)$ if $A\subset[n]$ is a set with property P. In his final open problems paper \cite{MR1601631}, Erd\H{o}s offers $\$100$ for a resolution of this problem. 
\begin{prob}[Erd\H{o}s]
Is there is an absolute constant $C$ such that if $n\!\geqslant\!1$ and $A\subset\{1,2,\dots,n\}$ has property P, then $|A|\leqslant \frac{n}{3} + C?$
\label{prob2}
\end{prob} In some of these papers, Erd\H{o}s states that the bound $|A|\leqslant \left\lfloor \frac{n}{3}\right\rfloor +1$, if true, is optimal in light of the example $A=\left\{\left\lceil \frac{2n}{3}\right\rceil,\dots,n\right\}$. This is a typo however, and the exact bound should be $|A|\leqslant \left\lceil \frac{n}{3}\right\rceil$ with the corresponding tight example being $A=\left\{\left\lfloor \frac{2n}{3}\right\rfloor+1,\dots, n\right\}$ which is easily seen to have property P.
The previous best known bound existing in the literature seems to be Erd\H{o}s's observation that $|A|\leqslant \left\lceil\frac{n}{2}\right\rceil$ which is more or less trivial.\footnote{A set with property P certainly cannot contain two numbers with one dividing the other, and it is well-known and not hard to prove that a subset of $[n]$ with this property has size at most $\left\lceil\frac{n}{2}\right\rceil$.}

\bigskip

We will establish the following resolution of these problems.
\begin{theorem}
There is an absolute constant $C$ such that for all $n\in\mathbf{N}$, if $A\subset\{1,2,\dots,n\}$ has property P, then
$|A|\leqslant  \frac{n}{3}+C.$
\label{theo1}
\end{theorem}
\begin{theorem}
For all sufficiently large $n\in \mathbf{N}$, if $A\subset\{1,2,\dots,n\}$ has property P, then
$|A|\leqslant  \left\lceil\frac{n}{3}\right\rceil.$
Moreover, this bound is tight for all such $n$ since $\left\{\left\lfloor \frac{2n}{3}\right\rfloor+1,\dots, n\right\}$ is a subset of $[n]$ with property P and size $\left\lceil\frac{n}{3}\right\rceil$.
\label{theo2}
\end{theorem}
Theorem \ref{theo2} in fact implies Theorem \ref{theo1} by choosing $C$ sufficiently large. In order to give a streamlined and relatively easy to read version of the argument, we make no effort to optimise the value of $C$.

\begin{flushleft}
\textbf{Acknowledgements.}
The author would like to thank Zachary Chase and his supervisor Ben Green for introducing him to the problem, for several helpful discussions, and for providing feedback on earlier versions of the paper. The author also gratefully acknowledges financial support from the EPSRC.
\end{flushleft}
\section{Notation}
We write $\mathbf{Z}$ for the set of integers and $\mathbf{N}$ for the set of positive integers.
For sets $X,Y\subset\mathbf{N}$ we define the sumset $X+Y=\{x+y:x\in X, y\in Y\}$, the difference set $X-Y=\{x-y:x\in X, y\in Y\}$ and for a rational number $q$, $q\!\cdot\!X$ denotes the dilated set $\{qx: x\in X\}$. For real numbers $\alpha\leqslant \beta$, we define $$(\alpha,\beta]\vcentcolon= \{m\in \mathbf{N}:\alpha<m\leqslant \beta\}$$ and similarly for other types of intervals. We also write $[\beta]$ for $[1,\beta]$. Throughout the paper, whenever we write the word `interval', we mean a set of consecutive integers. For a set $X\subset \mathbf{Z}$, we define $\diam X \vcentcolon= \max X-\min X$ and $\gcd_*(X)\vcentcolon= \gcd(X-X)$ is the greatest common divisor of all differences $x-x'$ with $x,x'\in X$. Equivalently, $\gcd_*(X)$ is the largest integer $d$ such that $X$ is contained in an arithmetic progression with common difference $d$.
\medskip

In our proofs of Theorems \ref{theo1} and \ref{theo2}, we will work with a set $A\subset [n]$ having property P. We will need to consider parts of $A$ lying in various subintervals of $[n]$ and hence we will use the following notation for real numbers $0\leqslant\alpha<\beta\leqslant 1$:
$$A_{(\alpha,\beta]} \vcentcolon= A\cap (\alpha n,\beta n],$$ and similarly for other types of intervals. The value of $n$ is hidden in this notation, but it will be clear from context. Finally, for a positive integer $q$ and a residue $a\!\!\mod q$, we write
$$A_{(\alpha,\beta]}^{a(q)} \vcentcolon=A_{(\alpha,\beta]}\cap(a+q\cdot\mathbf{N})= A\cap (\alpha n,\beta n]\cap(a+q\cdot\mathbf{N}).$$

\section{Preliminaries}
We begin with some easy but useful observations.
\begin{lemma}
If $A$ has property P, then
\begin{itemize}
    \item [(1)] $A$ is disjoint from $k\!\cdot\! A$ for any integer $k\geqslant 2$.
    \item [(2)] In particular, for any integer $k\geqslant 2$, the sets $A\,, k\!\cdot\! A\,, k^2\!\cdot\!A\,, k^3\!\cdot\!A,\dots$ are pairwise disjoint.
    \item [(3)]  $2\!\cdot\!A$ is disjoint from $k\!\cdot\! A$ for any integer $k\geqslant3$.
\end{itemize}

\label{basicInteger}
\end{lemma}
\begin{proof}
First, if $A\cap (k\!\cdot\! A)$ is non-empty for some $k\geqslant 2$, then there exist $a,a'\in A$ with $a = ka'$ so that $a>a'$ and $a'|2ka'=a+a$ contradicting that $A$ has property P. This proves $(1)$ from which $(2)$ follows immediately. For $(3)$, we need to show that $2\!\cdot\!A$ and $k\!\cdot\!A$ are disjoint for $k\geqslant 3$. If $(2\!\cdot\!A)\cap(k\!\cdot\!A)\neq \emptyset$, then there exist $a,a'\in A$ with $2a=ka'$ so $a>a'$ and $a'|ka'=a+a$ giving a contradiction.
\end{proof}
\begin{lemma}
Let $A\subset[n]$ have property P, let $k,a,q$ be positive integers and let $0\leqslant\alpha\leqslant1$. Suppose that $B$ is a set such that $k\cdot B$ consists exclusively of integer multiples of numbers in $A_{[\alpha]}=A\cap[\alpha n]$, that every number in  $k\cdot B$ is congruent to $a\!\!\mod q$ and that $I$ is an interval so that $k\cdot B\subset I$. Then
\begin{equation}
\left|B\right|+\left|\left(A_{(\alpha,1]}+A_{(\alpha,1]}\right)\cap I\cap(a+q\cdot \mathbf{N})\right|< \frac{|I|}{q}+1.
\label{basicmult}
\end{equation}
\label{basicmultl}
\end{lemma}
\begin{proof}
    We show that $k\cdot B$ and $\left(A_{(\alpha,1]}+A_{(\alpha,1]}\right)\cap I\cap(a+q\cdot \mathbf{N})$ are disjoint sets. If not, there exists some $b\in B$ so that $kb\in A_{(\alpha,1]}+A_{(\alpha,1]}$ contradicting that $A$ has property P as $kb$ is a multiple of some number in $A_{[\alpha]}$ by assumption. Hence, $k\cdot B$ and $\left(A_{(\alpha,1]}+A_{(\alpha,1]}\right)\cap I\cap(a+q\cdot \mathbf{N})$ are disjoint sets contained in $I\cap(a+q\cdot \mathbf{N})$. As $I$ is an interval, we have the bound $\left|I\cap(a+q\cdot \mathbf{N})\right|<\frac{|I|}{q}+1$ so \eqref{basicmult} follows.
\end{proof}

We will crucially make use of the following theorem, which is a lesser-known version of Freiman's $3k-4$ theorem. Freiman proved that if $S$ is a set of integers with doubling $|S+S|\leqslant 3|S|-4$, then $S+S$ contains an arithmetic progression of length $2|S|-1$. We shall need the following generalisation due to Bardaji and Grynkiewicz \cite{MR2684124}. Recall that for a set $S$ of integers, we define $\diam S \vcentcolon= \max S-\min S$ and $\gcd_*(S)\vcentcolon= \gcd(S-S)$ is the greatest common divisor of all differences $s-s'$ with $s,s'\in S$. One can see from this definition that the $\gcd_*$ of a set $S$ is in fact the largest integer $d$ such that $S$ is contained in an arithmetic progression with common difference $d$.

\begin{theorem}[Bardaji \& Grynkiewicz \cite{MR2684124}, Corollary 1.2]
Let $S,T$ be non-empty subsets of $\mathbf{Z}$ with $\diam T \leqslant \diam S$ and $\gcd_*(S+T)=1$. If
\begin{equation*}
    |S+T| \leqslant |S|+2|T|-4
\end{equation*}
and either $\gcd{}_*(S)=1$ or $|S+T|\leqslant 2|S|+|T|-3$,
then $S+T$ contains an arithmetic progression with common difference 1 and length $|S|+|T|-1$.
\label{bardajigrynkiewicz}
\end{theorem}
For convenience, we state the following corollary which is enough for our purposes.
\begin{theorem}
Let $S,T$ be non-empty subsets of $\mathbf{Z}$. Then one of the following conclusions holds:
\begin{itemize}
    \item [(1)] $|S+T| \geqslant |S|+|T|+\min(|S|,|T|)-3$.
    \item [(2)] $S+T$ contains an arithmetic progression with common difference
$\gcd_*(S+T)$ and length $|S|+|T|-1$.
\end{itemize}
\label{freimain}
\end{theorem}
\begin{proof}
Let $S,T$ be non-empty subsets of $\mathbf{Z}$ and, after translating, we may assume that $\min S=\min T =0$. Suppose that $(1)$ does not hold so that $|S+T| \leqslant |S|+|T|+\min(|S|,|T|)-4$. Let $d=\gcd_*(S+T)$ and note that $d$ divides $\gcd_*(S)$ and $\gcd_*(T)$ so that each of the three sets $S,T$ and $S+T$ lies in an arithmetic progression with common difference $d$. As $\min S=\min T =0$, we see that $S$ and $T$ consist of multiples of $d$ only. Define $S'= \frac{1}{d}\cdot S$ and $T'=\frac{1}{d}\cdot T$ so $S'$ and $T'$ are non-empty subsets of $\mathbf{Z}$ with $\gcd_*(S'+T')=1$ and $|S'+T'|=|S+T|\leqslant|S'|+|T'|+\min(|S'|,|T'|)-4$. Theorem \ref{bardajigrynkiewicz} then implies that $S'+T'$ contains an arithmetic progression with common difference
$1$ and length $|S'|+|T'|-1$. As $S+T=d\cdot(S'+T')$, $(2)$ follows.
\end{proof}

\section{The Proof}
We are now ready to begin the proof of Theorems \ref{theo1} and \ref{theo2}. We will use induction on $n$ to prove the following theorem which simultaneously implies both Theorem \ref{theo1} and Theorem \ref{theo2}.
\begin{theorem}
There exist absolute constants $\delta>0$ and $C$ such that the following holds. Let $n\in \mathbf{N}$ and let $A\subset[n]$ be a set with property P. Then
\begin{equation}
    |A|\leqslant \max\left(\left\lceil\frac{n}{3}\right\rceil, \left(\frac{1}{3}-\delta\right)n+C\right).
\label{inductiontheo}
\end{equation}
\label{theo5}
\end{theorem}
Throughout the paper, we assume that $C$ is a sufficiently large constant. First note that when $n\leqslant C$, the bound \eqref{inductiontheo} holds trivially. So from now on we may assume that $n>C$ is sufficiently large. Our induction hypothesis is that for all $m<n$, the upper bound $\max\left(\left\lceil\frac{m}{3}\right\rceil, \left(\frac{1}{3}-\delta\right)m+C\right)$ holds for subsets of $[m]$ having property P. Assume henceforth that $A\subset[n]$ has property P. To prove \eqref{inductiontheo}, we split the argument into three cases based on the size of the set $$A_{(\frac{2}{3},1]} \vcentcolon= A\cap \left(\frac{2n}{3},n\right].$$ These three cases of our proof will be handled in the three independent sections 5, 6 and 7. It is interesting to note that for large enough $n$, the only examples of sets with property P and size very close to $\left\lceil \frac{n}{3}\right\rceil$ seem to be sets containing almost all of $\left(\frac{2n}{3},n\right]$.\footnote{In fact, this is true and it can be proved with similar arguments to those used in the proof of Theorem \ref{theo5}.} One might therefore expect that the case where $\left|A_{(\frac{2}{3},1]}\right|$ is relatively large would cause the most trouble in the proof. This does not seem to be true however, and the first case that we consider, where $A_{(\frac{2}{3},1]}$ has density at least $\frac{2}{3}$ on $\left(\frac{2n}{3},n\right]$, has a far simpler proof than the remaining cases.

\section{Case 1:  $\left|A_{(\frac{2}{3},1]}\right|\geqslant \frac{2n}{9}+\frac{4}{3}$.}

Note that $A_{(\frac{2}{3},1]}+A_{(\frac{2}{3},1]}\subset \left(\frac{4n}{3},2n\right]$ so we get the trivial estimate $$\left|A_{(\frac{2}{3},1]}+A_{(\frac{2}{3},1]}\right|\leqslant \left\lceil \frac{2n}{3}\right\rceil.$$ As $3\left|A_{(\frac{2}{3},1]}\right|\geqslant \frac{2n}{3}+4$ by assumption, we get $3\left|A_{(\frac{2}{3},1]}\right|-4\geqslant \left\lceil \frac{2n}{3}\right\rceil\geqslant \left|A_{(\frac{2}{3},1]}+A_{(\frac{2}{3},1]}\right|$ so $(1)$ in Theorem \ref{freimain} does not hold when $S=T=A_{(\frac{2}{3},1]}$. Hence, Theorem \ref{freimain} implies that conclusion $(2)$ from the same theorem holds. We show that $\gcd_*\left(A_{(\frac{2}{3},1]}\right)=1$. Indeed, $A_{(\frac{2}{3},1]}$ is contained in a progression with common difference $\gcd_*\left(A_{(\frac{2}{3},1]}\right) = d$ so if $d\geqslant 2$, then $\left|A_{(\frac{2}{3},1]}\right|< \frac{n}{6}+1$ since $A_{(\frac{2}{3},1]}\subset\left(\frac{2n}{3},n\right]$, but this contradicts the assumption of Case 1. Hence, conclusion $(2)$ in Theorem \ref{freimain} gives that $A_{(\frac{2}{3},1]}+A_{(\frac{2}{3},1]}$ contains an interval $Q$ of length at least $2\left|A_{(\frac{2}{3},1]}\right|-1> \frac{4n}{9}+1$. Note that any $x\in \left[\frac{4n}{9}+1\right]$ has an integer multiple in every interval of length $x$ and hence also in $Q\subset A_{(\frac{2}{3},1]}+A_{(\frac{2}{3},1]}$ so that $x\notin A$ as $A$ has property P. Now let $s= \min A$ so that $s>\frac{4n}{9}+1$ by what we just showed. Without loss of generality, we may assume that $s\leqslant \left\lfloor \frac{2n}{3}\right\rfloor$ as $A\subset[s,n]$ so the desired bound \eqref{inductiontheo} holds trivially otherwise. $A-\{s\}$ cannot contain two numbers summing to $0\!\!\mod s$ as else we could find two numbers in $A$ larger than $s=\min A$ with $s$ dividing their sum, which would violate property P. Hence $\big|A\cap (s,2s]\big|\leqslant \left\lfloor\frac{s-1}{2}\right\rfloor$. If $\frac{n}{2}\leqslant s\leqslant \left\lfloor \frac{2n}{3}\right\rfloor,$ this gives in total
$$|A|\leqslant 1+\left\lfloor\frac{s-1}{2}\right\rfloor\leqslant 1+\left\lfloor\frac{\left\lfloor\frac{2n}{3}\right\rfloor-1}{2}\right\rfloor\leqslant\left\lceil \frac{n}{3}\right\rceil.$$ 
In the remaining case where $\frac{4n}{9}+1< s<\frac{n}{2}$, we trivially bound the number of elements of $A$ in $(2s,n]$ by $n-2s$. In total we get
$$|A|\leqslant 1+\frac{s-1}{2}+n-2s = n+\frac{1}{2}-\frac{3s}{2} \leqslant \frac{n}{3}-1,$$ as $s>\frac{4n}{9}+1$. So we have proved the desired bound \eqref{inductiontheo} in Case 1.

\section{Case 2: $\frac{n}{6}+24\leqslant \left|A_{(\frac{2}{3},1]}\right|< \frac{2n}{9}+\frac{4}{3}$.}

In Case 2, $A_{(\frac{2}{3},1]}$ has density roughly between $\frac{1}{2}$ and $\frac{2}{3}$ on the interval $\left(\frac{2n}{3},n\right]$. We begin with a useful lemma about sets which have density greater than half on an interval.

\begin{lemma}
Let $U\subset[k+1,k+m]$ be a set of integers, let $q$ be a positive integer and $a$ be a residue modulo $q$. If $|U|\geqslant \frac{m}{2}+\frac{q}{2}$, then the number of integers in the sumset $U+U$ that are $a(\!\!\!\!\mod q)$ is at least $\frac{2}{q}|U|-1$.
\label{doublinglemma}
\end{lemma}
\begin{proof}
Let $U_i = U\cap (i+q\!\cdot\!\mathbf{N})$ and pair the sets $U_i,U_{a-i}$ (some of the $U_i$ may be paired with themselves). Now look at the pair for which $|U_i|+|U_{a-i}|$ is maximal, say it is $U_j,U_{a-j}$. Then certainly $|U_j|+|U_{a-j}|\geqslant \frac{2}{q}|U|\geqslant \frac{m}{q}+1$ and in particular both $U_j,U_{a-j}$ are non-empty as we can trivially bound $|U_i|< \frac{m}{q}+1$ for all $i$. Using the well-known lower bound $|X+Y|\geqslant|X|+|Y|-1$ for the sumset of two non-empty sets of integers $X,Y$, we get that $|U_j+U_{a-j}|\geqslant |U_j|+|U_{a-j}|-1\geqslant \frac{2}{q}|U|-1$ as desired.
\end{proof}
Note also that the same result holds true when we consider subsets $U$ of an arithmetic progression with common difference $d>1$ as long as there is no obvious modular reason preventing it.
\begin{lemma}
Let $U\subset\{k+d,k+2d,\dots,k+md\}$, let $q$ be a positive integer coprime to $d$ and $a$ be a residue modulo $q$. If $|U|\geqslant \frac{m}{2}+\frac{q}{2}$, then the number of integers in the sumset $U+U$ that are $a(\!\!\!\!\mod q)$ is at least $\frac{2}{q}|U|-1$.
\label{doublinglemmad}
\end{lemma} 
\begin{proof}
The proof is the same as that of Lemma \ref{doublinglemma}, except that we have to use the assumption that $d$ and $q$ are coprime to deduce the upper bound $|U_i|< \frac{m}{q}+1$ for all $i$, where $U_i=U\cap (i+q\!\cdot\!\mathbf{N})$.
\end{proof}

We return to the main analysis of Case 2. First we will construct from $A$ an auxiliary set $B_1$ as follows. For every number $a\in A$ with $a\leqslant \frac{2n}{3}$ there is a unique power of $2$, say $2^{j_a}$, so that $2^{j_a}a \in \left(\frac{n}{3},\frac{2n}{3}\right]$. Call $B_1$ the set of all numbers obtained in this way, so
\begin{equation}
B_1=\left\{2^{j_a}a:a\in A\cap\left[\frac{2n}{3}\right]\right\}
\label{dyadicembedding}
\end{equation} and note that $B_1$ is a subset of $\left(\frac{n}{3},\frac{2n}{3}\right]$. Observe that 
\begin{equation}
    |A| = |B_1|+\left|A_{(\frac{2}{3},1]}\right|,
    \label{Adecompo}
\end{equation} because $|B_1|=\left|A\cap\left[\frac{2n}{3}\right]\right|$ since coincidences of the form $2^{j_a}a = 2^{j_b}b$ with $a\neq b$ are impossible by conclusion (2) in Lemma \ref{basicInteger}. In other words, the map $a\mapsto 2^{j_a}a$ is an injection. Also observe that any number in $B_1$ is a multiple of a number in $A\cap\left[\frac{2n}{3}\right]$ so as $A$ has property P, we retain the property that $A_{(\frac{2}{3},1]}+A_{(\frac{2}{3},1]}$ contains no multiples of any element in $B_1$. Our basic proof strategy in Case 2 is to show that many numbers in $\left(\frac{n}{3},\frac{2n}{3}\right]$ do have a multiple in the sumset $A_{(\frac{2}{3},1]}+A_{(\frac{2}{3},1]}$, so as these numbers cannot lie in $B_1$, we get an upper bound on $|B_1|$ which hopefully is strong enough to let us conclude the desired bound on $|A|$ using \eqref{Adecompo}.
\bigskip

We cover $B_1\subset\left(\frac{n}{3},\frac{2n}{3}\right]$ with the following sets. On the left half of the interval $\left(\frac{n}{3},\frac{2n}{3}\right]$ we split up $B_1$ into residue classes modulo $3$, so let $$B_1^{\text{L},i(3)} = B_1\cap\left(\frac{n}{3},\frac{n}{2}\right]\cap (i+3\!\cdot\!\mathbf{N})$$ for $i=0,1,2$. On the right half of the interval $\left(\frac{n}{3},\frac{2n}{3}\right]$ we do the same but with residue classes modulo $4$, so let $$B_1^{\text{R},i(4)} = B_1\cap\left(\frac{n}{2},\frac{2n}{3}\right]\cap (i+4\!\cdot\!\mathbf{N})$$ for $i=0,1,2,3$. As we are in Case 2, we have that $\left|A_{(\frac{2}{3},1]}\right|\geqslant\frac{n}{6}+24$ so that $A_{(\frac{2}{3},1]}\subset \left(\frac{2n}{3},n\right]$ satisfies the assumption of Lemma \ref{doublinglemma} with modulus $q=12$. Applying Lemma \ref{doublinglemma} to the set $A_{(\frac{2}{3},1]}$, we conclude that for each $0\leqslant j<12$:
\begin{equation}
    \left|\left(A_{(\frac{2}{3},1]}+A_{(\frac{2}{3},1]}\right)\cap(j+12\!\cdot\!\mathbf{N})\right| \geqslant \frac{1}{6}\left|A_{(\frac{2}{3},1]}\right|-1,
    \label{q12bound}
\end{equation} and note that $A_{(\frac{2}{3},1]}+A_{(\frac{2}{3},1]}\subset \left(\frac{4n}{3},2n\right]$.
Further, for each $i$ observe that $4\!\cdot\!B_1^{\text{L},i(3)} \subset \left(\frac{4n}{3},2n\right]\cap(4i+12\!\cdot\!\mathbf{N})$ and that $3\!\cdot\!B_1^{\text{R},i(4)} \subset \left(\frac{4n}{3},2n\right]\cap(3i+12\!\cdot\!\mathbf{N})$. By definition \eqref{dyadicembedding} of $B_1$, the sets $3\cdot B_1$ and $4\cdot B_1$ consist of multiples of numbers in $A\cap\left[\frac{2n}{3}\right]$ so we can apply Lemma \ref{basicmultl} with $\alpha=\frac{2}{3}$, $q=12$, $I=\left(\frac{4n}{3},2n\right]$ and $B=B_1^{\text{L},i(3)},B_1^{\text{R},i(4)}$ for each $i$. Plugging in the lower bound \eqref{q12bound} in the inequality \eqref{basicmult} from Lemma \ref{basicmultl} gives
\begin{align*}
    \frac{1}{6}\left|A_{(\frac{2}{3},1]}\right|-1+\left|B_1^{\text{L},i(3)}\right|&\leqslant \frac{n}{18}+1, \\
    \frac{1}{6}\left|A_{(\frac{2}{3},1]}\right|-1+\left|B_1^{\text{R},i(4)}\right|&\leqslant \frac{n}{18}+1. 
\end{align*}
Hence we conclude
\begin{align}
    \left|A_{(\frac{2}{3},1]}\right|+6\left|B_1^{\text{L},i(3)}\right|&\leqslant \frac{n}{3}+12, \label{go1}\\
    \left|A_{(\frac{2}{3},1]}\right|+6\left|B_1^{\text{R},i(4)}\right|&\leqslant \frac{n}{3}+12.\label{go2}
\end{align}

Having obtained the two inequalities above, we may assume for the remainder of the argument in Case 2 that $\left|B_1^{\text{L},i(3)}\right|<\frac{|B_1|}{6}+2$ and $\left|B_1^{\text{R},i(4)}\right|<\frac{|B_1|}{6}+2$ for all $i$ as otherwise we could plug in \eqref{go1} or \eqref{go2} in \eqref{Adecompo} to conclude that $|A|=\left|A_{(\frac{2}{3},1]}\right|+|B_1|\leqslant \frac{n}{3}$. We will use these extra assumptions in the final part of the argument in Case 2.

\bigskip

The main idea behind our proof  in Case 2 is to apply Theorem \ref{freimain} in a suitable way to $A_{(\frac{2}{3},1]}$. It is tempting to try applying Theorem \ref{freimain} to the set $A_{(\frac{2}{3},1]}$ directly. This does not seem to be enough however, and we first split $A_{(\frac{2}{3},1]}$ into the sets $E$ and $O$ consisting of the even/odd numbers in $A_{(\frac{2}{3},1]}$. The idea is then to apply Theorem \ref{freimain} to whichever of the two sets $E$ or $O$ contains most of $A_{(\frac{2}{3},1]}$. We shall continue under the assumption that $|O|\geqslant \frac{\left|A_{(\frac{2}{3},1]}\right|}{2}$, but the same proof works when $|E|\geqslant \frac{\left|A_{(\frac{2}{3},1]}\right|}{2}$ (after interchanging the roles of $E$ and $O$ in what follows). Note that $\gcd_*(O)$ is even, and if it is at least 4 then we would get $|O|< \frac{n}{12}+1$ since $O\subset \left(\frac{2n}{3},n\right]$. Because we are assuming in Case 2 that $\left|A_{(\frac{2}{3},1]}\right|\geqslant \frac{n}{6}+24$, we must have that $|O|\geqslant \frac{n}{12}+12$ so that $\gcd_*(O) = 2$. We apply Theorem \ref{freimain} to the sumset $O+O$ to deduce that either this sumset has size at least $3|O|-3$, or else that it contains a long arithmetic progression, and we prove the desired bound on $|A|$ in both cases. 

\bigskip

Assume first that $|O+O|\leqslant 3|O|-4$, then conclusion $(2)$ in Theorem \ref{freimain} holds so that $O+O$ contains an arithmetic progression $Q$ with common difference $\gcd_*(O)=2$ and size $2|O|-1 \geqslant \left|A_{(\frac{2}{3},1]}\right|-1$. Also note that $O+O\subset \left(\frac{4n}{3},2n\right]$ is fully contained within the even integers. Hence we can find an even integer $t$ so that \begin{equation}
    Q = \left\{t+4,t+6, \dots,t+2\left|A_{(\frac{2}{3},1]}\right|\right\}\subset \left(\frac{4n}{3},2n\right].
    \label{Qdefi}
\end{equation} So $Q$ contains at least $\frac{|Q|-1}{2}$ multiples of 4, and at least $\frac{|Q|-2}{3}$ multiples of 6. Now, we divide all the multiples of 4 in $Q$ by 4 and note that the set of the resulting quotients is contained in $\frac{1}{4}\!\cdot\!Q\subset\left(\frac{n}{3},\frac{n}{2}\right]$ by \eqref{Qdefi}. Similarly we divide all the multiples of 3 in $Q$ by 3 and in this case the resulting quotients lie in $\frac{1}{3}\!\cdot\!Q\subset \left(\frac{4n}{9},\frac{2n}{3}\right]$. Let $A'$ be the set of all the quotients obtained in this way from $Q$, so \begin{equation}
    A' \vcentcolon= \left(\frac{1}{3}\!\cdot\!Q\cup\frac{1}{4}\!\cdot\!Q\right)\cap\mathbf{N}.
    \label{A'defin}
\end{equation} Note that $A'$ is a subset of $\left(\frac{n}{3}, \frac{2n}{3}\right]$ and that each element of $A'$ has an integer multiple in $Q\subset A_{(\frac{2}{3},1]}+A_{(\frac{2}{3},1]}$. We now show that $\frac{1}{3}\!\cdot\!Q$ and $\frac{1}{4}\!\cdot\!Q$ are disjoint. Suppose for a contradiction that $\frac{1}{3}\!\cdot\!Q$ and $\frac{1}{4}\!\cdot\!Q$ intersect, then it would have to be the case that $\max\frac{1}{4}\!\cdot\!Q \geqslant\min\frac{1}{3}\!\cdot\!Q$ so that plugging in the values of $\max Q$ and $\min Q$ from \eqref{Qdefi} would give $$\frac{t}{4}+\frac{\left|A_{(\frac{2}{3},1]}\right|}{2}\geqslant\max\frac{1}{4}\!\cdot\!Q \geqslant\min\frac{1}{3}\!\cdot\!Q\geqslant \frac{t}{3}+\frac{4}{3}$$ whence $\left|A_{(\frac{2}{3},1]}\right|\geqslant \frac{t}{6} +\frac{8}{3} > \frac{2n}{9}+2$ since $t> \frac{4n}{3}-4$ by \eqref{Qdefi}. This however contradicts our Case 2 assumption $\left|A_{(\frac{2}{3},1]}\right|\leqslant \frac{2n}{9}+\frac{4}{3}$. Hence we deduce \begin{equation}
  \left|A'\right|=\left|\left(\frac{1}{4}\!\cdot\!Q\right)\cap\mathbf{N}\right|+\left|\left(\frac{1}{3}\!\cdot\!Q\right)\cap\mathbf{N}\right|\geqslant \frac{|Q|-1}{2}+\frac{|Q|-2}{3}\geqslant \frac{5\left|A_{(\frac{2}{3},1]}\right|}{6}-2,
 \label{A'sizebound1}
\end{equation} because $|Q|\geqslant \left|A_{(\frac{2}{3},1]}\right|-1$ by the definition \eqref{Qdefi} of $Q$. Next, as $\left|A_{(\frac{2}{3},1]}\right|\geqslant\frac{n}{6}+24$ by the assumptions of Case 2, Lemma \ref{doublinglemma} gives that 
\begin{equation}
    \left|\left(A_{(\frac{2}{3},1]}+A_{(\frac{2}{3},1]}\right)\cap(3+6\!\cdot\!\mathbf{N})\right|\geqslant \frac{\left|A_{(\frac{2}{3},1]}\right|}{3}-1.
\label{3Modulo6many}
\end{equation}
So we can find many numbers in $A_{(\frac{2}{3},1]}+A_{(\frac{2}{3},1]}$ that are $3\!\!\mod \!6$ and consider the set $\frac{1}{3}\!\cdot\!\bigg(\left(A_{(\frac{2}{3},1]}+A_{(\frac{2}{3},1]}\right)\cap(3+6\!\cdot\!\mathbf{N})\bigg)$ of quotients obtained by dividing these numbers by 3. Clearly, this set is contained in $\frac{1}{3}\!\cdot\!\left(A_{(\frac{2}{3},1]}+A_{(\frac{2}{3},1]}\right)\subset\frac{1}{3}\!\cdot\!\left(\frac{4n}{3},2n\right]=\left(\frac{4n}{9},\frac{2n}{3}\right]$. Moreover, this set of quotients consists of odd numbers only so it is disjoint from $\left(\frac{1}{3}\!\cdot\!Q\right)\cap\mathbf{N}$ because $Q$ is a subset of $O+O$ and therefore contains only even numbers. Since this set of quotients $\frac{1}{3}\!\cdot\!\bigg(\left(A_{(\frac{2}{3},1]}+A_{(\frac{2}{3},1]}\right)\cap(3+6\cdot\mathbf{N})\bigg)$ is contained in $\left(\frac{4n}{9},\frac{2n}{3}\right]$, its intersection with $\left(\frac{1}{4}\!\cdot\!Q\right)\cap\mathbf{N}$ trivially has size at most $\left|\left(\frac{4n}{9},\frac{n}{2}\right]\cap(1+2\!\cdot\!\mathbf{N})\right| < \frac{n}{36}+1$ because $\frac{1}{4}\!\cdot\!Q\subset \left(\frac{n}{3},\frac{n}{2}\right]$ by \eqref{Qdefi}. We now add this set of quotients to $A'$ to obtain a larger set $A''$ defined by \begin{align*}
    A'' \vcentcolon&= A'\cup \frac{1}{3}\!\cdot\!\bigg(\left(A_{(\frac{2}{3},1]}+A_{(\frac{2}{3},1]}\right)\cap(3+6\!\cdot\!\mathbf{N})\bigg)\\
    &=\left(\left(\frac{1}{3}\!\cdot\!Q\cup\frac{1}{4}\!\cdot\!Q\right)\cap\mathbf{N}\right)\cup\frac{1}{3}\!\cdot\!\bigg(\left(A_{(\frac{2}{3},1]}+A_{(\frac{2}{3},1]}\right)\cap(3+6\!\cdot\!\mathbf{N})\bigg).
\end{align*} By \eqref{3Modulo6many}, we see that we have added at least $$\left|A''\setminus{A'}\right|>\left|\left(A_{(\frac{2}{3},1]}+A_{(\frac{2}{3},1]}\right)\cap(3+6\!\cdot\!\mathbf{N})\right|- \frac{n}{36}-1\geqslant\frac{\left|A_{(\frac{2}{3},1]}\right|}{3}-\frac{n}{36}-2$$ new elements to $A'$ to obtain $A''$. Combining this with the lower bound \eqref{A'sizebound1} gives\begin{equation}
    \left|A''\right| > \frac{5\left|A_{(\frac{2}{3},1]}\right|}{6}-2+\frac{\left|A_{(\frac{2}{3},1]}\right|}{3}-\frac{n}{36}-2= \frac{7\left|A_{(\frac{2}{3},1]}\right|}{6}-\frac{n}{36}-4.
    \label{Aprimebound}
\end{equation} 
As we noted right after the definition \eqref{A'defin} of $A'$, $A'$ is a subset of $\left(\frac{n}{3},\frac{2n}{3}\right]$ and each of its members has a multiple in $A_{(\frac{2}{3},1]}+A_{(\frac{2}{3},1]}$. The same is true for $A''$ as this set is obtained from $A'$ by adding the set $\frac{1}{3}\!\cdot\!\Bigg(\left(A_{(\frac{2}{3},1]}+A_{(\frac{2}{3},1]}\right)\cap(3+6\!\cdot\!\mathbf{N})\Bigg)\subset \left(\frac{4n}{9},\frac{2n}{3}\right]$ and all of its elements also have a multiple in $A_{(\frac{2}{3},1]}+A_{(\frac{2}{3},1]}$. As $A$ has property P and $B_1$ consists of multiples of numbers in $A\cap\left[\frac{2n}{3}\right]$ 
by definition \eqref{dyadicembedding}, $A''$ must be disjoint from $B_1$ and since both $B_1$ and $A''$ are subsets of $\left(\frac{n}{3},\frac{2n}{3}\right]$ we deduce that $\left\lceil \frac{n}{3}\right\rceil\geqslant \left|A''\right|+|B_1|$. Using the lower bound \eqref{Aprimebound} for $\left|A''\right|$ in this inequality then yields the desired bound for $|A|$ as follows
\begin{align*}
    \left\lceil \frac{n}{3}\right\rceil&\geqslant \left|A''\right|+|B_1| \\
    &>\frac{7\left|A_{(\frac{2}{3},1]}\right|}{6}-\frac{n}{36}-4+|B_1| \\
    &\geqslant  \left|A_{(\frac{2}{3},1]}\right|+|B_1| \\
    &= |A|,
\end{align*}
where for the third inequality we used that $\frac{\left|A_{(\frac{2}{3},1]}\right|}{6}\geqslant \frac{n}{36}+4$ because we assume that $\left|A_{(\frac{2}{3},1]}\right|\geqslant\frac{n}{6}+24$ in Case 2, and for the final equality we used \eqref{Adecompo}.

\bigskip

This leaves us with alternative $(1)$ in Theorem \ref{freimain}, and hence we can now assume that \begin{equation}
    |O+O|\geqslant 3|O|-3\geqslant \frac{3\left|A_{(\frac{2}{3},1]}\right|}{2}-3.
    \label{O_1expansionsum}
\end{equation} Note that the sumset $O+O$ is fully contained in the set of even numbers in $\left(\frac{4n}{3},2n\right]$. Recall that by the discussion following inequalities \eqref{go1} and \eqref{go2} we may assume that $\left|B_1^{\text{R},i(4)}\right|<\frac{|B_1|}{6}+2$ for all $i$. Hence we get that 
\begin{align}
\left|B_1^{\text{L},0(3)}\right|+\left|B_1^{\text{L},1(3)}\right|+\left|B_1^{\text{L},2(3)}\right|+ \left|B_1^{\text{R},2(4)}\right|&= |B_1|-\left|B_1^{\text{R},0(4)}\right|-\left|B_1^{\text{R},1(4)}\right|-\left|B_1^{\text{R},3(4)}\right|\nonumber\\
&>\frac{|B_1|}{2}-6.
    \label{B_1LRbound}
\end{align} Furthermore, the dilated sets $4\!\cdot\!B_1^{\text{L},0(3)},4\!\cdot\!B_1^{\text{L},1(3)},4\!\cdot\!B_1^{\text{L},2(3)}$ and $3\!\cdot\!B_1^{\text{R},2(4)}$ are all sets of even numbers contained in $\left(\frac{4n}{3},2n\right]$ and they are pairwise disjoint since they all lie in distinct residue classes modulo 12 (to be precise they lie in the classes $0, 4, 8$ and $6\!\mod 12$ respectively). These dilated sets also consist only of multiples of numbers in $B_1$ so they are all disjoint from $O+O\subset A_{(\frac{2}{3},1]}+A_{(\frac{2}{3},1]}$ as $A$ has property P. We conclude that the sets $4\!\cdot\!B_1^{\text{L},0(3)},4\!\cdot\!B_1^{\text{L},1(3)},4\!\cdot\!B_1^{\text{L},2(3)},3\!\cdot\!B_1^{\text{R},2(4)}$ and $O+O$ are pairwise disjoint sets of even numbers in $\left(\frac{4n}{3},2n\right]$. As there are less than $\frac{n}{3}+1$ even numbers in $\left(\frac{4n}{3},2n\right]$, we get
\begin{align*}
    \frac{n}{3}+1&> |O+O|+\left|4\!\cdot\!B_1^{\text{L},0(3)}\right|+\left|4\!\cdot\!B_1^{\text{L},1(3)}\right|+\left|4\!\cdot\!B_1^{\text{L},2(3)}\right|+\left|3\!\cdot\!B_1^{\text{R},2(4)}\right|\\
    &>\frac{3\left|A_{(\frac{2}{3},1]}\right|}{2}-3+\frac{|B_1|}{2}-6,
\end{align*}
using the lower bounds \eqref{O_1expansionsum} and \eqref{B_1LRbound}. From rearranging this inequality we obtain the bound $|B_1|< \frac{2n}{3}-3\left|A_{(\frac{2}{3},1]}\right|+20$, so after using that $\left|A_{(\frac{2}{3},1]}\right|+|B_1|= |A|$ by \eqref{Adecompo}, we get in total $$|A|=\left|A_{(\frac{2}{3},1]}\right|+|B_1|< \frac{2n}{3}-2\left|A_{(\frac{2}{3},1]}\right|+20< \frac{n}{3},$$ where in the final inequality we used the assumption that $\left|A_{(\frac{2}{3},1]}\right|\geqslant \frac{n}{6}+24$ in Case 2. This finishes the proof of Case 2.
\vfill
\pagebreak

\section{Case 3: $\left|A_{(\frac{2}{3},1]}\right|< \frac{n}{6}+24$.}

In the final case of the argument, we assume that $A_{(\frac{2}{3},1]}=A\cap\left(\frac{2n}{3},n\right]$ has density at most $\frac{1}{2}$ on $\left(\frac{2n}{3},n\right]$. In Case 2 we noted that no number in the auxiliary set $B_1$ that we defined in \eqref{dyadicembedding} can have a multiple in $A_{(\frac{2}{3},1]}+A_{(\frac{2}{3},1]}$ because $A$ has property P. As $A_{(\frac{2}{3},1]}$ was quite large in Case 2, this gave a strong enough upper bound on $|B_1|$ so that we could conclude by using that $|A|=\left|A_{(\frac{2}{3},1]}\right|+|B_1|$ by \eqref{Adecompo}. If only a relatively small fraction of $A$ lies in $A_{(\frac{2}{3},1]}$ however, as we are assuming in Case 3, it is crucial for the argument that we make use of sums in $A+A$ involving elements smaller than $\frac{2n}{3}$. Hence, in this section we will frequently make use of the set $A\cap \left(\frac{n}{2},n\right]$ and recall that we write 
$A_{(\frac{1}{2},1]} \vcentcolon= A\cap \left(\frac{n}{2},n\right]$. We introduce some notation for the following important sets
$$A_{(\frac{1}{2},1]}^{i(3)} \vcentcolon= A_{(\frac{1}{2},1]}\cap (i + 3\!\cdot\!\mathbf{N})=A\cap\left(\frac{n}{2},n\right]\cap (i + 3\!\cdot\!\mathbf{N})$$ for $i=0,1,2$. So the sets $A_{(\frac{1}{2},1]}^{0(3)},A_{(\frac{1}{2},1]}^{1(3)},A_{(\frac{1}{2},1]}^{2(3)}$ are a partition of $A_{(\frac{1}{2},1]}$ corresponding to residue classes modulo $3$. Observe that $A_{(\frac{1}{2},1]} = A_{(\frac{1}{2},\frac{2}{3}]}\cup A_{(\frac{2}{3},1]}$ and in Case 3 we will study the whole sumset $A_{(\frac{1}{2},1]}+A_{(\frac{1}{2},1]}$ instead of just the sumset $A_{(\frac{2}{3},1]}+A_{(\frac{2}{3},1]}$ that was sufficient for Cases 1 and 2. 
\bigskip

Before we start with the argument, observe that by induction on $n$ we may assume that for any $a\geqslant1$:
\begin{equation}
\big|A\cap\{n-a+1,\dots,n\}\big|> \left(\frac{1}{3}-\delta\right)a
\label{inductionbound}
\end{equation}
Indeed, the set $A\cap\left[n-a\right]$ is a subset of $A$ and therefore also has property P. So by the induction hypothesis, $\big|A\cap\left[n-a\right]\big|\leqslant \max\left(\left\lceil\frac{n-a}{3}\right\rceil,\left(\frac{1}{3}-\delta\right)(n-a)+C\right)$ and if \eqref{inductionbound} failed to hold, then we would get $|A|\leqslant\max\left(\left\lceil\frac{n-a}{3}\right\rceil,\left(\frac{1}{3}-\delta\right)(n-a)+C\right)+\left(\frac{1}{3}-\delta\right)a$ which implies the desired bound \eqref{inductiontheo}. Further, for any integers $k>1$ and $l\geqslant1$, the induction hypothesis gives the bound
\begin{equation}
    \left|A\cap(k\!\cdot\!\mathbf{N})\cap\left[\frac{n}{l}\right]\right|\leqslant \frac{n}{3kl}+C.
\label{inductionboundmultiples}
\end{equation}
This follows from the observation that as $A$ has property P, so does the set $A\cap(k\!\cdot\!\mathbf{N})\cap\left[\frac{n}{l}\right]$, and hence so does $\frac{1}{k}\!\cdot\!\big(A\cap(k\!\cdot\!\mathbf{N})\cap\left[\frac{n}{l}\right]\big)\subset \left[\frac{n}{kl}\right]$.

\bigskip

To start the main argument of Case 3, recall the construction \eqref{dyadicembedding} of the auxiliary set $B_1$ that we obtained by mapping $A\cap\left[\frac{2n}{3}\right]$ injectively into $\left(\frac{n}{3},\frac{2n}{3}\right]$ via powers of 2. Let us here consider the following set of quotients 
\begin{equation}
   A'''\vcentcolon= \frac{1}{3}\!\cdot\!\bigg( \left(A_{(\frac{1}{2},1]}+A_{(\frac{1}{2},1]}\right)\cap3\!\cdot\!\mathbf{N}\bigg),
\label{A''definition}
\end{equation} obtained by dividing all multiples of 3 in $A_{(\frac{1}{2},1]}+A_{(\frac{1}{2},1]}$ by 3, and note that it is also contained in $\left(\frac{n}{3},\frac{2n}{3}\right]$ since $A_{(\frac{1}{2},1]}+A_{(\frac{1}{2},1]}\subset(n,2n]$. We show that $A'''$ is disjoint from $B_1$. Suppose for a contradiction that some $b_1\in B_1$ coincides with some $a'''\in A'''$. Then $3b_1 = 3a'''\in A_{(\frac{1}{2},1]}+A_{(\frac{1}{2},1]}$ so $3b_1 = x+y$ for some $x,y\in A_{(\frac{1}{2},1]}$. Recall that by the construction \eqref{dyadicembedding} of $B_1$, $b_1\in B_1$ is of the form $2^{j_a}a$ for some $a\in A$ so we get that $3\cdot2^{j_a}a=x+y$. As $A$ has property P, we must therefore have that one of $x,y$, say $y$, satisfies $y\leqslant a\leqslant b_1$. But then we get $x = 3b_1-y\geqslant 2b_1\geqslant 2y> n$ since $y\in A_{(\frac{1}{2},1]}\subset\left(\frac{n}{2},n\right]$, a contradiction. Hence we conclude that $A'''$ and $B_1$ are disjoint subsets of the interval $\left(\frac{n}{3},\frac{2n}{3}\right]$ so that
\begin{equation}
    |B_1|+\left|A'''\right|\leqslant \left\lceil \frac{n}{3}\right\rceil.
    \label{dyadicembed}
\end{equation}
We make the following observation which will be important for a later argument. If neither of $A_{(\frac{1}{2},1]}^{1(3)},A_{(\frac{1}{2},1]}^{2(3)}$ is empty (we will prove this later in Lemma \ref{nonemptylemma}), then we may assume that
\begin{equation}
    \left|A_{(\frac{1}{2},\frac{2}{3}]}\right|\leqslant \frac{\left|A_{(\frac{2}{3},1]}\right|}{2}+1.
\label{A_2lowe}
\end{equation} Indeed, otherwise we would have that $\left|A_{(\frac{1}{2},\frac{2}{3}]}\right|\geqslant \frac{\left|A_{(\frac{2}{3},1]}\right|}{2}+\frac{3}{2}$, so $\left|A_{(\frac{1}{2},1]}\right|=\left|A_{(\frac{1}{2},\frac{2}{3}]}\right|+\left|A_{(\frac{2}{3},1]}\right|\geqslant \frac{3\left|A_{(\frac{2}{3},1]}\right|}{2}+\frac{3}{2}$. As neither of $A_{(\frac{1}{2},1]}^{1(3)},A_{(\frac{1}{2},1]}^{2(3)}$ is empty, the set of multiples of 3 in $A_{(\frac{1}{2},1]}+A_{(\frac{1}{2},1]}$ has size at least $\max\left(2\left|A_{(\frac{1}{2},1]}^{0(3)}\right|-1,\left|A_{(\frac{1}{2},1]}^{1(3)}\right|+\left|A_{(\frac{1}{2},1]}^{2(3)}\right|-1\right)\geqslant\frac{2\left|A_{(\frac{1}{2},1]}\right|}{3}-1$. By definition \eqref{A''definition}, we then get $\left|A'''\right|=\left|\left(A_{(\frac{1}{2},1]}+A_{(\frac{1}{2},1]}\right)\cap(3\!\cdot\mathbf{N})\right|\geqslant \frac{2\left|A_{(\frac{1}{2},1]}\right|}{3}-1\geqslant\left|A_{(\frac{2}{3},1]}\right|$. Plugging in this lower bound in \eqref{dyadicembed} would then give the desired result
\begin{align*}
    \left\lceil \frac{n}{3}\right\rceil &\geqslant |B_1|+\left|A'''\right| \\
    &\geqslant |B_1|+\left|A_{(\frac{2}{3},1]}\right| \\
    &= |A|,
\end{align*}
because $|A|=|B_1|+\left|A_{(\frac{2}{3},1]}\right|$ by \eqref{Adecompo}.

\bigskip

Having indicated how \eqref{dyadicembed} can be useful by proving \eqref{A_2lowe}, we now return to the main argument. To successfully make use of \eqref{dyadicembed} in general to obtain a good bound on $|B_1|$, we  want $A'''$ to be large, i.e. we want to be able to find many multiples of $3$ in $A_{(\frac{1}{2},1]}+A_{(\frac{1}{2},1]}$.\footnote{This approach is too optimistic in general as we may not be able to find enough multiples of $3$ in $A_{(\frac{1}{2},1]}+A_{(\frac{1}{2},1]}$ to obtain a strong enough bound on $|B_1|$ using \eqref{dyadicembed}. However, in this case we will obtain enough structural information about $A_{(\frac{1}{2},1]}$ to proceed by a different argument.} For this reason it is natural to look at the sumset $A_{(\frac{1}{2},1]}^{0(3)}+A_{(\frac{1}{2},1]}^{0(3)}$ or $A_{(\frac{1}{2},1]}^{1(3)}+A_{(\frac{1}{2},1]}^{2(3)}$ depending on whether $A_{(\frac{1}{2},1]}$ proportionally has more elements in $A_{(\frac{1}{2},1]}^{0(3)}$, or in $A_{(\frac{1}{2},1]}^{1(3)}\cup A_{(\frac{1}{2},1]}^{2(3)}$. Hence, we further split up the argument in Case 3 into two subcases depending on which of the two inequalities $\left|A_{(\frac{1}{2},1]}^{1(3)}\right|+\left|A_{(\frac{1}{2},1]}^{2(3)}\right|\geqslant \frac{2\left|A_{(\frac{1}{2},1]}\right|}{3},$ or $\left|A_{(\frac{1}{2},1]}^{0(3)}\right|\geqslant \frac{\left|A_{(\frac{1}{2},1]}\right|}{3}$ holds. The main argument in both of these subcases is the same and based on applying Theorem \ref{freimain} to $A_{(\frac{1}{2},1]}^{1(3)}+A_{(\frac{1}{2},1]}^{2(3)}$ in the first subcase, and to $A_{(\frac{1}{2},1]}^{0(3)}+A_{(\frac{1}{2},1]}^{0(3)}$ in the second.
Before starting with these subcases, we state a lemma that will be useful for both.
\begin{lemma}
In Case 3 we either have that the desired bound \eqref{inductiontheo} holds, or else that $A_{(\frac{1}{2},\frac{2}{3}]}$ has size at least
\begin{equation}
    \left|A_{(\frac{1}{2},\frac{2}{3}]}\right|\geqslant \frac{n}{48}-17.
\label{A_2genelowe}
\end{equation}
\label{lemma4}
\end{lemma}
This bound is in fact stronger than we need, but the proof of \eqref{A_2genelowe} involves a basic version of a more elaborate argument that we shall employ later and so we give full details for the benefit of the reader.
\begin{proof}
Recall that $A_{[\frac{1}{2}]}=A\cap\left[\frac{n}{2}\right]$. Similarly to the way we previously constructed $B_1$, we now construct a set $B_{\frac{1}{2}}$ from $A_{[\frac{1}{2}]}$ as follows. Note that for every $a\in A_{[\frac{1}{2}]}$ there is a power of $2$, say $2^{p_a}$, such that $2^{p_a}a\in \left(\frac{n}{4},\frac{n}{2}\right]$ and let $B_{\frac{1}{2}}\vcentcolon=\left\{2^{p_a}a:a\in A_{[\frac{1}{2}]}\right\}\subset\left(\frac{n}{4},\frac{n}{2}\right]$. The map $a\mapsto 2^{p_a}a$ is injective by conclusion (2) in Lemma \ref{basicInteger} so that $\left|B_{\frac{1}{2}}\right|=\left|A_{[\frac{1}{2}]}\right|$. We further modify this construction so that the resulting set has as many of its elements as possible lying in $\left(\frac{n}{3},\frac{n}{2}\right]$. This will give stronger bounds because we have good control over how many numbers in $\left(\frac{n}{3},\frac{n}{2}\right]$ have a multiple in $A_{(\frac{2}{3},1]}+A_{(\frac{2}{3},1]}$. Thus, let 
\begin{equation}
    Z_{\frac{1}{2}} \vcentcolon= \bigg(B_{\frac{1}{2}}\cap\left(\frac{n}{3},\frac{n}{2}\right]\bigg)\cup\bigg\{\frac{3}{2}\!\cdot b: b \in B_{\frac{1}{2}}\cap\left(\frac{n}{4},\frac{n}{3}\right]\text{ and $b\equiv 2\!\!\!\!\mod 4$}\bigg\}.
\label{Z_1/2defi}
\end{equation} 
It is clear that $Z_{\frac{1}{2}}\subset \left(\frac{n}{3},\frac{n}{2}\right]$ and we will show that
\begin{equation}\left|B_{\frac{1}{2}}\right|= \left|Z_{\frac{1}{2}}\right|+\bigg|B_{\frac{1}{2}}\cap \left\{ b: b \in\left(\frac{n}{4},\frac{n}{3}\right]\text{ and $b\equiv 0,1$ or $3\!\!\!\!\mod 4$.}\right\}\bigg|\leqslant \left|Z_{\frac{1}{2}}\right|+\frac{n}{16}+3.
\label{lemmalo}
\end{equation}
The final inequality in \eqref{lemmalo} follows from using a trivial upper bound on the number of integers in $\left(\frac{n}{4},\frac{n}{3}\right]$ that are $0,1$ or $3\!\!\mod 4$. So it is enough to prove the first equality in \eqref{lemmalo} which follows from \eqref{Z_1/2defi} if we can prove that there are no incidences of the form $\frac{3}{2}\!\cdot b = b'$ for any $b,b'\in B_{\frac{1}{2}}$ with $b\equiv 2\!\!\mod 4$. This is rather easy to prove since $b,b'$ are of the form $2^{p_a}a,2^{p_{a'}}a'$ for some $a,a'\in A$ respectively. But $b' = \frac{3}{2}\!\cdot b$ is odd as $b\equiv 2\!\!\mod 4$ so that $p_{a'}=0$ and $b'=a'$. Then we get $(3\!\cdot2^{p_a})a=2a'$ contradicting (3) in Lemma \ref{basicInteger}. Hence, \eqref{lemmalo} follows.
So we obtain that
\begin{align}
    \left|A\setminus A_{(\frac{1}{2},\frac{2}{3}]}\right| 
    &= \left|A_{(\frac{2}{3},1]}\right|+\left|A_{[\frac{1}{2}]}\right|
    = \left|A_{(\frac{2}{3},1]}\right|+\left|B_{\frac{1}{2}}\right|\nonumber\\
    &\leqslant\left|A_{(\frac{2}{3},1]}\right|+\left|Z_{\frac{1}{2}}\right|+\frac{n}{16}+3, \label{jajaja}
\end{align}
where we used \eqref{lemmalo} for the final inequality. To prove this lemma, it now suffices to show that 
\begin{equation}
    \left|A_{(\frac{2}{3},1]}\right|+\left|Z_{\frac{1}{2}}\right|\leqslant \frac{n}{4}+14
\label{jajafinal}
\end{equation} as \eqref{jajaja} then shows that $\left|A\setminus A_{(\frac{1}{2},\frac{2}{3}]}\right|\leqslant \frac{15n}{48}+17$ thus giving either the desired bound \eqref{inductiontheo} or else the lower bound \eqref{A_2genelowe}.
\medskip

To prove \eqref{jajafinal}, let us write $A_{(\frac{2}{3},1]}^{i(4)}=A_{(\frac{2}{3},1]}\cap(i+4\!\cdot\!\mathbf{N})$ for $i=0,1,2,3$. First assume that either $A_{(\frac{2}{3},1]}^{1(4)}$ or $A_{(\frac{2}{3},1]}^{3(4)}$ is empty. Then we obtain 
\begin{equation}
    \left|A_{(\frac{2}{3},1]}\cap (1+2\!\cdot\!\mathbf{N})\right|\leqslant \max\left(\left|A_{(\frac{2}{3},1]}^{1(4)}\right|,\left|A_{(\frac{2}{3},1]}^{3(4)}\right|\right)\leqslant \frac{n}{12}+1
    \label{jaja1}
\end{equation} using a trivial bound on the number of integers in $\left(\frac{2n}{3},n\right]$ in a given residue class $\mod 4$. We also have that \begin{equation}
    \left|Z_{\frac{1}{2}}\right|+\left|A_{(\frac{2}{3},1]}\cap (2\!\cdot\!\mathbf{N})\right|\leqslant  \frac{n}{6}+1
\label{jaja2}
\end{equation} by applying Lemma \ref{basicmultl} with $B=Z_{\frac{1}{2}}$, $k=4$, $q=4$ and $I= \left(\frac{4n}{3},2n\right]$ as $A_{(\frac{2}{3},1]}+A_{(\frac{2}{3},1]}$ contains $2\cdot \left(A_{(\frac{2}{3},2]}\cap (2\!\cdot\!\mathbf{N})\right)$ so at least $\left|A_{(\frac{2}{3},1]}\cap (2\!\cdot\!\mathbf{N})\right|$ many multiples of $4$ in $I$, and as it is easy to see from the definition \eqref{Z_1/2defi} of $Z_{\frac{1}{2}}$ that every number in $4\!\cdot\!Z_{\frac{1}{2}}$ is an integer multiple of some $a\in A_{[\frac{1}{2}]}$. Combining the inequalities \eqref{jaja1} and \eqref{jaja2} gives $\left|Z_{\frac{1}{2}}\right|+\left|A_{(\frac{2}{3},1]}\right|\leqslant \frac{n}{12}+1+\frac{n}{6}+1$ so \eqref{jajafinal} holds as desired.
\smallskip

Finally, assume that neither $A_{(\frac{2}{3},1]}^{1(4)}$ nor $A_{(\frac{2}{3},1]}^{3(4)}$ is empty. Then $\left|A_{(\frac{2}{3},1]}^{1(4)}+A_{(\frac{2}{3},1]}^{3(4)}\right|\geqslant \left|A_{(\frac{2}{3},1]}^{1(4)}\right|+\left|A_{(\frac{2}{3},1]}^{3(4)}\right|-1$, and we also have $\left|A_{(\frac{2}{3},1]}^{j(4)}+A_{(\frac{2}{3},1]}^{j(4)}\right|\geqslant 2\left|A_{(\frac{2}{3},1]}^{j(4)}\right|-1$ for $j=0,2$. Hence, $A_{(\frac{2}{3},1]}+A_{(\frac{2}{3},1]}$ contains at least $$\max\left(\left|A_{(\frac{2}{3},1]}^{1(4)}\right|+\left|A_{(\frac{2}{3},1]}^{3(4)}\right|,2\left|A_{(\frac{2}{3},1]}^{0(4)}\right|,2\left|A_{(\frac{2}{3},1]}^{2(4)}\right|\right)-1\geqslant \frac{\left|A_{(\frac{2}{3},1]}\right|}{2}-1$$ multiples of 4 in $\left(\frac{4n}{3},2n\right]$. Using this bound in Lemma \eqref{basicmultl} with $B=Z_{\frac{1}{2}}$, $k=4$, $q=4$ and $I= \left(\frac{4n}{3},2n\right]$ as above yields the inequality $\left|Z_{\frac{1}{2}}\right|+\frac{\left|A_{(\frac{2}{3},1]}\right|}{2}-1\leqslant \frac{n}{6}+1$. So $\left|A_{(\frac{2}{3},1]}\right|+\left|Z_{\frac{1}{2}}\right|\leqslant \frac{n}{6}+2+\frac{\left|A_{(\frac{2}{3},1]}\right|}{2}\leqslant \frac{n}{4}+14$ by our Case 3 assumption $\left|A_{(\frac{2}{3},1]}\right|\leqslant\frac{n}{6}+24$. This proves \eqref{jajafinal} and hence concludes the proof of Lemma \ref{lemma4}.
\end{proof}
We have now finished the general set-up for our proof of Case 3. Recall that $A_{(\frac{1}{2},1]}^{i(3)} = A_{(\frac{1}{2},1]}\cap(i+3\cdot\mathbf{N})$ for $i=0,1,2$. As we mentioned before, we now split up the proof of Case 3 into two subcases depending on which of the two inequalities $\left|A_{(\frac{1}{2},1]}^{1(3)}\right|+\left|A_{(\frac{1}{2},1]}^{2(3)}\right|\geqslant \frac{2\left|A_{(\frac{1}{2},1]}\right|}{3},$ or $\left|A_{(\frac{1}{2},1]}^{0(3)}\right|\geqslant \frac{\left|A_{(\frac{1}{2},1]}\right|}{3}$ holds.
\vfill
\pagebreak

\begin{flushleft}
\textbf{{\large Subcase 3.1: $\left|A_{(\frac{1}{2},1]}^{1(3)}\right|+\left|A_{(\frac{1}{2},1]}^{2(3)}\right|\geqslant\frac{2\left|A_{(\frac{1}{2},1]}\right|}{3}$.}}
\end{flushleft}
\medskip

We want to apply Theorem \ref{freimain} to the sumset $A_{(\frac{1}{2},1]}^{1(3)}+A_{(\frac{1}{2},1]}^{2(3)}$, but we first need to show that neither of $A_{(\frac{1}{2},1]}^{1(3)}, A_{(\frac{1}{2},1]}^{2(3)}$ is empty. 
\begin{lemma}
If $\left|A_{(\frac{1}{2},1]}^{1(3)}\right|+\left|A_{(\frac{1}{2},1]}^{2(3)}\right|\geqslant\frac{2\left|A_{(\frac{1}{2},1]}\right|}{3}$ and one of $A_{(\frac{1}{2},1]}^{1(3)}, A_{(\frac{1}{2},1]}^{2(3)}$ is empty, then the desired bound \eqref{inductiontheo} holds.
\label{nonemptylemma}
\end{lemma}
\begin{proof}
We prove the lemma when $A_{(\frac{1}{2},1]}^{2(3)}$ is empty as the case where $A_{(\frac{1}{2},1]}^{1(3)}$ is empty is analogous. First note that by using \eqref{inductionboundmultiples} with $l=1$ and $k=3$, we get \begin{equation}
\left|A\cap(3\cdot\mathbf{N})\right|\leqslant\frac{n}{9}+C.
\label{yaya0}
\end{equation} We now want to bound the number of elements in $A$ which are not multiples of $3$. We use a familiar type of construction to do this. For any number $a$ in $A\cap\left[\frac{n}{2}\right]$, there is a unique power of $2$, say $2^{p_a}$, so that $2^{p_a}a$ lies in $(\frac{n}{4},\frac{n}{2}]$.  Let
\begin{equation}
C_1\vcentcolon=\left\{2^{p_a}a:a\in A\cap\left[\frac{n}{2}\right]\text{ and } 3\nmid a\right\}.
\label{B_1'defi}
\end{equation} By (2) in Lemma \ref{basicInteger} there are no coincidences $2^{p_a}a=2^{p_b}b$ unless $a=b$, so that the map $a\mapsto 2^{p_a}a$ is an injection. Hence,
\begin{equation}
\left|C_1\right|=\left|\left\{a\in A_{[\frac{1}{2}]}: 3\nmid a\right\}\right|.
\label{B_1'decompo}
\end{equation}
Let $C_{1}^{i(3)} = C_1\cap(i+3\!\cdot\!\mathbf{N})$ for $i=1,2$ noting that by definition, $C_1$ contains no multiples of $3$. We have that $\left|A_{(\frac{1}{2},1]}\right|>\left(\frac{1}{3}-\delta\right)\frac{n}{2}$ by \eqref{inductionbound} so as $A_{(\frac{1}{2},1]}^{2(3)}$ is empty and $\left|A_{(\frac{1}{2},1]}^{1(3)}\right|+\left|A_{(\frac{1}{2},1]}^{2(3)}\right|\geqslant\frac{2\left|A_{(\frac{1}{2},1]}\right|}{3}$ by the assumptions of the lemma, we have that $\left|A_{(\frac{1}{2},1]}^{1(3)}\right|\geqslant\frac{2\left|A_{(\frac{1}{2},1]}\right|}{3}>\left(\frac{1}{3}-\delta\right)\frac{n}{3}$. So $A_{(\frac{1}{2},1]}^{1(3)}$ contains at least $\left(\frac{1}{3}-\delta\right)\frac{n}{3}$ of the numbers which are $1\!\!\mod 3$ in $\left(\frac{n}{2},n\right]$ and by choosing $\delta>0$ small enough, we can guarantee that $\left|A_{(\frac{1}{2},1]}^{1(3)}\right|>\left(\frac{1}{3}-\delta\right)\frac{n}{3}>\frac{n}{12}+3$. We can therefore use Lemma \ref{doublinglemmad} with $d=3$ and $q=4$ to deduce that $A_{(\frac{1}{2},1]}^{1(3)}+A_{(\frac{1}{2},1]}^{1(3)}$ contains at least $\frac{\left|A_{(\frac{1}{2},1]}^{1(3)}\right|}{2}-1$ multiples of 4. Note that every number in $A_{(\frac{1}{2},1]}^{1(3)}+A_{(\frac{1}{2},1]}^{1(3)}$ is  $2\!\!\mod 3$, so we have shown that $A_{(\frac{1}{2},1]}^{1(3)}+A_{(\frac{1}{2},1]}^{1(3)}$ contains at least $\frac{\left|A_{(\frac{1}{2},1]}^{1(3)}\right|}{2}-1$ numbers in $(n,2n]$ which are $8\!\!\mod12$. Observe that $4\cdot C_{1}^{2(3)}$ lies in $(n,2n]$ and also consists only of numbers which are $8\!\!\mod 12$. Lemma \ref{basicmultl} therefore gives \begin{equation}
\left|C_{1}^{2(3)}\right|+\frac{\left|A_{(\frac{1}{2},1]}^{1(3)}\right|}{2}-1\leqslant  \frac{n}{12}+1.
\label{yaya1}
\end{equation} 
The assumptions of lemma \ref{doublinglemmad} with $d=3$ and $q=5$ are also satisfied as $\left|A_{(\frac{1}{2},1]}^{1(3)}\right|>\left(\frac{1}{3}-\delta\right)\frac{n}{3}>\frac{n}{12}+3$, so $A_{(\frac{1}{2},1]}^{1(3)}+A_{(\frac{1}{2},1]}^{1(3)}$ contains at least $\frac{2\left|A_{(\frac{1}{2},1]}^{1(3)}\right|}{5}-1$ multiples of 5 and hence this many numbers in $(n,2n]$ congruent to $5\!\!\mod15$, again noting that $A_{(\frac{1}{2},1]}^{1(3)}+A_{(\frac{1}{2},1]}^{1(3)}$ consists of numbers which are $2\!\!\mod 3$ only. Note that $5\cdot C_{1}^{1(3)}\subset\left(n,\frac{5n}{2}\right]\cap(5+15\cdot\mathbf{N})$ consists only of multiples of numbers in $A_{[\frac{1}{2}]}$ so Lemma \ref{basicmultl} gives 
\begin{equation}
    \left|C_{1}^{1(3)}\right|+\frac{2\left|A_{(\frac{1}{2},1]}^{1(3)}\right|}{5}-1\leqslant  \frac{n}{10}+1.
\label{yaya2}
\end{equation}
Using \eqref{B_1'decompo} and that $A_{(\frac{1}{2},1]}^{2(3)}$ is empty by assumption, we obtain
\begin{align*}
|A|&= \left|A\cap(3\cdot\mathbf{N})\right|+\left|A_{(\frac{1}{2},1]}^{1(3)}\right|+\left|A_{(\frac{1}{2},1]}^{2(3)}\right|+|C_1|\\
&=\left|A\cap(3\cdot\mathbf{N})\right|+\left|A_{(\frac{1}{2},1]}^{1(3)}\right|+\left|C_{1}^{1(3)}\right|+\left|C_1^{2(3)}\right|\\
&\leqslant \frac{n}{9}+C+\left|A_{(\frac{1}{2},1]}^{1(3)}\right|+\frac{n}{12}+2-\frac{\left|A_{(\frac{1}{2},1]}^{1(3)}\right|}{2}+\frac{n}{10}+2-\frac{2\left|A_{(\frac{1}{2},1]}^{1(3)}\right|}{5}\\
&= \frac{\left|A_{(\frac{1}{2},1]}^{1(3)}\right|}{10}+\frac{53n}{180}+C+
4\leqslant \frac{14n}{45}+C+5
\end{align*} where we used \eqref{yaya0},\eqref{yaya1} and \eqref{yaya2} to obtain the third line, and for the final inequality we plugged in the trivial bound $\left|A_{(\frac{1}{2},1]}^{1(3)}\right|<\frac{n}{6}+1$ as there are at most this many numbers in $\left(\frac{n}{2},n\right]$ congruent to $1\!\!\mod 3$. This is stronger than the desired bound \eqref{inductiontheo} for $n$ sufficiently large and hence this concludes the proof of Lemma \ref{nonemptylemma}. 
\end{proof}
\bigskip

By Lemma \ref{nonemptylemma} may now assume that neither of $A_{(\frac{1}{2},1]}^{1(3)}, A_{(\frac{1}{2},1]}^{2(3)}$ is empty, so we can apply Theorem \ref{freimain} to $A_{(\frac{1}{2},1]}^{1(3)}+A_{(\frac{1}{2},1]}^{2(3)}$. Hence, we deduce that either conclusion $(1)$ from Theorem \ref{freimain} holds so $\left|A_{(\frac{1}{2},1]}^{1(3)}+A_{(\frac{1}{2},1]}^{2(3)}\right|\geqslant \left|A_{(\frac{1}{2},1]}^{1(3)}\right|+\left|A_{(\frac{1}{2},1]}^{2(3)}\right|+\min\left(\left|A_{(\frac{1}{2},1]}^{1(3)}\right|,\left|A_{(\frac{1}{2},1]}^{2(3)}\right|\right)-3$ or else that conclusion $(2)$ from Theorem \ref{freimain} holds so that $A_{(\frac{1}{2},1]}^{1(3)}+A_{(\frac{1}{2},1]}^{2(3)}$ contains an arithmetic progression $Q$ with common difference $d=\gcd_*\left(A_{(\frac{1}{2},1]}^{1(3)}+A_{(\frac{1}{2},1]}^{2(3)}\right)$ and size $|Q|\geqslant \left|A_{(\frac{1}{2},1]}^{1(3)}\right|+\left|A_{(\frac{1}{2},1]}^{2(3)}\right|-1$. Let us assume first that conclusion $(2)$ holds, so $A_{(\frac{1}{2},1]}^{1(3)}+A_{(\frac{1}{2},1]}^{2(3)}$ contains a long arithmetic progression $Q$. We show that we must have $d=3,6$ or $9$. Indeed, the sumset $A_{(\frac{1}{2},1]}^{1(3)}+A_{(\frac{1}{2},1]}^{2(3)}$ by definition consists of multiples of $3$ only so that $3|d$. Further, as $d=\gcd_*\left(A_{(\frac{1}{2},1]}^{1(3)}+A_{(\frac{1}{2},1]}^{2(3)}\right)$, we see that $A_{(\frac{1}{2},1]}^{1(3)}, A_{(\frac{1}{2},1]}^{2(3)}\subset\left(\frac{n}{2},n\right]$ both lie in progressions with common difference $d$. We can therefore trivially bound $\left|A_{(\frac{1}{2},1]}^{1(3)}\right|+ \left|A_{(\frac{1}{2},1]}^{2(3)}\right|\leqslant 2\left\lceil\frac{n}{2d}\right\rceil$. As we are in Subcase 3.1, we also get that \begin{equation}
\left|A_{(\frac{1}{2},1]}^{1(3)}\right|+ \left|A_{(\frac{1}{2},1]}^{2(3)}\right|\geqslant\frac{2\left|A_{(\frac{1}{2},1]}\right|}{3}> \left(\frac{1}{3}-\delta\right)\frac{n}{3} 
\label{3.1lowerbound}
\end{equation}
where the last inequality follows as $\left|A_{(\frac{1}{2},1]}\right|>\left(\frac{1}{3}-\delta\right)\frac{n}{2}$ by \eqref{inductionbound}. Comparing these upper and lower bounds on $\left|A_{(\frac{1}{2},1]}^{1(3)}\right|+ \left|A_{(\frac{1}{2},1]}^{2(3)}\right|$, we deduce that $d=3, 6$ or $9$. We deal with these three cases separately.
\medskip
\begin{flushleft}
\textbf{{3.1.1: Let $\gcd_*\left(A_{(\frac{1}{2},1]}^{1(3)}+A_{(\frac{1}{2},1]}^{2(3)}\right) = 9$.}}
\end{flushleft}
\medskip

In this case, $A_{(\frac{1}{2},1]}^{1(3)}$ and $A_{(\frac{1}{2},1]}^{2(3)}$ are subsets of $\left(\frac{n}{2},n\right]$ which each lie in a progression with common difference $d=9$, so we deduce that $A_{(\frac{1}{2},1]}^{1(3)}$ is contained in the set of numbers in $\left(\frac{n}{2},n\right]$ which are $a_1 \!\!\mod 9$ and that $A_{(\frac{1}{2},1]}^{2(3)}$ is contained in the set of numbers in $\left(\frac{n}{2},n\right]$ which are $a_2\!\!\mod 9$, for some $a_1\in\{1,4,7\}$ and $a_2\in\{2,5,8\}$.
First assume that $a_1+a_2\equiv 0\!\!\mod 9$, then $A_{(\frac{1}{2},1]}^{1(3)}+A_{(\frac{1}{2},1]}^{2(3)}$ contains at least $\left|A_{(\frac{1}{2},1]}^{1(3)}\right|+\left|A_{(\frac{1}{2},1]}^{2(3)}\right|-1\geqslant\left(\frac{1}{3}-\delta\right)\frac{n}{3}-1$ multiples of 9 in $(n,2n]$ by \eqref{3.1lowerbound}. Hence, from the definition \eqref{A''definition} we see that $A'''$ contains at least $\left(\frac{1}{3}-\delta\right)\frac{n}{3}-1$ many multiples of $3$ in $\left(\frac{n}{3},\frac{2n}{3}\right]$. Using that $B_1$ and $A'''$ are disjoint subsets of $\left(\frac{n}{3},\frac{2n}{3}\right]$ (which we proved in the paragraph above \eqref{dyadicembed}), there can be at most $\left\lceil\frac{n}{9}\right\rceil-\left(\frac{1}{3}-\delta\right)\frac{n}{3}+1\leqslant \frac{\delta n}{3}+2$ multiples of $3$ in $B_1$. By the definition \eqref{dyadicembedding} of $B_1$, this implies that there are also at most $\frac{\delta n}{3}+2$ multiples of $3$ in $A\cap\left[\frac{2n}{3}\right]$. To bound the number of elements of $A\cap\left[\frac{2n}{3}\right]$ which are not a multiple of $3$, we use the following construction. For any $a\in A\cap\left[\frac{n}{2}\right]$ which is not a multiple of $3$, there is a unique a power of $2$, say $2^{l_a}$, so that $2^{l_a}a$ lies in $\left(\frac{n}{2},n\right]$. Let $C_1'$ be the set of these numbers so $$C_1'= \left\{2^{l_a}a:a\in A\cap \left[\frac{n}{2}\right] \text{ and } 3\nmid a\right\},$$ and in particular $C_1'$ consists only of even numbers. By (2) in Lemma \ref{basicInteger} the map $a\mapsto 2^{l_a}a$ is injective so that 
\begin{equation}
    \left|A\cap\left[\frac{n}{2}\right]\right|\leqslant \left|C_1'\right|+\frac{\delta n}{3}+2,
\label{dada0}
\end{equation} where the extra terms $\frac{\delta n}{3}+2$ account for the possible existence of this many multiples of $3$ in $A\cap\left[\frac{n}{2}\right]$. As $A$ has property P, the sets $C_1',A_{(\frac{1}{2},1]}^{1(3)}$ and $A_{(\frac{1}{2},1]}^{2(3)}$ must be pairwise disjoint subsets of $\left(\frac{n}{2},n\right]$. Now note that every number in $A_{(\frac{1}{2},1]}^{1(3)}\cup A_{(\frac{1}{2},1]}^{2(3)}\cup C_1'$ lies in one of the residue classes $a_1,a_1+9,a_2,a_2+9,2,4,8,10,14,16 \mod 18$ because $C_1'$ contains only even numbers not divisible by $3$, while $A_{(\frac{1}{2},1]}^{1(3)}$ and $A_{(\frac{1}{2},1]}^{2(3)}$ only contain numbers which are $a_1,a_2\!\!\mod 9$. Amongst these $10$ numbers there are in fact only $8$ distinct residues modulo $18$ as one of $a_1,a_1+9$ must be even and not a multiple of $3$, and similarly for one of $a_2,a_2+9$. Hence, the trivial bound gives \begin{equation}
\left|A_{(\frac{1}{2},1]}^{1(3)}\right|+\left|A_{(\frac{1}{2},1]}^{2(3)}\right|+\left|C_1'\right|\leqslant 8\left\lceil \frac{n}{36}\right\rceil
\label{dada1}
\end{equation} as the interval $\left(\frac{n}{2},n\right]$ contains at most $\left\lceil \frac{n}{36}\right\rceil$ numbers in any given residue class $\!\!\mod 18$. As we are in Subcase 3.1, $\left|A_{(\frac{1}{2},1]}^{0(3)}\right|<\frac{\left|A_{(\frac{1}{2},1]}^{1(3)}\right|+\left|A_{(\frac{1}{2},1]}^{2(3)}\right|}{2}\leqslant \left\lceil \frac{n}{18}\right\rceil$ where the final bound follows as $A_{(\frac{1}{2},1]}^{1(3)}$ and $A_{(\frac{1}{2},1]}^{2(3)}$ lie in progressions with common difference $9$ by the assumption of case 3.1.1. So by combining \eqref{dada0}, \eqref{dada1} and this bound on $\left|A_{(\frac{1}{2},1]}^{0(3)}\right|$, we get
\begin{align*}
|A|&= \left|A\cap\left[\frac{n}{2}\right]\right|+\left|A_{(\frac{1}{2},1]}^{1(3)}\right|+\left|A_{(\frac{1}{2},1]}^{2(3)}\right|+\left|A_{(\frac{1}{2},1]}^{0(3)}\right|\\
&\leqslant \frac{\delta n}{3}+2+\left|C_1'\right|+\left|A_{(\frac{1}{2},1]}^{1(3)}\right|+\left|A_{(\frac{1}{2},1]}^{2(3)}\right|+\left\lceil\frac{n}{18}\right\rceil\\
&\leqslant  \frac{\delta n}{3}+2+8\left\lceil \frac{n}{36}\right\rceil+\left\lceil \frac{n}{18}\right\rceil\\
&\leqslant \frac{5n}{18}+\frac{\delta n}{3}+11,
\end{align*}
as desired.
\medskip

To finish the proof of case 3.1.1, we need to consider the remaining possibilities where $a_1+a_2\equiv 3,6 \!\!\mod 9$. These are very similar so we only give the proof when $a_1+a_2\equiv 6\!\!\mod 9$. Recall that by \eqref{3.1lowerbound},  $A_{(\frac{1}{2},1]}^{1(3)}+A_{(\frac{1}{2},1]}^{2(3)}$ contains at least $\left|A_{(\frac{1}{2},1]}^{1(3)}\right|+\left|A_{(\frac{1}{2},1]}^{2(3)}\right|-1\geqslant\left(\frac{1}{3}-\delta\right)\frac{n}{3}-1$ numbers in $(n,2n]$ which are $6\!\!\mod 9$. From the definition \eqref{A''definition}, $A'''$ therefore contains all except at most \begin{equation}
    \left\lceil\frac{n}{9}\right\rceil-\left(\frac{1}{3}-\delta\right)\frac{n}{3}-1\leqslant \frac{\delta n}{3}+2
\label{zaza0}
\end{equation} of the numbers which are $2\!\!\mod 3$ in $\left(\frac{n}{3},\frac{2n}{3}\right]$. We will use this fact to bound the number of integers in $A\cap\left[\frac{2n}{3}\right]$ which are not divisible by $3$. To do this, note that we can multiply any $a\in A\cap[\frac{n}{2}]$ by a power of $2$, say $2^{p_a}$, so that $2^{p_a}a$ lies in $\left(\frac{n}{4},\frac{n}{2}\right]$. Let
\begin{equation}
C_1\vcentcolon=\left\{2^{p_a}a:a\in A\cap\left[\frac{n}{2}\right]\text{ and } 3\nmid a\right\}
\label{zazaxtra}
\end{equation} be the resulting set and note that this is the same construction that we defined in \eqref{B_1'defi}. Note that $C_1$ contains no multiples of $3$ by definition. Comparing the constructions \eqref{dyadicembedding} of $B_1$ and \eqref{zazaxtra} of $C_1$, we see that $$C_1\cap\left(\frac{n}{3},\frac{n}{2}\right]\subset B_1\text{ and } 2\cdot\bigg( C_1\cap\left(\frac{n}{4},\frac{n}{3}\right] \bigg)\subset B_1.$$ In the paragraph preceding \eqref{dyadicembed}, we showed that $B_1$ and $A'''$ are disjoint. By \eqref{zaza0}, $A'''$ contains all except at most $\frac{\delta n}{3}+2$ of the numbers which are $2\!\!\mod 3$ in $\left(\frac{n}{3},\frac{2n}{3}\right]$, so the disjointness of $A'''$ and $B_1$ and the two inclusions above imply that $\left(C_1\cap\left(\frac{n}{3},\frac{n}{2}\right]\right)\cup2\cdot\bigg( C_1\cap\left(\frac{n}{4},\frac{n}{3}\right] \bigg)$ contains at most $\frac{\delta n}{3}+2$ numbers which are $2\!\!\mod 3$. Hence, with the exception of at most $\frac{\delta n}{3}+2$ numbers in $C_1$, every element of $C_1\cap\left(\frac{n}{4},\frac{n}{3}\right]$ must be $2\!\!\mod 3$ and every element of $C_1\cap\left(\frac{n}{3},\frac{n}{2}\right]$ must be $1\!\!\mod 3$. This yields the bound \begin{equation}
    \left|C_1\right|\leqslant \left\lceil \frac{n}{36}\right\rceil+\left\lceil \frac{n}{18}\right\rceil+\frac{\delta n}{3}+2.
\label{zaza1}
\end{equation} By \eqref{inductionboundmultiples}, we have
\begin{equation}
\left|A\cap(3\!\cdot\!\mathbf{N})\cap\left[\frac{n}{2}\right]\right|\leqslant \frac{n}{18}+C.
\label{zaza2}
\end{equation} As we are in case 3.1.1, each of $A_{(\frac{1}{2},1]}^{1(3)},A_{(\frac{1}{2},1]}^{2(3)}$ lie in progressions with common difference $9$ so that $\left|A_{(\frac{1}{2},1]}\right|\leqslant \frac{3\left|A_{(\frac{1}{2},1]}^{1(3)}\right|+3\left|A_{(\frac{1}{2},1]}^{2(3)}\right|}{2}\leqslant 3\left\lceil \frac{n}{18}\right\rceil$ where the first inequality follows from the assumption of Subcase 3.1. Noting that $\left|A_{[\frac{1}{2}]}\right|=\left|A\cap(3\!\cdot\!\mathbf{N})\cap\left[\frac{n}{2}\right]\right|+|C_1|$ by \eqref{B_1'decompo} and then using \eqref{zaza1}, \eqref{zaza2} and this bound on $\left|A_{(\frac{1}{2},1]}\right|$ gives \begin{align*}
|A|&=\left|A_{(\frac{1}{2},1]}\right|+\left|A_{[\frac{n}{2}]}\right|\\
&\leqslant 3\left\lceil \frac{n}{18}\right\rceil+\left|A\cap(3\!\cdot\!\mathbf{N})\cap\left[\frac{n}{2}\right]\right|+\left|C_1\right| \\
&\leqslant 3\left\lceil \frac{n}{18}\right\rceil+\frac{n}{18}+C+\left\lceil \frac{n}{36}\right\rceil+\left\lceil \frac{n}{18}\right\rceil+\frac{\delta n}{3}+2\\
&\leqslant \frac{11n}{36}+\frac{\delta n}{3}+C+7.
\end{align*} This is stronger than the desired bound \eqref{inductiontheo} for $n$ sufficiently large and this finishes the proof of case 3.1.1. \qed
\medskip
\begin{flushleft}
\textbf{{3.1.2: Let $\gcd_*\left(A_{(\frac{1}{2},1]}^{1(3)}+A_{(\frac{1}{2},1]}^{2(3)}\right) = 6$.}}
\end{flushleft}
\medskip

Because $\gcd_*\left(A_{(\frac{1}{2},1]}^{1(3)}+A_{(\frac{1}{2},1]}^{2(3)}\right) = 6$ and $A_{(\frac{1}{2},1]}^{1(3)}+A_{(\frac{1}{2},1]}^{2(3)}$ consists of multiples of $3$ only, the progression $Q\subset A_{(\frac{1}{2},1]}^{1(3)}+A_{(\frac{1}{2},1]}^{2(3)}$ that we obtained from conclusion $(2)$ of Theorem \ref{freimain} must be fully contained in either $6\cdot\mathbf{N}$ or $3+6\cdot\mathbf{N}$. These two possibilities require different arguments. Let us begin with the first case, so assume that $Q\subset A_{(\frac{1}{2},1]}^{1(3)}+A_{(\frac{1}{2},1]}^{2(3)}$ is a progression of size  
\begin{equation}
    |Q|\geqslant \left|A_{(\frac{1}{2},1]}^{1(3)}\right|+\left|A_{(\frac{1}{2},1]}^{2(3)}\right|-1\geqslant\frac{2\left|A_{(\frac{1}{2},1]}\right|}{3}-1
\label{Qlowerb}
\end{equation}(the final inequality follows as we are in Subcase 3.1) with common difference 6 consisting of multiples of 6. So $\frac{1}{3}\cdot Q\subset A'''\cap (2\cdot\mathbf{N})$ by the definition \eqref{A''definition} of $A'''$. Recall the construction \eqref{dyadicembedding} of $B_1$ and let $E_{B_1},O_{B_1}$ be the even/odd numbers in $B_1$. By the discussion preceding \eqref{dyadicembed}, we have that $E_{B_1}\subset B_1$ and $\frac{1}{3}\cdot Q\subset A'''\cap(2\cdot\mathbf{N})$ are disjoint sets of even numbers contained in $\left(\frac{n}{3},\frac{2n}{3}\right]$. This yields the bound \begin{equation}
    |E_{B_1}|\leqslant \frac{n}{6}+1 -|Q| \leqslant  \frac{n}{6} - \frac{2\left|A_{(\frac{1}{2},1]}\right|}{3}+2.
    \label{eq1}
\end{equation}
We constructed $B_1$ from $A\cap\left[\frac{2n}{3}\right]$ by multiplying every $a\in A\cap\left[\frac{2n}{3}\right]$ by an appropriate power $2^{j_a}$ of 2 as described in \eqref{dyadicembedding}. In particular, all odd numbers in $B_1\cap\left(\frac{n}{2},\frac{2n}{3}\right]$ must lie in $A_{(\frac{1}{2},\frac{2}{3}]}= A\cap\left(\frac{n}{2},\frac{2n}{3}\right]$ as they cannot have been obtained from multiplying a number in $A$ by a power of 2 greater than 1. So \begin{equation}
    \left|O_{B_1}\right|\leqslant \left|A_{(\frac{1}{2},\frac{2}{3}]}\right|+\left|O_{B_1}\cap\left(\frac{n}{3},\frac{n}{2}\right]\right|.
    \label{shii}
\end{equation}
To make use of this inequality, we need an upper bound on $\left|O_{B_1}\cap\left(\frac{n}{3},\frac{n}{2}\right]\right|$. The progression $Q\subset(n,2n]\cap (6\!\cdot\! \mathbf{N})$ has common difference $6$ and length at least $\frac{2\left|A_{(\frac{1}{2},1]}\right|}{3}-1>\left(\frac{1}{3}-\delta\right)\frac{n}{3}-1$ by \eqref{Qlowerb} and as $\left|A_{(\frac{1}{2},1]}\right|>\left(\frac{1}{3}-\delta\right)\frac{n}{2}$ by \eqref{inductionbound}. So $Q$ contains at least $\left(\frac{1}{3}-\delta\right)\frac{n}{12}-2$ numbers which are $12\!\!\mod 24$. $Q\subset A_{(\frac{1}{2},1]}+A_{(\frac{1}{2},1]}$ is contained in $(n,2n]$ and the interval $\left(n,\frac{4n}{3}\right]$ contains at most $\frac{n}{72}+1$ numbers congruent to $12 \!\!\mod 24$. Hence for $\delta>0$ sufficiently small, $Q$ contains at least $\left(\frac{1}{3}-\delta\right)\frac{n}{12}-2-\frac{n}{72}-1> \frac{n}{100}$ of numbers in $\left(\frac{4n}{3},2n\right]$ which are $12 \!\!\mod 24$. Note that all these numbers lie in $A_{(\frac{1}{2},1]}+A_{(\frac{1}{2},1]}$ and are of the form $4x$ for an odd integer $x\in \left(\frac{n}{3},\frac{n}{2}\right]$ so that as $A$ has property P, no such $x$ can lie in $B_1$. This shows that $\big|O_{B_1}\cap\left(\frac{n}{3},\frac{n}{2}\right]\big|\leqslant  \frac{n}{12}+1- \frac{n}{100}$ as there are at most $ \frac{n}{12}+1$ odd numbers in $\left(\frac{n}{3},\frac{n}{2}\right]$. Hence, \eqref{shii} gives
\begin{equation}
    \left|O_{B_1}\right|-\left|A_{(\frac{1}{2},\frac{2}{3}]}\right|\leqslant\frac{n}{12}+1-\frac{n}{100}.\label{eq2}
\end{equation}
By starting with \eqref{Adecompo} and combining inequalities \eqref{eq1} and \eqref{eq2} to bound $|E_{B_1}|$ and $|O_{B_1}|$, we deduce the desired bound on $|A|$ as follows
\begin{align}
    |A| &= \left|A_{(\frac{2}{3},1]}\right|+|B_1| = \left|A_{(\frac{1}{2},1]}\right|-\left|A_{(\frac{1}{2},\frac{2}{3}]}\right|+|B_1| \nonumber\\
    &= \left|A_{(\frac{1}{2},1]}\right|+|E_{B_1}|+|O_{B_1}|-\left|A_{(\frac{1}{2},\frac{2}{3}]}\right|\nonumber\\
    &\leqslant \left|A_{(\frac{1}{2},1]}\right|+ \frac{n}{6}-\frac{2\left|A_{(\frac{1}{2},1]}\right|}{3}+2 +\frac{n}{12}+1-\frac{n}{100}\nonumber\\
    &= \frac{\left|A_{(\frac{1}{2},1]}\right|}{3}+\frac{6n}{25}+3\nonumber\\
    &\leqslant \frac{n}{12}+13 +\frac{6n}{25}+3\nonumber= \frac{n}{3}-\frac{n}{100}+16,\nonumber
\end{align}
noting for the penultimate inequality that $\left|A_{(\frac{1}{2},1]}\right| = \left|A_{(\frac{1}{2},\frac{2}{3}]}\right|+\left|A_{(\frac{2}{3},1]}\right|\leqslant \frac{3\left|A_{(\frac{2}{3},1]}\right|}{2}+1\leqslant\frac{n}{4}+37$ using \eqref{A_2lowe} and that we are in Case 3 so $\left|A_{(\frac{2}{3},1]}\right|< \frac{n}{6}+24$.
\medskip

To complete the argument in case 3.1.2, we still need to consider the case where $Q$ is an arithmetic progression consisting of numbers which are $3\!\!\mod 6$ in $\left(n,2n\right]$. Again we split $B_1$ into $E_{B_1},O_{B_1}$, its even/odd elements. Now $\frac{1}{3}\cdot Q$ is a subset of  $A'''$ by definition \eqref{A''definition}, so $O_{B_1}$ and  $\frac{1}{3}\cdot Q$ are disjoint sets of odd numbers in $\left(\frac{n}{3},\frac{2n}{3}\right]$ by the discussion preceding \eqref{dyadicembed}. Hence,
\begin{align}|O_{B_1}|&\leqslant \frac{n}{6}+1-|Q|\leqslant \frac{n}{6}-\left|A_{(\frac{1}{2},1]}^{1(3)}\right|-\left|A_{(\frac{1}{2},1]}^{2(3)}\right|+2
\label{oddA'}
\end{align}
by \eqref{Qlowerb}.
We also have that $\frac{2}{3}\cdot A_{(\frac{1}{2},1]}^{0(3)}\subset A'''$ by \eqref{A''definition} so $E_{B_1}$ and $\frac{2}{3}\cdot A_{(\frac{1}{2},1]}^{0(3)}$ are disjoint sets of even numbers in $\left(\frac{n}{3},\frac{2n}{3}\right]$. Hence, \begin{equation}|E_{B_1}|\leqslant \frac{n}{6}+1-\left|A_{(\frac{1}{2},1]}^{0(3)}\right|.
\label{evenA'}
\end{equation}
Using \eqref{Adecompo} and plugging in the inequalities \eqref{oddA'} and \eqref{evenA'} gives
\begin{align*}
|A|&= \left|A_{(\frac{2}{3},1]}\right|+|B_1|= \left|A_{(\frac{2}{3},1]}\right|+\left|E_{B_1}\right|+\left|O_{B_1}\right| \\ &\leqslant \left|A_{(\frac{2}{3},1]}\right|+\frac{n}{3}+3-\left|A_{(\frac{1}{2},1]}^{0(3)}\right|-\left|A_{(\frac{1}{2},1]}^{1(3)}\right|-\left|A_{(\frac{1}{2},1]}^{2(3)}\right|\\ &\leqslant  \frac{n}{3}+3-\left|A_{(\frac{1}{2},\frac{2}{3}]}\right|,
\end{align*}
as $\left|A_{(\frac{1}{2},1]}^{0(3)}\right|+\left|A_{(\frac{1}{2},1]}^{1(3)}\right|+\left|A_{(\frac{1}{2},1]}^{2(3)}\right|=\left|A_{(\frac{1}{2},1]}\right|=\left|A_{(\frac{2}{3},1]}\right|+\left|A_{(\frac{1}{2},\frac{2}{3}]}\right|$ by definition. Plugging in the lower bound \eqref{A_2genelowe} then gives the desired bound \eqref{inductiontheo} on $|A|$ and this concludes the proof of case 3.1.2.
\qed
\medskip

\begin{flushleft}
\textbf{{3.1.3: Let $\gcd_*\left(A_{(\frac{1}{2},1]}^{1(3)}+A_{(\frac{1}{2},1]}^{2(3)}\right)=3$.}}
\end{flushleft}
\medskip

Under the assumption that $\gcd_*\left(A_{(\frac{1}{2},1]}^{1(3)}+A_{(\frac{1}{2},1]}^{2(3)}\right) = 3$, conclusion $(2)$ from Theorem \ref{freimain} gives a progression $Q\subset A_{(\frac{1}{2},1]}^{1(3)}+A_{(\frac{1}{2},1]}^{2(3)}$ with common difference 3 consisting of multiples of 3 and having size $|Q|= \left|A_{(\frac{1}{2},1]}^{1(3)}\right|+\left|A_{(\frac{1}{2},1]}^{2(3)}\right|-1\geqslant\frac{2\left|A_{(\frac{1}{2},1]}\right|}{3}-1$, where the final inequality follows from the assumption of Subcase 3.1. Since $\left|A_{(\frac{1}{2},1]}\right|> \left(\frac{1}{3}-\delta\right)\frac{n}{2}$ by \eqref{inductionbound}, we find that $Q$ has length at least $\left(\frac{1}{3}-\delta\right)\frac{n}{3}-1\geqslant \frac{n}{9}-\delta n$ for $n$ sufficiently large. From this and as $A$ has property P, we conclude that $A$ contains no numbers in $\left[\frac{n}{9}-\delta n\right]$ because any $x\in \left[\frac{n}{9}-\delta n\right]$ has a multiple in the arithmetic progression $Q\subset A_{(\frac{1}{2},1]}+A_{(\frac{1}{2},1]}$ precisely because $|Q|\geqslant \frac{n}{9}-\delta n\geqslant x$. Furthermore since $Q$ consists of multiples of 3 only, $A$ also contains no multiples of 3 in $\left[\frac{n}{3}-3\delta n\right]$ because all multiples of $3$ in $\left[\frac{n}{3}-3\delta n\right]$ also divide an element of $Q$. Indeed if $x$ is an integer with $3x\in\left[\frac{n}{3}-3\delta n\right]$, then as $\frac{1}{3}\cdot Q$ is an interval of length at least $\frac{n}{9}-\delta n \geqslant x$ it must contain a multiple of $x$ so that $3x$ has a multiple in $Q$ as we claimed. For notational convenience, we shall in fact assume for the remainder of this case 3.1.3 that $A$ contains no elements in $A\cap\left[\frac{n}{9}\right]$ and no multiples of $3$ in $A\cap \left[\frac{n}{3}\right]$. Note that by the argument above, we need to remove at most $4\delta n$ elements from $A$ so that the resulting set satisfies this extra assumption. To finish the proof of case 3.1.3, we prove a strong enough upper bound on the size of sets $A$ satisfying this extra assumption (namely that $A$ contains no elements in $A\cap\left[\frac{n}{9}\right]$ and no multiples of $3$ in $A\cap \left[\frac{n}{3}\right]$) so that the desired bound \eqref{inductiontheo} holds even after adding $4\delta n$ to it.
\medskip

We previously constructed the set $B_1$ from $A\cap\left[\frac{2n}{3}\right]$  using powers of $2$ and later used similar constructions for $C_1$ and $C_1'$. Now we will make use of a different construction to obtain another auxiliary set $B_2$ from $A\cap\left[\frac{n}{2}\right]$. The set $B_2$ is more tedious to define, and we describe its construction using the following steps. From the discussion in the first paragraph, we may assume that $A\cap\left[\frac{n}{9}\right]$ is empty. Next, for every $a\in A\cap \left(\frac{n}{9},\frac{n}{6}\right]$, we add $3a$ to the set $B_2$. For every $a\in A\cap\left(\frac{n}{6},\frac{n}{4}\right]$, we add $2a$ to the set $B_2$. Also, for each even number $a \in A\cap\left(\frac{n}{4},\frac{n}{3}\right]$, observe that $\frac{3a}{2}$ is an integer and we add $\frac{3a}{2}$ to $B_2$. Finally we add all odd numbers in $A\cap\left(\frac{n}{4},\frac{n}{3}\right]$ and all of $A\cap\left(\frac{n}{3},\frac{n}{2}\right]$ to $B_2$. At the end of this process, we obtain a set of integers $B_2\subset(\frac{n}{4},\frac{n}{2}]$. In fact, $B_2$ can be defined explicitly as follows
\begin{align}
B_2\vcentcolon= \;\;\;\;&3\!\cdot\!\bigg(A\cap \left(\frac{n}{9},\frac{n}{6}\right]\bigg)
\label{B_2definition}\\
\bigcup&\;2\!\cdot\!\bigg(A\cap \left(\frac{n}{6},\frac{n}{4}\right]\bigg)\nonumber\\
\bigcup&\;\frac{3}{2}\!\cdot\!\bigg(A\cap(2\!\cdot\!\mathbf{N})\cap \left(\frac{n}{4},\frac{n}{3}\right]\bigg)\nonumber\\
\bigcup&\;\bigg(A\cap(1+2\!\cdot\!\mathbf{N})\cap \left(\frac{n}{4},\frac{n}{3}\right]\bigg)\nonumber\\
\bigcup&\;\bigg(A\cap\left(\frac{n}{3},\frac{n}{2}\right]\bigg).\nonumber
\end{align}
Note that we can view the construction of $B_2$ as the image of a function $A\cap\left[ \frac{n}{2}\right]\to B_2$ which maps each $a\in A\cap\left[ \frac{n}{2}\right]$ to $3a,2a,\frac{3a}{2}$ or $a$ according to whether $a$ lies in $\left(\frac{n}{9},\frac{n}{6}\right], \left(\frac{n}{6},\frac{n}{4}\right], (2\!\cdot\!\mathbf{N})\cap \left(\frac{n}{4},\frac{n}{3}\right]$ or $\bigg((1+2\!\cdot\!\mathbf{N})\cap \left(\frac{n}{4},\frac{n}{3}\right]\bigg)
\bigcup \left(\frac{n}{3},\frac{n}{2}\right]$ respectively.
We show that this mapping is injective so that
\begin{equation}
|B_2|=\bigg|A\cap\left[\frac{n}{2}\right]\bigg|  \label{B_2equals}
\end{equation} which implies that \begin{equation}
|A|= \left|A_{(\frac{1}{2},1]}\right|+|B_2|
\label{B_2decompo}
\end{equation} recalling that $A_{(\frac{1}{2},1]}= A\cap\left(\frac{n}{2},n\right]$. Indeed, if this mapping from $A\cap\left[\frac{n}{2}\right]$ to $B_2$ is not injective, then there are incidences of the form $2a = b$, $3a = b$, $3a = 2b$, $\frac{3a}{2} = b$, $3a= \frac{3b}{2}$ or $\frac{3a}{2}=2b$ for some $a,b\in A\cap \left[ \frac{n}{2}\right]$. However, the first five of these would imply that $b>a$ and $a|b+b$ so they are impossible as $A$ has property P. The last one $\frac{3a}{2}=2b$ is impossible too as it would imply that $b$ is a multiple of 3 in $A$ with $b\leqslant \frac{n}{3}$, but we are assuming after our discussion in the first paragraph of case 3.1.3 that $A$ contains no multiples of $3$ in $\left[\frac{n}{3}\right]$.
The point of this somewhat elaborate construction is that we can obtain good upper bounds on $|B_2|$ which then immediately give a corresponding bound for $\left|A\cap\left[\frac{n}{2}\right]\right|$ by \eqref{B_2equals}. To do this, we split $B_2$ into its `left' part $B_2^{\operatorname{L}}$ and its `right' part $B_2^{\operatorname{R}}$ defined by
\begin{align}
    B_2^{\operatorname{L}}&\vcentcolon= B_2\cap\left(\frac{n}{4},\frac{n}{3}\right]=A\cap(1+2\!\cdot\!\mathbf{N})\cap \left(\frac{n}{4},\frac{n}{3}\right] \label{B_2LRdefi}\\
    B_2^{\operatorname{R}}&\vcentcolon= B_2\cap\left(\frac{n}{3},\frac{n}{2}\right]. \nonumber
\end{align} Note that by the definition \eqref{B_2definition} of $B_2$, we have that $B_2=B_2^{\operatorname{L}}\cup B_2^{\operatorname{R}}$ and that $B_2^{\operatorname{L}}=A\cap(1+2\!\cdot\!\mathbf{N})\cap \left(\frac{n}{4},\frac{n}{3}\right]$ so $B_2^{\operatorname{L}}$ consists of odd numbers only. We obtain bounds on $B_2^{\operatorname{L}}$ and $B_2^{\operatorname{R}}$ in the following two lemmas. 
\begin{lemma}
Let $A_{(\frac{2}{3},1]}^{i(4)}$ be the set of numbers in $A_{(\frac{2}{3},1]}$ that are $i\!\!\mod 4$. Then for each of $i=1,3$ we have the inequality
\begin{align}
    \left|B_2^{\operatorname{L}}\right|+\left|A_{(\frac{2}{3},1]}^{i(4)}\right|\leqslant \frac{n}{12}+3.
\label{Coddu}
\end{align}
\label{oddsmall}
\end{lemma}
\begin{proof}
For the proof of this lemma, we use the notation $A_{(\frac{2}{3},1]}^{i(4),j(3)}\vcentcolon=A_{(\frac{2}{3},1]}^{i(4)}\cap(j+3\!\cdot\!\mathbf{N})$ for $j=0,1,2$. First assume that both of $A_{(\frac{2}{3},1]}^{i(4),1(3)},A_{(\frac{2}{3},1]}^{i(4),2(3)}$ are non-empty. Then as $i=1$ or $3$, the sumset $A_{(\frac{2}{3},1]}^{i(4)}+A_{(\frac{2}{3},1]}^{i(4)}$ contains at least
\begin{align}
    &\max\left(\left|A_{(\frac{2}{3},1]}^{i(4),1(3)}+A_{(\frac{2}{3},1]}^{i(4),2(3)}\right|,\left|A_{(\frac{2}{3},1]}^{i(4),0(3)}+A_{(\frac{2}{3},1]}^{i(4),0(3)}\right|\right)\nonumber\\
    &\geqslant\max\left(\left|A_{(\frac{2}{3},1]}^{i(4),1(3)}\right|+\left|A_{(\frac{2}{3},1]}^{i(4),2(3)}\right|,2\left|A_{(\frac{2}{3},1]}^{i(4),0(3)}\right|\right)-1\nonumber\\
    &\geqslant \frac{2\left|A_{(\frac{2}{3},1]}^{i(4)}\right|}{3}-1
\label{xaxa1}
\end{align}
numbers which are $6\!\!\mod 12$. If one of $A_{(\frac{2}{3},1]}^{i(4),1(3)},A_{(\frac{2}{3},1]}^{i(4),2(3)}$ is empty, then $\left|A_{(\frac{2}{3},1]}^{i(4),0(3)}\right|= \left|A_{(\frac{2}{3},1]}^{i(4)}\right|-\max\left(\left|A_{(\frac{2}{3},1]}^{i(4),1(3)}\right|,\left|A_{(\frac{2}{3},1]}^{i(4),2(3)}\right|\right)$. We can trivially bound $\left|A_{(\frac{2}{3},1]}^{i(4),j(3)}\right|< \frac{n}{36}+1$ as the interval $\left(\frac{2n}{3},n\right]$ contains at most this many numbers in any given residue class modulo 12. So $\left|A_{(\frac{2}{3},1]}^{i(4),0(3)}\right|\geqslant \max\left(\left|A_{(\frac{2}{3},1]}^{i(4)}\right|- \frac{n}{36}-1,0\right)$ and we obtain the lower bound \begin{equation}
\left|A_{(\frac{2}{3},1]}^{i(4),0(3)}+A_{(\frac{2}{3},1]}^{i(4),0(3)}\right|\geqslant \max\left(2\left|A_{(\frac{2}{3},1]}^{i(4),0(3)}\right|-1,0\right)\geqslant \max\left(2\left|A_{(\frac{2}{3},1]}^{i(4)}\right|- \frac{n}{18}-2,0\right).
\label{yaboy}
\end{equation}
Combining \eqref{xaxa1} and \eqref{yaboy} shows that 
\begin{equation}
    \left|\left(A_{(\frac{2}{3},1]}\!+\!A_{(\frac{2}{3},1]}\right)\cap(6\!+\!12\!\cdot\!\mathbf{N})\right|\geqslant \min\left(\frac{2\left|A_{(\frac{2}{3},1]}^{i(4)}\right|}{3}\!-\!1,\max\left(2\left|A_{(\frac{2}{3},1]}^{i(4)}\right|-\! \frac{n}{18}\!-\!2,0\right)\right).
    \label{yayoy}
\end{equation}
Now note that as $B_2^{\operatorname{L}}$ only contains odd numbers by \eqref{B_2LRdefi}, the dilated set $6\cdot B_2^{\operatorname{L}}$ consists only of numbers which are $6\!\!\mod 12$ and crucially, even though not every element of $B_2$ is necessarily a multiple of some $a\in A$, it can be seen from \eqref{B_2definition} that $6b_2$ is an integer multiple of some $a\in A\cap\left[\frac{n}{2}\right]$ for every $b_2\in B_2$. Hence, using Lemma \ref{basicmultl} with $I=\left(\frac{4n}{3},2n\right]$, $q=12$ and \eqref{yayoy} yields the inequality $$\left|B_2^{\operatorname{L}}\right|\leqslant \frac{n}{18}+1-\min\left(\frac{2\left|A_{(\frac{2}{3},1]}^{i(4)}\right|}{3}\!-\!1,\max\left(2\left|A_{(\frac{2}{3},1]}^{i(4)}\right|-\! \frac{n}{18}\!-\!2,0\right)\right).$$
So we get the desired inequality
$$\left|B_2^{\operatorname{L}}\right|+\left|A_{(\frac{2}{3},1]}^{i(4)}\right|\leqslant \frac{n}{18}\!+\!2\!+\!\max\left(\frac{\left|A_{(\frac{2}{3},1]}^{i(4)}\right|}{3},\min\left(\frac{n}{18}\!+\!1\!-\!\left|A_{(\frac{2}{3},1]}^{i(4)}\right|,\left|A_{(\frac{2}{3},1]}^{i(4)}\right|\right)\right)\leqslant\frac{n}{12}\!+\!3$$
where in the final inequality we used the trivial bound $\left|A_{(\frac{2}{3},1]}^{i(4)}\right|<\frac{n}{12}+1$ as there are at most this many numbers in $\left(\frac{2n}{3},n\right]$ which are $i\!\!\mod 4$ and also that the function $f(x) = \min\left(\frac{n}{18}-x,x\right)$ has maximum value $\frac{n}{36}$ for $x\in \mathbf{R}$.
\end{proof}

\begin{lemma}
Either the following inequality holds
\begin{align}
    2\left|B_2^{\operatorname{R}}\right|+\left|A_{(\frac{2}{3},1]}\right|\leqslant \frac{n}{3}+4,
    \label{expands4}
\end{align}
or we have that
\begin{align}
    |B_2|+\left|A_{(\frac{2}{3},1]}\right|\leqslant \frac{n}{4}+4.
    \label{noexpands4u}
\end{align}
\label{divideby4}
\end{lemma}
\begin{proof}
Recall that $A_{(\frac{2}{3},1]}^{i(4)}$ is the set of numbers in $A_{(\frac{2}{3},1]}$ that are $i\!\!\mod 4$. First assume that either $A_{(\frac{2}{3},1]}^{1(4)}$ or $A_{(\frac{2}{3},1]}^{3(4)}$ is empty. Then by \eqref{Coddu} we obtain $\left|B_2^{\operatorname{L}}\right|+\big|A_{(\frac{2}{3},1]}\cap (1+2\!\cdot\!\mathbf{N})\big|\leqslant \left|B_2^{\operatorname{L}}\right|+\max\left(\left|A_{(\frac{2}{3},1]}^{1(4)}\right|,\left|A_{(\frac{2}{3},1]}^{3(4)}\right|\right)\leqslant \frac{n}{12} +3$. We also have that $\left|B_2^{\operatorname{R}}\right|+\big|A_{(\frac{2}{3},1]}\cap (2\!\cdot\!\mathbf{N})\big|\leqslant  \frac{n}{6}+1$ since $2\!\cdot\!B_2^{\operatorname{R}}$ and $A_{(\frac{2}{3},1]}\cap (2\!\cdot\!\mathbf{N})$ are disjoint sets of even numbers in $\left(\frac{2n}{3},n\right]$ as $A$ has property P, see Lemma \ref{basicInteger}. Combining these inequalities and using that $|B_2|=\left|B_2^{\operatorname{L}}\right|+\left|B_2^{\operatorname{R}}\right|$ by definition \eqref{B_2LRdefi}, we obtain the desired inequality \eqref{noexpands4u} as follows: $|B_2|+\left|A_{(\frac{2}{3},1]}\right|\leqslant \left|B_2^{\operatorname{L}}\right|+\left|A_{(\frac{2}{3},1]}\cap (1+2\!\cdot\!\mathbf{N})\right|+\left|B_2^{\operatorname{R}}\right|+\left|A_{(\frac{2}{3},1]}\cap (2\!\cdot\!\mathbf{N})\right|\leqslant \frac{n}{4}+4$.

Now assume that neither $A_{(\frac{2}{3},1]}^{1(4)}$ nor $A_{(\frac{2}{3},1]}^{3(4)}$ is empty. Then $\left|A_{(\frac{2}{3},1]}^{1(4)}+A_{(\frac{2}{3},1]}^{3(4)}\right|\geqslant \left|A_{(\frac{2}{3},1]}^{1(4)}\right|+\left|A_{(\frac{2}{3},1]}^{3(4)}\right|-1$, and we also have $\left|A_{(\frac{2}{3},1]}^{j(4)}+A_{(\frac{2}{3},1]}^{j(4)}\right|\geqslant 2\left|A_{(\frac{2}{3},1]}^{j(4)}\right|-1$ for $j=0,2$. Hence, $A_{(\frac{2}{3},1]}+A_{(\frac{2}{3},1]}$ contains at least $$\max\left(\left|A_{(\frac{2}{3},1]}^{1(4)}\right|+\left|A_{(\frac{2}{3},1]}^{3(4)}\right|,2\left|A_{(\frac{2}{3},1]}^{0(4)}\right|,2\left|A_{(\frac{2}{3},1]}^{2(4)}\right|\right)-1\geqslant \frac{\left|A_{(\frac{2}{3},1]}\right|}{2}-1$$ multiples of 4 in $\left(\frac{4n}{3},2n\right]$. Note that the dilated set $4\cdot B_2^{\operatorname{R}}$ also consists only of multiples of 4 in $\left(\frac{4n}{3},2n\right]$ so Lemma \ref{basicmultl} yields the inequality $\left|B_2^{\operatorname{R}}\right|+\frac{\left|A_{(\frac{2}{3},1]}\right|}{2}-1\leqslant \frac{n}{6}+1$ which is inequality \eqref{expands4}
\end{proof}
Having obtained an upper bound on the auxiliary set $B_2$ defined by \eqref{B_2definition}, let us now continue with the main argument of case 3.1.3. Lemma \ref{divideby4} gave us two possible conclusions. First suppose that \eqref{noexpands4u} holds, then using \eqref{B_2equals} and \eqref{noexpands4u} yields
\begin{align*}
    |A|&=\left|A_{[\frac{1}{2}]}\right|+\left|A_{(\frac{1}{2},\frac{2}{3}]}\right|+\left|A_{(\frac{2}{3},1]}\right|\\
    &= |B_2|+\left|A_{(\frac{1}{2},\frac{2}{3}]}\right|+\left|A_{(\frac{2}{3},1]}\right|\\
    &\leqslant \frac{n}{4}+4+\left|A_{(\frac{1}{2},\frac{2}{3}]}\right|.
\end{align*}
So if \eqref{noexpands4u} holds then we may suppose that \begin{equation}
\left|A_{(\frac{1}{2},\frac{2}{3}]}\right|\geqslant\frac{n}{12}-4
\label{A_2case1}
\end{equation} or else the desired bound \eqref{inductiontheo} holds by the above so we are done. If \eqref{noexpands4u} does not hold, then by Lemma \ref{divideby4} we may assume that \eqref{expands4} holds and we will again deduce a lower bound on $\left|A_{(\frac{1}{2},\frac{2}{3}]}\right|$. From \eqref{expands4} we get that \begin{align}
\left|A\setminus A_{(\frac{1}{2},\frac{2}{3}]}\right|&=\left|A_{(\frac{2}{3},1]}\right|+\left|A_{[\frac{n}{2}]}\right|\nonumber\\
&=\left|A_{(\frac{2}{3},1]}\right|+|B_2|\nonumber\\
&=\left|A_{(\frac{2}{3},1]}\right|+\left|B_2^{\operatorname{R}}\right|+\left|B_2^{\operatorname{L}}\right|\nonumber\\
&\leqslant  \frac{\left|A_{(\frac{2}{3},1]}\right|}{2}+\frac{n}{6}+2+\left|B_2^{\operatorname{L}}\right|\nonumber\\
&\leqslant \frac{7n}{36}+3+\frac{\left|A_{(\frac{2}{3},1]}\right|}{2} 
\label{AminusB}
\end{align} using the bound $\left|B_2^{\operatorname{L}}\right|\leqslant \frac{n}{36}+1$ which holds as by definition \eqref{B_2LRdefi}, $B_2^{\operatorname{L}}=A\cap(1+2\!\cdot\!\mathbf{N})\cap \left(\frac{n}{4},\frac{n}{3}\right]$ is a subset of $\left(\frac{n}{4},\frac{n}{3}\right]$ consisting of odd numbers not divisible by $3$ only, as $A$ contains no multiples of $3$ in $\left[\frac{n}{3}\right]$ (which we showed in the first paragraph of case 3.1.3). So we can without loss of generality assume that 
\begin{equation}
    \left|A_{(\frac{1}{2},\frac{2}{3}]}\right|\geqslant  \frac{5n}{36}-\frac{\left|A_{(\frac{2}{3},1]}\right|}{2} -3,
    \label{Bbound}
\end{equation} or else we could conclude the desired bound \eqref{inductiontheo} from the inequality \eqref{AminusB}. 
Combining \eqref{A_2case1} and \eqref{Bbound} shows that whichever of the two possibilities in Lemma \ref{divideby4} holds, we always get that
\begin{equation}
    \left|A_{(\frac{1}{2},\frac{2}{3}]}\right|\geqslant\min\left(\frac{n}{12}-4,\frac{5n}{36}-\frac{\left|A_{(\frac{2}{3},1]}\right|}{2}-3\right).
\label{A_2T}
\end{equation}
Recall that $\frac{\left|A_{(\frac{2}{3},1]}\right|}{2}+1\geqslant\left|A_{(\frac{1}{2},\frac{2}{3}]}\right|$ by \eqref{A_2lowe} and using this in \eqref{A_2T} shows that
$\left|A_{(\frac{2}{3},1]}\right|\geqslant \min\left(\frac{n}{6}-10,\frac{5n}{36}-4\right)= \frac{5n}{36}-4$.
Combining this with \eqref{A_2T} gives \begin{align}\left|A_{(\frac{1}{2},1]}\right|&=\left|A_{(\frac{2}{3},1]}\right|+\left|A_{(\frac{1}{2},\frac{2}{3}]}\right|\nonumber\\
&\geqslant \min\left(\frac{n}{12}-4+\left|A_{(\frac{2}{3},1]}\right|,\frac{5n}{36}+\frac{\left|A_{(\frac{2}{3},1]}\right|}{2}-3\right)\label{minineq}\\
&\geqslant \min\left(\frac{2n}{9}-8,\frac{5n}{24}-5\right)= \frac{5n}{24}-5\nonumber.
\end{align}
In particular, the arithmetic progression $Q\subset A_{(\frac{1}{2},1]}^{1(3)}+A_{(\frac{1}{2},1]}^{2(3)}$ that we obtained from conclusion (2) in Theorem \ref{freimain} has size at least 
\begin{equation}
    |Q|\geqslant \left|A_{(\frac{1}{2},1]}^{1(3)}\right|+\left|A_{(\frac{1}{2},1]}^{2(3)}\right|-1\geqslant \frac{2\left|A_{(\frac{1}{2},1]}\right|}{3}-1\geqslant \frac{5n}{36}-5
\label{3.1.3Q}
\end{equation} by the assumption of Subcase 3.1 and the above lower bound on $\left|A_{(\frac{1}{2},1]}\right|$.

\medskip

We can now repeat the argument from the first paragraph of case 3.1.3 with this improved bound $|Q|\geqslant\frac{5n}{36}-5$ to deduce that since $Q\subset A_{(\frac{1}{2},1]}+A_{(\frac{1}{2},1]}$ is an arithmetic progression with common difference $3$ consisting of multiples of $3$, it contains an integer multiple of any multiple of $3$ in $\left[\frac{5n}{12}-15\right]$. Hence, as $A$ has property P, $A$ contains no multiples of $3$ in $\left[\frac{5n}{12}-15\right]$.
We are now ready to finish the argument. We conclude by considering the location of the interval $I\vcentcolon= \frac{1}{3}\!\cdot\!Q$, and note that $I \vcentcolon= [i_m,i_M] \subset \frac{1}{3}\cdot\left(A_{(\frac{1}{2},1]}^{1(3)}+A_{(\frac{1}{2},1]}^{2(3)}\right)\subset\left(\frac{n}{3},\frac{2n}{3}\right]$. Note further that $I\subset A'''$ by the definition \eqref{A''definition} of $A'''$ and we showed in the paragraph preceding \eqref{dyadicembed} that $A'''$ is disjoint from $B_1$. Hence, $I$ is also disjoint from $B_1$. Also recall that we defined $B_1=\left\{2^{j_a}a:a\in A\cap\left[\frac{2n}{3}\right]\right\}$ in \eqref{dyadicembedding} so that $B_1$ does not contain any multiples of $3$ in $\left[\frac{5n}{12}-15\right]$ as we just proved that neither does $A$. Finally let $\varepsilon>0$ be a small positive number to be determined later.
\medskip

\begin{enumerate}
    \item[(i)] If $i_m>\frac{7n}{18}+\varepsilon n$, then $B_1\cap\left(\frac{7n}{18}+\varepsilon n,\frac{2n}{3}\right]$ and $I=\frac{1}{3}\!\cdot\!Q$ are disjoint subsets of $\left(\frac{7n}{18}+\varepsilon n,\frac{2n}{3}\right]$ so by using the second inequality in \eqref{3.1.3Q}, we get $$\left|B_1\cap\left(\frac{7n}{18}+\varepsilon n,\frac{2n}{3}\right]\right|\leqslant \left\lceil \frac{5n}{18}-\varepsilon n\right\rceil-|I|\leqslant  \frac{5n}{18}-\varepsilon n-\frac{2\left|A_{(\frac{1}{2},1]}\right|}{3}+2.$$ Further note that $\left|B_1\cap\left(\frac{n}{3},\frac{7n}{18}+\varepsilon n\right]\right|\leqslant \left\lceil\frac{n}{27}+\frac{2\varepsilon n}{3}\right\rceil$ as $B_1$ cannot contain any multiples of $3$ in $\left[\frac{5n}{12}-15\right]\supset \left[\frac{7n}{18}+\varepsilon n\right]$, provided we choose $\varepsilon$ small enough. So in total we get
    \begin{align*}
|B_1|&\leqslant  \frac{5n}{18}-\varepsilon n-\frac{2\left|A_{(\frac{1}{2},1]}\right|}{3}+2+\left\lceil\frac{n}{27}+\frac{2\varepsilon n}{3}\right\rceil\\
&\leqslant\frac{17n}{54} -\frac{\varepsilon n}{3}-\frac{2\left|A_{(\frac{1}{2},1]}\right|}{3}+3\\&\leqslant \frac{17n}{54}-\frac{\varepsilon n}{3}-\min\left(\frac{n}{18}+\frac{2\left|A_{(\frac{2}{3},1]}\right|}{3},\frac{5n}{54}+\frac{\left|A_{(\frac{2}{3},1]}\right|}{3}\right)+6
    \end{align*} by using \eqref{minineq} for the final inequality. Using \eqref{Adecompo} and plugging in this bound on $|B_1|$, we get in total that \begin{align*}
        |A|&= \left|A_{(\frac{2}{3},1]}\right|+|B_1|\\
        &\leqslant \frac{17n}{54}-\frac{\varepsilon n}{3}+6+\max\left(\frac{\left|A_{(\frac{2}{3},1]}\right|}{3}-\frac{n}{18},\frac{2\left|A_{(\frac{2}{3},1]}\right|}{3}-\frac{5n}{54}\right)\\&\leqslant\frac{17n}{54}-\frac{\varepsilon n}{3}+6+\max\bigg(8,\frac{n}{54}+16\bigg)= \frac{n}{3}-\frac{\varepsilon n}{3}+22, \end{align*} where we used the assumption that $\left|A_{(\frac{2}{3},1]}\right|\leqslant\frac{n}{6}+24$ in Case 3. This gives the desired bound \eqref{inductiontheo} if we choose $\varepsilon > 15\delta$.
    \medskip
    
    \item[(ii)] If we are not in case (i), then $i_m\leqslant \frac{7n}{18}+\varepsilon n$. Using \eqref{3.1.3Q}, we see that \begin{equation}
     i_M\geqslant i_m+|I|-1=i_m+|Q|-1\geqslant i_m+\frac{2\left|A_{(\frac{1}{2},1]}\right|}{3}-2.
     \label{i_M}
    \end{equation} We have shown in the paragraph following \eqref{3.1.3Q} that $B_1$ is disjoint from $I$ and that $B_1\cap\left[\frac{5n}{12}-15\right]$ contains no multiples of $3$. Further note that $\left(B_1\setminus A_{(\frac{1}{2},\frac{2}{3}]}\right)\cap\left(\frac{n}{2},\frac{2n}{3}\right]$ consists only of even numbers not divisible by $3$ as by definition \eqref{dyadicembedding}, every number in $\left(B_1\setminus A_{(\frac{1}{2},\frac{2}{3}]}\right)\cap\left(\frac{n}{2},\frac{2n}{3}\right]$ is of the form $2^{j_a}a$ for some $a\in A\cap\left[\frac{n}{3}\right]$ so $j_a\geqslant 1$ and by the first paragraph of case 3.1.3, $A\cap\left[\frac{n}{3}\right]$ contains no multiples of $3$ so $3\nmid 2^{j_a}a$. Hence, if $i_M\leqslant\frac{n}{2}$, then 
    \begin{align*}
        \left|B_1\setminus A_{(\frac{1}{2},\frac{2}{3}]}\right|&\leqslant \left|\left(\frac{n}{3},\frac{n}{2}\right]\setminus I\right|+\left|\left\{x\in \left(\frac{n}{2},\frac{2n}{3}\right]: x \text{ is even and $3\nmid x$}\right\}\right|\\
        &\leqslant \frac{n}{6}+1-|I|+\frac{n}{18}+1\\
        &\leqslant \frac{2n}{9}+3-\frac{2\left|A_{(\frac{1}{2},1]}\right|}{3},
    \end{align*}
using that $|I|=|Q|$ and \eqref{3.1.3Q}.
Hence by starting with \eqref{Adecompo}, in this case we obtain the desired bound
\begin{align*}
|A|&=\left|B_1\setminus A_{(\frac{1}{2},\frac{2}{3}]}\right|+\left|A_{(\frac{1}{2},1]}\right|\\
&\leqslant \frac{2n}{9}+3+\frac{\left|A_{(\frac{1}{2},1]}\right|}{3}\\
&\leqslant \frac{2n}{9}+4+\frac{\left|A_{(\frac{2}{3},1]}\right|}{2}\leqslant \frac{11n}{36}+16,
\end{align*} using that $\left|A_{(\frac{1}{2},1]}\right|\leqslant \frac{3\left|A_{(\frac{2}{3},1]}\right|}{2}+1$ by \eqref{A_2lowe} and that $\left|A_{(\frac{2}{3},1]}\right|\leqslant \frac{n}{6}+24$ in Case 3. Finally, we may assume that $i_M>\frac{n}{2}$ and that $i_m\leqslant \frac{7n}{18}+\varepsilon n$. Recall from the discussion at the beginning of (ii) that $B_1$ is disjoint from $I=[i_m,i_M]$, that $B_1$ contains no multiples of $3$ in $\left[\frac{5n}{12}-15\right]\supset[i_m]$ and that $\left(B_1\setminus A_{(\frac{1}{2},\frac{2}{3}]}\right)\cap\left(\frac{n}{2},\frac{2n}{3}\right]$ consists only of even numbers not divisible by $3$, so we get
\begin{align*}
    \left|B_1\setminus A_{(\frac{1}{2},\frac{2}{3}]}\right|&\leqslant \left|B_1\cap\left(\frac{n}{3},i_m\right]\right|+\left|\left\{x\in\left(i_M,\frac{2n}{3}\right]:x \text{ is even and $3\nmid x$}\right\}\right|\\
    &\leqslant \left\lceil\frac{2}{3}\left(i_m-\frac{n}{3}\right)\right\rceil+ \left\lceil\frac{2n}{9}-\frac{i_M}{3}\right\rceil\\
    &\leqslant \frac{2i_m}{3}-\frac{i_M}{3}+2\\
    &\leqslant \frac{i_m}{3}-\frac{2\left|A_{(\frac{1}{2},1]}\right|}{9}+4,
\end{align*}
by plugging in the lower bound \eqref{i_M} on $i_M$. So by \eqref{Adecompo} and as $i_m\leqslant \frac{7n}{18}+ \varepsilon n$ by the assumption of (ii), we get in total that \begin{align*}
|A|&=\left|A_{(\frac{1}{2},1]}\right|+\left|B_1\setminus A_{(\frac{1}{2},\frac{2}{3}]}\right|\\
&\leqslant \left|A_{(\frac{1}{2},1]}\right|+\frac{i_m}{3}-\frac{2\left|A_{(\frac{1}{2},1]}\right|}{9}+4\\
&\leqslant \frac{7\left|A_{(\frac{1}{2},1]}\right|}{9}+\frac{7n}{54}+\frac{\varepsilon n}{3}+4\\
&\leqslant \frac{7\left|A_{(\frac{2}{3},1]}\right|}{6}+\frac{7n}{54}+\frac{\varepsilon n}{3}+5\\
&\leqslant \frac{35n}{108}+\frac{\varepsilon n}{3}+23
\end{align*} using for the penultimate inequality that $\left|A_{(\frac{1}{2},1]}\right|\leqslant\frac{3\left|A_{(\frac{2}{3},1]}\right|}{2}+1$ by \eqref{A_2lowe} and for the final inequality that $\left|A_{(\frac{2}{3},1]}\right|\leqslant \frac{n}{6}+24$ in Case 3. This gives the desired bound \eqref{inductiontheo} if we choose $\varepsilon$ small enough.
\end{enumerate}
This finishes the proof of 3.1.3. \qed \\
\bigskip

Hence, we have proved the desired bound \eqref{inductiontheo} in each of the three cases 3.1.1, 3.1.2 and 3.1.3 and this finishes the proof of Subcase 3.1 under the assumption that conclusion (2) in Theorem \ref{freimain} holds, namely that $A_{(\frac{1}{2},1]}^{1(3)}+A_{(\frac{1}{2},1]}^{2(3)}$ contains an arithmetic progression $Q$ of size $\left|A_{(\frac{1}{2},1]}^{1(3)}\right|+\left|A_{(\frac{1}{2},1]}^{2(3)}\right|-1$.
\bigskip

To finish the proof of Subcase 3.1, the final case that we need to consider is the following. By the remaining conclusion (1) from Theorem \ref{freimain}, we may assume that 
\begin{equation}
    \left|A_{(\frac{1}{2},1]}^{1(3)}+A_{(\frac{1}{2},1]}^{2(3)}\right|\geqslant \left|A_{(\frac{1}{2},1]}^{1(3)}\right|+\left|A_{(\frac{1}{2},1]}^{2(3)}\right|+\min\left(\left|A_{(\frac{1}{2},1]}^{1(3)}\right|,\left|A_{(\frac{1}{2},1]}^{2(3)}\right|\right)-3.
\label{freimaindouble}
\end{equation}
In this final part of the argument, we shall make use of yet another construction which from the numbers in $A\cap\left[\frac{2n}{3}\right]$ produces a subset $Z$ of $\left(\frac{n}{3},\frac{2n}{3}\right]$.\footnote{Following our earlier notation it might be more natural to call this constructed set $B_3$ instead of $Z$, but this would make later use of subscripts confusing.} For every number $a\in A\cap[\frac{n}{2}]$, there is a unique power of $3$, say $3^{m_a}$, so that $3^{m_a}a \in \left(\frac{n}{6},\frac{n}{2}\right]$. Note that there are no coincidences of the form $3^{m_a}a=3^{m_b}b$ with $b\neq a$ by (2) in Lemma \ref{basicInteger}. We define \begin{equation}
Z\vcentcolon = \left\{3^{m_a}a: \text{$3^{m_a}a\in \left(\frac{2n}{9},\frac{n}{2}\right]$}\right\}\cup\left\{2\cdot3^{m_a}a: \text{$3^{m_a}a\in \left(\frac{n}{6},\frac{2n}{9}\right]$}\right\}\subset \left(\frac{2n}{9},\frac{n}{2}\right] 
\label{Zdefi}
\end{equation}Again there are no coincidences as else we would get an equality of the form $2\cdot3^{m_a}a = 3^{m_b}b$ but if $m_a\geqslant m_b$ then $b>a$ and $a|b$, while if $m_a<m_b$ then $a>b$ and $b|a+a$ so both of these lead to a contradiction as $A$ has property P. So we have constructed a set $Z\subset \left(\frac{2n}{9},\frac{n}{2}\right]$ of size
\begin{equation}
    |Z|=\left|A\cap\left[\frac{n}{2}\right]\right|.
    \label{Zsize}
\end{equation} Let us split $Z$ into the following parts which by definition form a partition of $Z$:
\begin{align}
Z_{(\frac{2}{9},\frac{1}{3}]}^{1(2)} &\vcentcolon= Z\cap\left(\frac{2n}{9},\frac{n}{3}\right]\cap\left(1+2\!\cdot\!\mathbf{N}\right), \label{Zpart}\\
Z_B &\vcentcolon= \bigg\{z\in Z\cap\left(\frac{2n}{9},\frac{n}{3}\right]\cap(2\!\cdot\!\mathbf{N}): \text{$\frac{3z}{2}\in Z$}.\bigg\},\nonumber\\
Z_G &\vcentcolon= \bigg\{z\in Z\cap\left(\frac{2n}{9},\frac{n}{3}\right]\cap(2\!\cdot\!\mathbf{N}): \text{$\frac{3z}{2}\notin Z$}.\bigg\}\nonumber\\
Z_{(\frac{1}{3},\frac{1}{2}]} &\vcentcolon= Z\cap\left(\frac{n}{3},\frac{n}{2}\right].\nonumber
\end{align} 
Note that $Z_{(\frac{2}{9},\frac{1}{3}]}^{1(2)}$ by definition consists of odd numbers only. In \eqref{Zpart}, we split $Z_{(\frac{2}{9},\frac{1}{3}]}^{0(2)}=Z\cap\left(\frac{2n}{9},\frac{n}{3}\right]\cap(2\!\cdot\!\mathbf{N})$ into the two sets $Z_B$ and $Z_G$. We think of $Z_B$ as the set of `bad' numbers in $\left(\frac{2n}{9},\frac{n}{3}\right]$ as we cannot further replace such a number $z\in Z_B$ by $\frac{3z}{2}$ since this number is already in $Z$. $Z_G$ on the other hand is the set of `good' numbers in $\left(\frac{2n}{9},\frac{n}{3}\right]$. To make use of these `good' numbers, we define \begin{equation}
    Z_{(\frac{1}{3},\frac{1}{2}],G}\vcentcolon= \left(\frac{3}{2}\!\cdot\!Z_G\right)\cup Z_{(\frac{1}{3},\frac{1}{2}]}\subset \left(\frac{n}{3},\frac{n}{2}\right],
    \label{Z_RGdefi}
\end{equation}
so $Z_{(\frac{1}{3},\frac{1}{2}],G}$ is a set of integers as $Z_G$ contains only even numbers, and \begin{equation}
    \left|Z_{(\frac{1}{3},\frac{1}{2}],G}\right|=\left|Z_{(\frac{1}{3},\frac{1}{2}]}\right|+\left|Z_G\right|
    \label{Z_RG}
\end{equation}
as the sets $Z_{(\frac{1}{3},\frac{1}{2}]}$ and $\frac{3}{2}\!\cdot\!Z_G$ are disjoint by definition of $Z_G$. We shall need two lemmas which are very similar to Lemmas \ref{oddsmall} and \ref{divideby4}.
\begin{lemma}
Let $A_{(\frac{2}{3},1]}^{i(4)}$ be the set of numbers in $A_{(\frac{2}{3},1]}$ that are $i\!\!\mod 4$. Then
\begin{align}
    \left|Z_{(\frac{2}{9},\frac{1}{3}]}^{1(2)}\right|+\left|A_{(\frac{2}{3},1]}^{1(4)}\right|\leqslant \frac{n}{12}+3, \nonumber \\
    \left|Z_{(\frac{2}{9},\frac{1}{3}]}^{1(2)}\right|+\left|A_{(\frac{2}{3},1]}^{3(4)}\right|\leqslant \frac{n}{12}+3,\nonumber
\end{align}
\label{Coddcase2}
\end{lemma}
\begin{proof}
The proof is completely analogous to the proof of Lemma \ref{oddsmall} if we replace $B_2^{\operatorname{L}}$ by $Z_{(\frac{2}{9},\frac{1}{3}]}^{1(2)}$ throughout. 
\end{proof}
\begin{lemma}
One of the following three statements holds.
\begin{itemize}
\item [(1)] We have that
\begin{align*}
    \left|Z_{(\frac{2}{9},\frac{1}{3}]}^{1(2)}\right|+\left|Z_{(\frac{1}{3},\frac{1}{2}],G}\right|+\left|A_{(\frac{2}{3},1]}\right|\leqslant \frac{n}{4}+4.
\end{align*}
    \item [(2)] We have that \begin{align*}
    &\left|Z_{(\frac{1}{3},\frac{1}{2}],G}\right|
    \leqslant  \frac{n}{6}+5\\
    &\,\,\,\,\,\,\,\,\,\,\,\,\,\,\,\,\,\,\,\,\,\,\,\,-\max\left(\left|A_{(\frac{2}{3},1]}^{1(4)}\right|\!+\!\left|A_{(\frac{2}{3},1]}^{3(4)}\right|+\min\left(\left|A_{(\frac{2}{3},1]}^{1(4)}\right|,\left|A_{(\frac{2}{3},1]}^{3(4)}\right|\right),3\left|A_{(\frac{2}{3},1]}^{0(4)}\right|,3\left|A_{(\frac{2}{3},1]}^{2(4)}\right|\right).
    \end{align*}
\item [(3)] We have that \begin{align*}
    \left|Z_{(\frac{1}{3},\frac{1}{2}],G}\right|\leqslant \frac{n}{6}+2-\max\left(\left|A_{(\frac{2}{3},1]}^{1(4)}\right|+\left|A_{(\frac{2}{3},1]}^{3(4)}\right|,2\left|A_{(\frac{2}{3},1]}^{0(4)}\right|,2\left|A_{(\frac{2}{3},1]}^{2(4)}\right|\right)
    \end{align*} and that $\left(A_{(\frac{2}{3},1]}+A_{(\frac{2}{3},1]}\right)\cap (4\!\cdot\!\mathbf{N})$ contains an arithmetic progression of size at least $\frac{\left|A_{(\frac{2}{3},1]}\right|}{2}-1$.
\end{itemize}
\label{lemmaZ}
\end{lemma}
\begin{proof}
The proof is analogous to the proof of Lemma \ref{divideby4}.
First assume that either $A_{(\frac{2}{3},1]}^{1(4)}$ or $A_{(\frac{2}{3},1]}^{3(4)}$ is empty. Then by Lemma \ref{Coddcase2} we obtain \begin{align}
    \left|Z_{(\frac{2}{9},\frac{1}{3}]}^{1(2)}\right|+\left|A_{(\frac{2}{3},1]}\cap (1+2\!\cdot\!\mathbf{N})\right|&\leqslant \left|Z_{(\frac{2}{9},\frac{1}{3}]}^{1(2)}\right|+\max\left(\left|A_{(\frac{2}{3},1]}^{1(4)}\right|,\left|A_{(\frac{2}{3},1]}^{3(4)}\right|\right)
    \label{rara1}\\
    &\leqslant \frac{n}{12} +3.
    \nonumber
\end{align} 
We also have that
\begin{align}
    \left|Z_{(\frac{1}{3},\frac{1}{2}],G}\right|+\left|A_{(\frac{2}{3},1]}\cap (2\!\cdot\!\mathbf{N})\right|\leqslant \frac{n}{6}+1
    \label{rara2}
\end{align} since $2\!\cdot\!Z_{(\frac{1}{3},\frac{1}{2}],G}$ and $A_{(\frac{2}{3},1]}\cap (2\!\cdot\!\mathbf{N})$ are disjoint sets of even numbers in $\left(\frac{2n}{3},n\right]$, as otherwise a number in $2\!\cdot\!Z_{(\frac{1}{3},\frac{1}{2}],G}$ would divide a number in $A_{(\frac{2}{3},1]}$ contradicting that $A$ has property P by the definitions \eqref{Zpart} and \eqref{Z_RGdefi} of $Z_{(\frac{1}{3},\frac{1}{2}],G}$. So we obtain the desired inequality in conclusion (1) in Lemma \ref{lemmaZ} by adding the bounds \eqref{rara1} and \eqref{rara2}: \begin{align*}
    &\left|Z_{(\frac{2}{9},\frac{1}{3}]}^{1(2)}\right|+\big|A_{(\frac{2}{3},1]}\cap (1+2\!\cdot\!\mathbf{N})\big|+\left|Z_{(\frac{1}{3},\frac{1}{2}],G}\right|+\big|A_{(\frac{2}{3},1]}\cap (2\!\cdot\!\mathbf{N})\big|
    \leqslant \frac{n}{4}+4.
\end{align*}

Now assume that neither $A_{(\frac{2}{3},1]}^{1(4)}$ nor $A_{(\frac{2}{3},1]}^{3(4)}$ is empty. Then $\left|A_{(\frac{2}{3},1]}^{1(4)}+A_{(\frac{2}{3},1]}^{3(4)}\right|\geqslant \left|A_{(\frac{2}{3},1]}^{1(4)}\right|+\left|A_{(\frac{2}{3},1]}^{3(4)}\right|-1$, and we also have $\left|A_{(\frac{2}{3},1]}^{j(4)}+A_{(\frac{2}{3},1]}^{j(4)}\right|\geqslant 2\left|A_{(\frac{2}{3},1]}^{j(4)}\right|-1$ for $j=0,2$. Hence, $A_{(\frac{2}{3},1]}+A_{(\frac{2}{3},1]}$ contains at least 
\begin{align}
    \left|\left(A_{(\frac{2}{3},1]}+A_{(\frac{2}{3},1]}\right)\cap(4\!\cdot\!\mathbf{N})\right|&\geqslant\max\left(\left|A_{(\frac{2}{3},1]}^{1(4)}\right|+\left|A_{(\frac{2}{3},1]}^{3(4)}\right|,2\left|A_{(\frac{2}{3},1]}^{0(4)}\right|,2\left|A_{(\frac{2}{3},1]}^{2(4)}\right|\right)-1
    \label{rara3}
\end{align} multiples of 4 in $\left(\frac{4n}{3},2n\right]$. Note that the dilated set $4\cdot Z_{(\frac{1}{3},\frac{1}{2}],G}$ also consists only of multiples of 4 in $\left(\frac{4n}{3},2n\right]$ so using Lemma \ref{basicmultl} with \eqref{rara3} yields the inequality \begin{align}
\left|Z_{(\frac{1}{3},\frac{1}{2}],G}\right|&\leqslant \frac{n}{6}+1 -\left|\left(A_{(\frac{2}{3},1]}+A_{(\frac{2}{3},1]}\right)\cap(4\!\cdot\!\mathbf{N})\right|\nonumber\\
&\leqslant \frac{n}{6}+2-\max\left(\left|A_{(\frac{2}{3},1]}^{1(4)}\right|+\left|A_{(\frac{2}{3},1]}^{3(4)}\right|,2\left|A_{(\frac{2}{3},1]}^{0(4)}\right|,2\left|A_{(\frac{2}{3},1]}^{2(4)}\right|\right) 
\label{rara4}
\end{align}
which is the desired inequality from (3) in this lemma. In fact, we can improve on this using Theorem \ref{freimain} to either obtain a larger sumset if (1) in Theorem \ref{freimain} holds so that
\begin{align*}
    &\left|\left(A_{(\frac{2}{3},1]}+A_{(\frac{2}{3},1]}\right)\cap(4\!\cdot\!\mathbf{N})\right|\\
    &\geqslant\max\left(\left|A_{(\frac{2}{3},1]}^{1(4)}\right|+\left|A_{(\frac{2}{3},1]}^{3(4)}\right|+\min\left(\left|A_{(\frac{2}{3},1]}^{1(4)}\right|,\left|A_{(\frac{2}{3},1]}^{3(4)}\right|\right),3\left|A_{(\frac{2}{3},1]}^{0(4)}\right|,3\left|A_{(\frac{2}{3},1]}^{2(4)}\right|\right)-3,
\end{align*} which gives the inequality from part (2) in this lemma by using this instead of \eqref{rara3} in \eqref{rara4}, or else the inequality \eqref{rara4} still holds and we deduce from conclusion (2) in Theorem \ref{freimain} that $\left(A_{(\frac{2}{3},1]}+A_{(\frac{2}{3},1]}\right)\cap(4\!\cdot\!\mathbf{N})$ contains an arithmetic progression of size at least $$\max\left(\left|A_{(\frac{2}{3},1]}^{1(4)}\right|+\left|A_{(\frac{2}{3},1]}^{3(4)}\right|,2\left|A_{(\frac{2}{3},1]}^{0(4)}\right|,2\left|A_{(\frac{2}{3},1]}^{2(4)}\right|\right)-1\geqslant\frac{\left|A_{(\frac{2}{3},1]}\right|}{2} -1.$$ Hence, in this case part (3) in this lemma holds and this finishes the proof of Lemma \ref{lemmaZ}.
\end{proof}
The following corollary will be important.
\begin{cor}
If either of $(1)$ or $(2)$ in Lemma \ref{lemmaZ} holds, then we may assume that
\begin{align}
    \left|A_{(\frac{1}{2},\frac{2}{3}]}\right|+|Z_B|\geqslant \frac{n}{12}-8.
\label{cor11}
\end{align}
If $(3)$ in  Lemma \ref{lemmaZ} holds, then we may assume that
\begin{align}
    \left|A_{(\frac{1}{2},\frac{2}{3}]}\right|+|Z_B|\geqslant \frac{n}{24}-11.
\label{cor120}
\end{align}
\label{cor1}
\end{cor}
\begin{proof}
Recall that $|Z|=\left|A\cap\left[\frac{n}{2}\right]\right|$ by \eqref{Zsize}, so we get
\begin{align}
|A| -\left|A_{(\frac{1}{2},\frac{2}{3}]}\right|-|Z_B| 
&= \left|A_{(\frac{2}{3},1]}\right|+\left|A_{[\frac{1}{2}]}\right|-|Z_B| \nonumber\\
&= \left|A_{(\frac{2}{3},1]}\right|+|Z|-|Z_B|\nonumber\\
&=\left|A_{(\frac{2}{3},1]}\right|+\left|Z_{(\frac{2}{9},\frac{1}{3}]}^{1(2)}\right|+\left|Z_{(\frac{1}{3},\frac{1}{2}],G}\right|
\label{bababa1}
\end{align}
where the last equality follows as $|Z|=\left|Z_{(\frac{2}{9},\frac{1}{3}]}^{1(2)}\right|+\left|Z_{(\frac{1}{3},\frac{1}{2}],G}\right|+|Z_B|$ by \eqref{Zpart} and \eqref{Z_RG}.
If the inequality from case (1) in Lemma \ref{lemmaZ} holds, then using this in \eqref{bababa1} gives $|A|-\left|A_{(\frac{1}{2},\frac{2}{3}]}\right|-|Z_B|\leqslant \frac{n}{4} +4$ so we either get \eqref{cor11} or else that $|A|\leqslant \frac{n}{3}$ and we are done.
If (2) in Lemma \ref{lemmaZ} holds, we can use this upper bound for $\left|Z_{(\frac{1}{3},\frac{1}{2}],G}\right|$ together with the bound on $\left|Z_{(\frac{2}{9},\frac{1}{3}]}^{1(2)}\right|$ from Lemma \ref{Coddcase2} in \eqref{bababa1} to obtain
\begin{align*}
    &|A| -\left|A_{(\frac{1}{2},\frac{2}{3}]}\right|-|Z_B| = \left|A_{(\frac{2}{3},1]}\right|+\left|Z_{(\frac{2}{9},\frac{1}{3}]}^{1(2)}\right|+\left|Z_{(\frac{1}{3},\frac{1}{2}],G}\right|\\
    &\leqslant \left|A_{(\frac{2}{3},1]}\right|+ \frac{n}{6}+5\\
    &\,\,\,\,\,\,-\max\left(\left|A_{(\frac{2}{3},1]}^{1(3)}\right|+\left|A_{(\frac{2}{3},1]}^{3(4)}\right|+\min\left(\left|A_{(\frac{2}{3},1]}^{1(3)}\right|,\left|A_{(\frac{2}{3},1]}^{3(4)}\right|\right),3\left|A_{(\frac{2}{3},1]}^{0(3)}\right|,3\left|A_{(\frac{2}{3},1]}^{2(3)}\right|\right)\\
    &\,\,\,\,\,\,+\frac{n}{12}+3-\max\bigg(\left|A_{(\frac{2}{3},1]}^{1(3)}\right|,\left|A_{(\frac{2}{3},1]}^{3(4)}\right|\bigg) \\
    &\leqslant \left|A_{(\frac{2}{3},1]}\right|+\frac{n}{4}+8-\left|A_{(\frac{2}{3},1]}\right|= \frac{n}{4}+8,
\end{align*} using the basic inequality $\max(x\!+\!y\!+\!\min(x,y),3z,3w)\!+\!\max(x,y)\geqslant x\!+\!y\!+\!z\!+\!w$. Hence, we may assume that \eqref{cor11} holds as else $|A|\leqslant \frac{n}{3}$ so we would be done. Finally, if case (3) in Lemma \ref{lemmaZ} holds, we can use this upper bound for $\left|Z_{(\frac{1}{3},\frac{1}{2}],G}\right|$ together with the bound on $\left|Z_{(\frac{2}{9},\frac{1}{3}]}^{1(2)}\right|$ from Lemma \ref{Coddcase2} in \eqref{bababa1} so that
\begin{align*}
    &|A| -\left|A_{(\frac{1}{2},\frac{2}{3}]}\right|-|Z_B| = \left|A_{(\frac{2}{3},1]}\right|+\left|Z_{(\frac{2}{9},\frac{1}{3}]}^{1(2)}\right|+\left|Z_{(\frac{1}{3},\frac{1}{2}],G}\right|\\
    &\leqslant \left|A_{(\frac{2}{3},1]}\right|+\frac{n}{6}+2-\max\left(\left|A_{(\frac{2}{3},1]}^{1(3)}\right|+\left|A_{(\frac{2}{3},1]}^{3(4)}\right|,2\left|A_{(\frac{2}{3},1]}^{0(3)}\right|,2\left|A_{(\frac{2}{3},1]}^{2(3)}\right|\right)\\
    &\,\,\,\,\,\,+\frac{n}{12}+3-\max\left(\left|A_{(\frac{2}{3},1]}^{1(3)}\right|,\left|A_{(\frac{2}{3},1]}^{3(4)}\right|\right) \\
    &\leqslant \left|A_{(\frac{2}{3},1]}\right|+\frac{n}{4}+5-\frac{3\left|A_{(\frac{2}{3},1]}\right|}{4}\leqslant \frac{7n}{24}+11,
\end{align*}
using that $\max(x+y,2z,2w)+\max(x,y)\geqslant \frac{3(x+y+z+w)}{4}$ and in last inequality that $\left|A_{(\frac{2}{3},1]}\right|<\frac{n}{6}+24$ as we are in Case 3. So either $|A|\leqslant \frac{n}{3}$ and we are done or we may assume that \eqref{cor120} holds.

\end{proof}
Recall that the starting point of our proof in Subcase 3.1 was to note that the set $A'''$, which we defined in \eqref{A''definition} as $A'''=\frac{1}{3}\cdot\left(\left(A_{(\frac{1}{2},1]}+A_{(\frac{1}{2},1]}\right)\cap 3\cdot\mathbf{N}\right)$, and the auxiliary set $B_1$ are disjoint subsets of $\left(\frac{n}{3},\frac{2n}{3}\right]$ because $A$ has property P. This idea directly led to the crucial inequality \eqref{dyadicembed} and it was also used with various other auxiliary sets in the proofs of Cases 3.1.1, 3.1.2 and 3.1.3. Here, for the final time, we find one more way of mapping all of $A\cap \left[\frac{2n}{3}\right]$ into $\left(\frac{n}{3},\frac{2n}{3}\right]$ which then leads to the construction of an auxiliary set $B_3\subset \left(\frac{n}{3},\frac{2n}{3}\right]$ that will be disjoint from the set $A'''$. We do this as follows by making use of the partition of $Z$ into $Z_{(\frac{2}{9},\frac{1}{3}]}^{1(2)},Z_B,Z_G$ and $Z_{(\frac{1}{3},\frac{1}{2}]}$ that we defined in \eqref{Zpart}. Define
\begin{align}
    B_3\vcentcolon= \bigg(2\!\cdot\!\left(Z_{(\frac{2}{9},\frac{1}{3}]}^{1(2)}\cup Z_B\cup Z_G\right)\bigg)\cup Z_{(\frac{1}{3},\frac{1}{2}]}\cup\bigg(\frac{9}{4}\!\cdot\!Z_B\bigg)\cup A_{(\frac{1}{2},\frac{2}{3}]}.
\label{B_3defi}
\end{align}
We need check a couple of crucial properties of $B_3$ which we collect in the following lemma.
\begin{lemma}
We have that $B_3$ is a set of integers contained in $\left(\frac{n}{3},\frac{2n}{3}\right]$, that $B_3$ has size $|B_3|= \left|A_{[\frac{2}{3}]}\right|+|Z_B|$, and crucially that $B_3$ is disjoint from the set $A'''$ that we defined in \eqref{A''definition}.
\label{Ycrucial}
\end{lemma}
\begin{proof}
Recall that $|Z|=\left|A\cap\left[\frac{n}{2}\right]\right|$ by \eqref{Zsize} so $\left|A_{(\frac{1}{2},\frac{2}{3}]}\right|+|Z|=\left|A\cap\left(\frac{n}{2},\frac{2n}{3}\right]\right|+\left|A\cap\left[\frac{n}{2}\right]\right|=\left|A_{[\frac{2}{3}]}\right|$. Hence, if we show that the sets $2\!\cdot\!Z_{(\frac{2}{9},\frac{1}{3}]}^{1(2)},2\!\cdot\!Z_B, 2\!\cdot\!Z_G, Z_{(\frac{1}{3},\frac{1}{2}]},\frac{9}{4}\!\cdot\!Z_B$ and $A_{(\frac{1}{2},\frac{2}{3}]}$ are pairwise disjoint, then by \eqref{B_3defi} we get that $B_3$ has size 
\begin{align*}
|B_3|&=\left|Z_{(\frac{2}{9},\frac{1}{3}]}^{1(2)}\right|+|Z_B|+|Z_G|+\left|Z_{(\frac{1}{3},\frac{1}{2}]}\right|+|Z_B|+\left|A_{(\frac{1}{2},\frac{2}{3}]}\right|\\
&=|Z|+|Z_B|+\left|A_{(\frac{1}{2},\frac{2}{3}]}\right|\\
&=\left|A_{[\frac{2}{3}]}\right|+|Z_B|
\end{align*} as desired.

\bigskip

First, it is easy to see from the definition \eqref{Zpart} that $2\!\cdot\!Z_{(\frac{2}{9},\frac{1}{3}]}^{1(2)},2\!\cdot\!Z_B, 2\!\cdot\!Z_G, Z_{(\frac{1}{3},\frac{1}{2}]}$ and $A_{(\frac{1}{2},\frac{2}{3}]}=A\cap\left(\frac{n}{2},\frac{2n}{3}\right]$ are integer subsets of $\left(\frac{n}{3},\frac{2n}{3}\right]$. They are pairwise disjoint as by \eqref{Zdefi}, any number in $2\!\cdot\!Z_{(\frac{2}{9},\frac{1}{3}]}^{1(2)},2\!\cdot\!Z_B, 2\!\cdot\!Z_G$ or $Z_{(\frac{1}{3},\frac{1}{2}]}$ is of the form $3^{m_a}a$ or $2\cdot3^{m_a}a$ for some $a\in A_{[\frac{1}{2}]}$ and we recall that incidences of the form $3^{m_a}a=3^{m_b}b$ or $3^{m_a}a=2\cdot3^{m_b}b$ with $a,b\in A$ are impossible for $a\neq b$ by Lemma \ref{basicInteger}. Since every element of $2\!\cdot\!Z_{(\frac{2}{9},\frac{1}{3}]}^{1(2)},2\!\cdot\!Z_B, 2\!\cdot\!Z_G$ or $Z_{(\frac{1}{3},\frac{1}{2}]}$ is a multiple of a number in $A_{[\frac{1}{2}]}$ by definitions \eqref{Zdefi} and \eqref{Zpart}, no such element lies in $A_{(\frac{1}{2},\frac{2}{3}]}$ as $A$ has property P. So $2\!\cdot\!Z_{(\frac{2}{9},\frac{1}{3}]}^{1(2)},2\!\cdot\!Z_B, 2\!\cdot\!Z_G$, $Z_{(\frac{1}{3},\frac{1}{2}]}$ and $A_{(\frac{1}{2},\frac{2}{3}]}$ are pairwise disjoint subsets of $\left(\frac{n}{3},\frac{2n}{3}\right]$. We have also shown that any number in $2\!\cdot\!Z_{(\frac{2}{9},\frac{1}{3}]}^{1(2)},2\!\cdot\!Z_B, 2\!\cdot\!Z_G$, $Z_{(\frac{1}{3},\frac{1}{2}]}$ or $A_{(\frac{1}{2},\frac{2}{3}]}$ is a multiple of some number in $A_{[\frac{2}{3}]}$ so all these sets are disjoint from $A'''$ as we proved in the argument preceding \eqref{dyadicembed} that $A'''$ does not contain a multiple of any number in $A_{[\frac{2}{3}]}$.

\bigskip

It only remains to show that $\frac{9}{4}\!\cdot\!Z_B$ is a set of integers contained in $\left(\frac{n}{3},\frac{2n}{3}\right]$, that it is disjoint from $2\!\cdot\!Z_{(\frac{2}{9},\frac{1}{3}]}^{1(2)},2\!\cdot\!Z_B, 2\!\cdot\!Z_G$, $Z_{(\frac{1}{3},\frac{1}{2}]}$ and $A_{(\frac{1}{2},\frac{2}{3}]}$ and that it is also disjoint from $A'''$. Let us analyse what an element of $Z_B$ looks like. If $z\in Z_B$, then by the definition \eqref{Zpart} we have that $\frac{3z}{2}\in Z\cap\left(\frac{n}{3},\frac{n}{2}\right]$, so we must have that $\frac{3z}{2} = 3^{m_a}a$ or $\frac{3z}{2} = 2\!\cdot\!3^{m_a}a$ for some $a\in A_{[\frac{1}{2}]}$. The first of these is impossible by Lemma \ref{basicInteger}. So $\frac{3z}{2}=2\cdot3^{m_a}a$ for some $a\in A_{[\frac{1}{2}]}$ and this implies that $\frac{3z}{2}=2\!\cdot\!3^{m_a}a\in\left(\frac{n}{3},\frac{4n}{9}\right]$ as one can see from \eqref{Zdefi} that we constructed $Z$ in such a way that any number of the form of $2\!\cdot\!3^{m_a}a$ lies in $\left(\frac{n}{3},\frac{4n}{9}\right]$. The point of all this is that $\frac{9z}{4} = \frac{3}{2}\!\cdot\!2\cdot3^{m_a}a = 3^{m_a+1}a\in\frac{3}{2}\cdot\left(\frac{n}{3},\frac{4n}{9}\right]\subset\left(\frac{n}{2},\frac{2n}{3}\right]$. Hence, $\frac{9}{4}\!\cdot\!Z_B$ is a set of integers contained in $\left(\frac{n}{2},\frac{2n}{3}\right]$ so it is trivially disjoint from $Z_{(\frac{1}{3},\frac{1}{2}]}\subset \left(\frac{n}{3},\frac{n}{2}\right]$. Further note that every element in $\frac{9}{4}\!\cdot\!Z_B$ is a multiple of some $a\in A_{[\frac{1}{2}]}$ of the form $3^{m_a+1}a$ so it cannot lie in $A_{(\frac{1}{2},\frac{2}{3}]}$ as $A$ has property P, and for the same reason $\frac{9}{4}\!\cdot\!Z_B$ is disjoint from $A'''$ as all elements of $A'''$ have a multiple in $A_{(\frac{1}{2},1]}+A_{(\frac{1}{2},1]}$ by definition \eqref{A''definition}. Finally, $\frac{9}{4}\!\cdot\!Z_B$ is disjoint from each of the sets $2\!\cdot\!Z_{(\frac{2}{9},\frac{1}{3}]}^{1(2)},2\!\cdot\!Z_B$ and $2\!\cdot\!Z_G$ as else, recalling the definitions \eqref{Zdefi} and \eqref{Zpart}, we would obtain an equality of the form $2\!\cdot\!3^{m_b}b=3^{m_a+1}a$ for some $a,b\in A$ which is impossible by Lemma \ref{basicInteger}.
\end{proof}

We proved in the Lemma \ref{Ycrucial} that $B_3$ and $A'''$ are disjoint subsets of $\left(\frac{n}{3},\frac{2n}{3}\right]$. As $|B_3| = \left|A_{[\frac{2}{3}]}\right|+|Z_B|$ by Lemma \ref{Ycrucial} we obtain
\begin{equation}
    \left|A_{[\frac{2}{3}]}\right|+|Z_B|+\left|A'''\right|=|B_3|+\left|A'''\right|\leqslant \left\lceil \frac{n}{3}\right\rceil.
\label{goodineq}
\end{equation}
As $A'''=\frac{1}{3}\cdot\left(\left(A_{(\frac{1}{2},1]}+A_{(\frac{1}{2},1]}\right)\cap 3\cdot\mathbf{N}\right)$ by definition \eqref{A''definition}, we can use \eqref{freimaindouble} to get that 
\begin{align*}
    \left|A'''\right|=\left|\left(A_{(\frac{1}{2},1]}+A_{(\frac{1}{2},1]}\right)\cap 3\cdot\mathbf{N}\right|
    \geqslant \left|A_{(\frac{1}{2},1]}^{1(3)}\right|+\left|A_{(\frac{1}{2},1]}^{2(3)}\right|+\min\left(\left|A_{(\frac{1}{2},1]}^{1(3)}\right|,\left|A_{(\frac{1}{2},1]}^{2(3)}\right|\right)-3
\end{align*} and so we deduce from \eqref{goodineq} that
\begin{equation}
    \left|A_{[\frac{2}{3}]}\right|+|Z_B|+\left|A_{(\frac{1}{2},1]}^{1(3)}\right|+\left|A_{(\frac{1}{2},1]}^{2(3)}\right|+\min\left(\left|A_{(\frac{1}{2},1]}^{1(3)}\right|,\left|A_{(\frac{1}{2},1]}^{2(3)}\right|\right)-3\leqslant \left\lceil \frac{n}{3}\right\rceil.
\label{goodineqnew}
\end{equation}
From now on we suppose that $\min\left(\left|A_{(\frac{1}{2},1]}^{1(3)}\right|,\left|A_{(\frac{1}{2},1]}^{2(3)}\right|\right) = \left|A_{(\frac{1}{2},1]}^{2(3)}\right|$ as the proof in the other case is the same after interchanging the roles of $A_{(\frac{1}{2},1]}^{1(3)}$ and $A_{(\frac{1}{2},1]}^{2(3)}$ in what follows. We may assume for the remainder of the proof that
\begin{equation}
    \left|A_{(\frac{1}{2},1]}^{1(3)}\right|+2\left|A_{(\frac{1}{2},1]}^{2(3)}\right|+|Z_B|-3\leqslant \left|A_{(\frac{2}{3},1]}\right|,
\label{Mainbound}
\end{equation}
for if this inequality did not hold, then as $|A|=\left|A_{[\frac{2}{3}]}\right|+\left|A_{(\frac{2}{3},1]}\right|$, we would get \begin{align*}
    |A|&= \left|A_{[\frac{2}{3}]}\right|+\left|A_{(\frac{2}{3},1]}\right|\\
    &\leqslant \left|A_{[\frac{2}{3}]}\right|+\left|A_{(\frac{1}{2},1]}^{1(3)}\right|+2\left|A_{(\frac{1}{2},1]}^{2(3)}\right|+|Z_B|-3\leqslant\left\lceil \frac{n}{3}\right\rceil,
\end{align*}
since the final inequality is precisely inequality \eqref{goodineqnew}.
\medskip

The final ingredient required to finish the argument is a good bound on $\left|A_{[\frac{1}{2}]}\right|$, so this is our next goal. Fortunately, it turns out that a relatively simple argument suffices here. We begin by showing that $\left|A_{(\frac{1}{2},1]}^{1(3)}\right|=\max\left(\left|A_{(\frac{1}{2},1]}^{1(3)}\right|,\left|A_{(\frac{1}{2},1]}^{2(3)}\right|\right)$ is fairly large. By \eqref{Mainbound} and as $\left|A_{(\frac{1}{2},1]}^{1(3)}\right|+\left|A_{(\frac{1}{2},1]}^{2(3)}\right|\geqslant\frac{2\left|A_{(\frac{1}{2},1]}\right|}{3}$ by our assumption in Subcase 3.1, we have that
\begin{align*}
   \left|A_{(\frac{2}{3},1]}\right|&\geqslant \left|A_{(\frac{1}{2},1]}^{1(3)}\right|+2\left|A_{(\frac{1}{2},1]}^{2(3)}\right|+|Z_B|-3 \\
   &\geqslant \frac{2\left|A_{(\frac{1}{2},1]}\right|}{3}+\left|A_{(\frac{1}{2},1]}^{2(3)}\right|+|Z_B|-3.
\end{align*}
After expanding $\left|A_{(\frac{1}{2},1]}\right|=\left|A_{(\frac{1}{2},\frac{2}{3}]}\right|+\left|A_{(\frac{2}{3},1]}\right|$, this rearranges to
\begin{equation}
    \left|A_{(\frac{1}{2},1]}^{2(3)}\right|\leqslant\frac{\left|A_{(\frac{2}{3},1]}\right|}{3}-\frac{2\left|A_{(\frac{1}{2},\frac{2}{3}]}\right|}{3}-|Z_B|+3.
\label{mamama0}
\end{equation}
Hence, starting with the assumption of Subcase 3.1 and using the lower bound \eqref{mamama0} yields
\begin{align*}
    \left|A_{(\frac{1}{2},1]}^{1(3)}\right|&\geqslant \frac{2\left|A_{(\frac{1}{2},1]}\right|}{3}-\left|A_{(\frac{1}{2},1]}^{2(3)}\right|\\
    &\geqslant \frac{2\left|A_{(\frac{1}{2},1]}\right|}{3}+\frac{2\left|A_{(\frac{1}{2},\frac{2}{3}]}\right|}{3}-\frac{\left|A_{(\frac{2}{3},1]}\right|}{3}+|Z_B|-3\\
    &= \frac{\left|A_{(\frac{1}{2},1]}\right|}{3}+\left|A_{(\frac{1}{2},\frac{2}{3}]}\right|+|Z_B|-3.
\end{align*}
We have that $\left|A_{(\frac{1}{2},\frac{2}{3}]}\right|+|Z_B|\geqslant \frac{n}{24}-11$ by Corollary \ref{cor1} and that $\left|A_{(\frac{1}{2},1]}\right|\geqslant \left(\frac{1}{3}-\delta\right)\frac{n}{2}$  by \eqref{inductionbound}, so using this in the inequality above gives
\begin{equation}
    \left|A_{(\frac{1}{2},1]}^{1(3)}\right|\geqslant \left(\frac{1}{3}-\delta\right)\frac{n}{6}+\frac{n}{24}-14>\frac{n}{12}+\frac{n}{100}
\label{A_u,1dense}
\end{equation}
if we choose $\delta$ sufficiently small.
This shows that $A_{(\frac{1}{2},1]}^{1(3)}$ has density significantly greater than a half on the arithmetic progression of numbers in $\left(\frac{n}{2},n\right]$ which are $1\!\!\mod 3$. We can therefore apply Lemma \ref{doublinglemmad} to $A_{(\frac{1}{2},1]}^{1(3)}$ to obtain the lemma below. First, we recall the construction of the set $B_{\frac{1}{2}}$ from the proof of Lemma \ref{lemma4}. Note that for every $a\in A_{[\frac{1}{2}]}$ there is a power of $2$, say $2^{p_a}$, such that $2^{p_a}a\in \left(\frac{n}{4},\frac{n}{2}\right]$ and let $B_{\frac{1}{2}}\vcentcolon=\left\{2^{p_a}a:a\in A_{[\frac{1}{2}]}\right\}\subset\left(\frac{n}{4},\frac{n}{2}\right]$. The map $a\mapsto 2^{p_a}a$ is injective by Lemma \ref{basicInteger} so \begin{equation}
\left|B_{\frac{1}{2}}\right|=\left|A_{[\frac{1}{2}]}\right|.
\label{B_1/2size}
\end{equation}
\begin{lemma}
Let $B_{\frac{1}{2}}^{i(3)}=B_{\frac{1}{2}}\cap(i+3\!\cdot\!\mathbf{N})$ for $i=0,1,2$. Then
\begin{align}
    \left|B_{\frac{1}{2}}^{2(3)}\right|&\leqslant \frac{n}{12}+2-\frac{\left|A_{(\frac{1}{2},1]}^{1(3)}\right|}{2} \label{final1}\\
    \left|B_{\frac{1}{2}}^{1(3)}\right|&\leqslant \frac{n}{10}+2-\frac{2\left|A_{(\frac{1}{2},1]}^{1(3)}\right|}{5}. \label{final2}
\end{align}
Furthermore, if case $(3)$ in Lemma \ref{lemmaZ} holds, then either $\left|A_{(\frac{2}{3},1]}\right|\leqslant \frac{n}{9}+4$, or we additionally have the following inequality
\begin{equation}
    \left|B_{\frac{1}{2}}^{1(3)}\right|\leqslant \frac{11n}{90}+4-\frac{2\left|A_{(\frac{1}{2},1]}^{1(3)}\right|}{5}-\frac{\left|A_{(\frac{2}{3},1]}\right|}{6}.
    \label{final2case3}
\end{equation}
\label{B_lbound}
\end{lemma}
\begin{proof}
By \eqref{A_u,1dense}, $A_{(\frac{1}{2},1]}^{1(3)}$ has density much greater than a half on the progression $\left(\frac{n}{2},n\right]\cap (1+3\cdot\!\mathbf{N})$ so it satisfies the assumptions of Lemma \ref{doublinglemmad} with $d=3$ and $q=4$ or $q=5$. By applying Lemma \ref{doublinglemmad} with $q=4$, we get that $A_{(\frac{1}{2},1]}^{1(3)}+A_{(\frac{1}{2},1]}^{1(3)}$ contains at least $\frac{\left|A_{(\frac{1}{2},1]}^{1(3)}\right|}{2}-1$ multiples of 4.
As all numbers in $A_{(\frac{1}{2},1]}^{1(3)}+A_{(\frac{1}{2},1]}^{1(3)}$ are $2\!\!\mod 3$ and lie in $(n,2n]$ we get 
\begin{align*}
\left|\left(A_{(\frac{1}{2},1]}^{1(3)}+A_{(\frac{1}{2},1]}^{1(3)}\right)\cap(8+12\cdot \mathbf{N})\right|\geqslant \frac{\left|A_{(\frac{1}{2},1]}^{1(3)}\right|}{2}-1.
\end{align*} The set $4\cdot B_{\frac{1}{2}}^{2(3)}$ also consists only of numbers in $(n,2n]$ which are $8\!\!\mod 12$, so Lemma \ref{basicmultl} gives $\left|B_{\frac{1}{2}}^{2(3)}\right|+\frac{\left|A_{(\frac{1}{2},1]}^{1(3)}\right|}{2}-1\leqslant \frac{n}{12}+1$ which is the desired inequality \eqref{final1}. By applying Lemma \ref{doublinglemmad} with $q=5$, $A_{(\frac{1}{2},1]}^{1(3)}+A_{(\frac{1}{2},1]}^{1(3)}$ contains at least $\frac{2\left|A_{(\frac{1}{2},1]}^{1(3)}\right|}{5}-1$ multiples of $5$ and as every number in $A_{(\frac{1}{2},1]}^{1(3)}+A_{(\frac{1}{2},1]}^{1(3)}$ is $2\!\!\mod 3$, we get that
\begin{align*}
\left|\left(A_{(\frac{1}{2},1]}^{1(3)}+A_{(\frac{1}{2},1]}^{1(3)}\right)\cap(5+15\cdot \mathbf{N})\right|\geqslant \frac{2\left|A_{(\frac{1}{2},1]}^{1(3)}\right|}{5}-1.
\end{align*}
The set $5\cdot\left( B_{\frac{1}{2}}^{1(3)}\cap\left[\frac{2n}{5}\right]\right)$ consists only of numbers in $\left(n,2n\right]$ which are $5\!\!\mod 15$ so Lemma \ref{basicmultl} gives \begin{equation}
\left|B_{\frac{1}{2}}^{1(3)}\cap\left[\frac{2n}{5}\right]\right|+\frac{2\left|A_{(\frac{1}{2},1]}^{1(3)}\right|}{5}-1< \frac{n}{15}+1.
\label{lemmaB_bound}
\end{equation} Combining this with a trivial bound $\left|B_{\frac{1}{2}}^{1(3)}\cap\left(\frac{2n}{5},\frac{n}{2}\right]\right|< \frac{n}{30}+1$ on the number of integers in $\left(\frac{2n}{5},\frac{n}{2}\right]$ that are $1\!\!\mod 3$, we get the required inequality \eqref{final2}.
\medskip

Our final task is to prove that either $\left|A_{(\frac{2}{3},1]}\right|\leqslant \frac{n}{9}+4$ or \eqref{final2case3} follows under the assumption that (3) in Lemma \ref{lemmaZ} holds. So assume that case (3) in Lemma \ref{lemmaZ} holds, then $\left(A_{(\frac{2}{3},1]}+A_{(\frac{2}{3},1]}\right)\cap (4\!\cdot\!\mathbf{N})$ contains an arithmetic progression $Q'$ of size $|Q'|\geqslant\frac{\left|A_{(\frac{2}{3},1]}\right|}{2}-1$. If the common difference is at least 12, then $|Q'|\leqslant  \frac{n}{18}+1$ as $Q'\subset A_{(\frac{2}{3},1]}+A_{(\frac{2}{3},1]}\subset\left(\frac{4n}{3},2n\right]$, so then we would get $\frac{\left|A_{(\frac{2}{3},1]}\right|}{2}-1\leqslant|Q'|\leqslant  \frac{n}{18}+1$ implying the desired conclusion that $\left|A_{(\frac{2}{3},1]}\right|\leqslant \frac{n}{9}+4$. Otherwise the common difference is $4$ or $8$ so that the progression $Q'\subset \left(A_{(\frac{2}{3},1]}+A_{(\frac{2}{3},1]}\right)\cap (4\!\cdot\!\mathbf{N})$ contains at least $\frac{\left|A_{(\frac{2}{3},1]}\right|}{6}-1$ numbers which are $4\!\! \mod 12$. Note that $4\cdot\left( B_{\frac{1}{2}}^{1(3)}\cap\left(\frac{2n}{5},\frac{n}{2}\right]\right)$ is also a subset of $\left(\frac{4n}{3},2n\right]$ containing only numbers which are $4\!\! \mod 12$ so Lemma \ref{basicmultl} yields
\begin{equation*}
    \left| B_{\frac{1}{2}}^{1(3)}\cap\left(\frac{2n}{5},\frac{n}{2}\right]\right|\leqslant \frac{n}{18}+2-\frac{\left|A_{(\frac{2}{3},1]}\right|}{6}.
\end{equation*}
Combining this improved bound with \eqref{lemmaB_bound} yields the desired inequality \eqref{final2case3}
\begin{align*}
    \left| B_{\frac{1}{2}}^{1(3)}\right|&= \left| B_{\frac{1}{2}}^{1(3)}\cap\left(\frac{2n}{5},\frac{n}{2}\right]\right|+\left|B_{\frac{1}{2}}^{1(3)}\cap\left[\frac{2n}{5}\right]\right|\\
    &\leqslant \frac{n}{18}+2-\frac{\left|A_{(\frac{2}{3},1]}\right|}{6}+\frac{n}{15}+2-\frac{2\left|A_{(\frac{1}{2},1]}^{1(3)}\right|}{5} \\
    &\leqslant \frac{11n}{90}+4-\frac{2\left|A_{(\frac{1}{2},1]}^{1(3)}\right|}{5}-\frac{\left|A_{(\frac{2}{3},1]}\right|}{6}.
\end{align*}
\end{proof}

Now we have all the main ingredients in place to finish the argument. By \eqref{B_1/2size}, we have \begin{align*}
    |A| &= \left|A_{[\frac{1}{2}]}\right|+\left|A_{(\frac{1}{2},1]}\right| = \left|B_{\frac{1}{2}}\right|+\left|A_{(\frac{1}{2},1]}\right|\\
     &= \left|B_{\frac{1}{2}}^{1(3)}\right|+\left|B_{\frac{1}{2}}^{2(3)}\right|+\left|B_{[\frac{1}{2}]}^{0(3)}\right|+\left|A_{(\frac{1}{2},1]}^{0(3)}\right|+\left|A_{(\frac{1}{2},1]}^{1(3)}\right|+\left|A_{(\frac{1}{2},1]}^{2(3)}\right|\\
     &= \left|B_{\frac{1}{2}}^{1(3)}\right|+\left|B_{\frac{1}{2}}^{2(3)}\right|+\left|A\cap(3\!\cdot\!\mathbf{N})\right|+\left|A_{(\frac{1}{2},1]}^{1(3)}\right|+\left|A_{(\frac{1}{2},1]}^{2(3)}\right|
\end{align*}
so using the induction hypothesis \eqref{inductionboundmultiples} to bound $\left|A\cap(3\!\cdot\!\mathbf{N})\right|$, Lemma \ref{B_lbound} to bound $\left|B_{\frac{1}{2}}^{1(3)}\right|$ and $\left|B_{\frac{1}{2}}^{2(3)}\right|$ and \eqref{Mainbound} to bound $\left|A_{(\frac{1}{2},1]}^{1(3)}\right|$, we get
\begin{align}
    |A| &\leqslant \left(\frac{n}{12}+2-\frac{\left|A_{(\frac{1}{2},1]}^{1(3)}\right|}{2}\right) +\left(
    \frac{n}{10}+2-\frac{2\left|A_{(\frac{1}{2},1]}^{1(3)}\right|}{5}\right)+\frac{n}{9}+C\nonumber\\
    &\,\,\,\,\,\,+ \left(\left|A_{(\frac{2}{3},1]}\right|+3-2\left|A_{(\frac{1}{2},1]}^{2(3)}\right|-\left|Z_B\right|\right)+\left|A_{(\frac{1}{2},1]}^{2(3)}\right|\nonumber\\
     &= \frac{53n}{180}+C+7+\frac{\left|A_{(\frac{1}{2},1]}^{1(3)}\right|}{10}+ \left|A_{(\frac{2}{3},1]}\right|-\left|A_{(\frac{1}{2},1]}^{1(3)}\right|-\left|A_{(\frac{1}{2},1]}^{2(3)}\right|-\left|Z_B\right|.\label{ja}
\end{align}
Writing $\left|A_{(\frac{2}{3},1]}\right|=\left|A_{(\frac{1}{2},1]}\right|-\left|A_{(\frac{1}{2},\frac{2}{3}]}\right|$ and using that $\left|A_{(\frac{1}{2},1]}^{1(3)}\right|+\left|A_{(\frac{1}{2},1]}^{2(3)}\right|\geqslant \frac{2\left|A_{(\frac{1}{2},1]}\right|}{3}$ in Subcase 3.1 yields
\begin{align}
    |A| &\leqslant \frac{53n}{180}+C+7+\frac{\left|A_{(\frac{1}{2},1]}^{1(3)}\right|}{10}+ \frac{\left|A_{(\frac{1}{2},1]}\right|}{3}-\left|A_{(\frac{1}{2},\frac{2}{3}]}\right|-\left|Z_B\right|.
\label{penu}
\end{align}
$A_{(\frac{1}{2},1]}^{1(3)}$ consists only of numbers in $\left(\frac{n}{2},n\right]$ which are $1\!\!\mod 3$ so we can trivially bound $\left|A_{(\frac{1}{2},1]}^{1(3)}\right|\leqslant  \frac{n}{6}+1$. Using \eqref{A_2lowe} we can also bound $\left|A_{(\frac{1}{2},1]}\right|\leqslant \frac{3\left|A_{(\frac{2}{3},1]}\right|}{2}+1\leqslant \frac{n}{4}+37$ as we assume that $\left|A_{(\frac{2}{3},1]}\right|\leqslant \frac{n}{6}+24$ in Case 3. If one of (1) or (2) in Lemma \ref{lemmaZ} holds, then by Corollary \ref{cor1} the inequality \eqref{cor11} holds so $\left|A_{(\frac{1}{2},\frac{2}{3}]}\right|+\left|Z_B\right|\geqslant\frac{n}{12}-8$ and we can plug this in in \eqref{penu} to get in total that
\begin{align*}
    |A| &\leqslant \frac{53n}{180}+C+7+\frac{\frac{n}{6}+1}{10}+ \frac{\frac{n}{4}+37}{3}-\frac{n}{12}+8 < \frac{14n}{45}+C+30,
\end{align*}
and we are done.
The only remaining case to consider is when (3) in Lemma \ref{lemmaZ} holds, so by Corollary \ref{cor1} we have that $\left|A_{(\frac{1}{2},\frac{2}{3}]}\right|+|Z_B|\geqslant \frac{n}{24}-11$.
As (3) in Lemma \ref{lemmaZ} holds, Lemma \ref{B_lbound} gives either that $\left|A_{(\frac{2}{3},1]}\right|\leqslant \frac{n}{9}+4$ or that the bound \eqref{final2case3} on $\left|B_{\frac{1}{2}}^{1(3)}\right|$ holds. First if $\left|A_{(\frac{2}{3},1]}\right|\leqslant \frac{n}{9}+4$ then by \eqref{A_2lowe} we get $\left|A_{(\frac{1}{2},1]}\right|\leqslant \frac{3\left|A_{(\frac{2}{3},1]}\right|}{2}+1\leqslant \frac{n}{6}+7$. So plugging this in in \eqref{penu} gives
\begin{align*}
     |A| &\leqslant \frac{53n}{180}+C+7+\frac{\left|A_{(\frac{1}{2},1]}^{1(3)}\right|}{10}+ \frac{\left|A_{(\frac{1}{2},\frac{2}{3}]}\right|}{3}-\left|A_{(\frac{1}{2},1]}\right|-\left|Z_B\right|\\
     &\leqslant \frac{53n}{180}+C+7+\frac{\frac{n}{6}+1}{10}+ \frac{\frac{n}{6}+7}{3}-\frac{n}{24}+11\\
     &\leqslant \frac{13n}{40}+C+22,
\end{align*}
as desired.
Finally, we may assume that \eqref{final2case3} holds by Lemma \ref{B_lbound} and then we use it to bound $\left|B_{\frac{1}{2}}^{1(3)}\right|$. We bound the remaining quantities in the same way as we did in \eqref{ja}, meaning that we use the induction hypothesis \eqref{inductionboundmultiples} to bound $\left|A\cap(3\!\cdot\!\mathbf{N})\right|$, Lemma \ref{B_lbound} to bound $\left|B_{\frac{1}{2}}^{2(3)}\right|$ and \eqref{Mainbound} to bound $\left|A_{(\frac{1}{2},1]}^{1(3)}\right|$. This gives
\begin{align*}
    |A| =& \left|B_{\frac{1}{2}}^{1(3)}\right|+\left|B_{\frac{1}{2}}^{2(3)}\right|+\left|A\cap(3\!\cdot\!\mathbf{N})\right|+\left|A_{(\frac{1}{2},1]}^{1(3)}\right|+\left|A_{(\frac{1}{2},1]}^{2(3)}\right| \\
    \leqslant& \left(\frac{11n}{90}+4-\frac{2\left|A_{(\frac{1}{2},1]}^{1(3)}\right|}{5}-\frac{\left|A_{(\frac{2}{3},1]}\right|}{6}\right) +\left(\frac{n}{12}+2-\frac{\left|A_{(\frac{1}{2},1]}^{1(3)}\right|}{2}\right)\\
    &+\frac{n}{9}+C+\left(\left|A_{(\frac{2}{3},1]}\right|-2\left|A_{(\frac{1}{2},1]}^{2(3)}\right|-|Z_B|+3\right)+\left|A_{(\frac{1}{2},1]}^{2(3)}\right|\\
    =& \frac{57n}{180}+C+9+\frac{\left|A_{(\frac{1}{2},1]}^{1(3)}\right|}{10}+ \frac{5\left|A_{(\frac{2}{3},1]}\right|}{6}-\left|A_{(\frac{1}{2},1]}^{1(3)}\right|-\left|A_{(\frac{1}{2},1]}^{2(3)}\right|-\left|Z_B\right|\\
    \leqslant& \frac{19n}{60}+C+9+\frac{\left|A_{(\frac{1}{2},1]}^{1(3)}\right|}{10}+\frac{5\left|A_{(\frac{2}{3},1]}\right|}{6}-\frac{2\left|A_{(\frac{1}{2},1]}\right|}{3}-|Z_B|\\
    =& \frac{19n}{60}+C+9+\frac{\left|A_{(\frac{1}{2},1]}^{1(3)}\right|}{10}+\frac{\left|A_{(\frac{2}{3},1]}\right|}{6}-\frac{2\left|A_{(\frac{1}{2},\frac{2}{3}]}\right|}{3}-|Z_B|\\
    \leqslant& \frac{19n}{60}+C+10+\frac{n}{90}+\frac{\left|A_{(\frac{2}{3},1]}\right|}{6}-\frac{17\left|A_{(\frac{1}{2},\frac{2}{3}]}\right|}{30}-|Z_B|\\
    \leqslant& \frac{59n}{180}+C+10+\frac{\frac{n}{6}+24}{6}-\frac{17}{30}\left(\frac{n}{24}-11\right)\\
    \leqslant& \frac{239n}{720}+C+22,
\end{align*}
using the assumption of Subcase 3.1 that $\left|A_{(\frac{1}{2},1]}^{1(3)}\right|+\left|A_{(\frac{1}{2},1]}^{2(3)}\right|\geqslant\frac{2\left|A_{(\frac{1}{2},1]}\right|}{3}=\frac{2\left|A_{(\frac{2}{3},1]}\right|}{3}+\frac{2\left|A_{(\frac{1}{2},\frac{2}{3}]}\right|}{3}$ for lines 5 and 6 in the display above, for the penultimate inequality that $\left|A_{(\frac{1}{2},1]}^{1(3)}\right|= \left|A_{(\frac{1}{2},\frac{2}{3}]}^{1(3)}\right|+\left|A_{(\frac{2}{3},1]}^{1(3)}\right|\leqslant \left|A_{(\frac{1}{2},\frac{2}{3}]}\right|+\frac{n}{9}+1$ as $A_{(\frac{1}{2},\frac{2}{3}]}^{1(3)}\subset A_{(\frac{1}{2},\frac{2}{3}]}$ and as there are at most $\frac{n}{9}+1$ numbers in $\left(\frac{2n}{3},n\right]$ congruent to $1\!\!\mod 3$, and for the final inequality that $\left|A_{(\frac{1}{2},\frac{2}{3}]}\right|+|Z_B|\geqslant \frac{n}{24}-11$ by (3) in Corollary \ref{cor1} and that $\left|A_{(\frac{2}{3},1]}\right|\leqslant\frac{n}{6}+24 $ in Case 3.
This finishes the proof of Subcase 3.1.
\vfill
\pagebreak

\begin{flushleft}
\textbf{{\large Subcase 3.2: $\left|A_{(\frac{1}{2},1]}^{0(3)}\right|\geqslant \frac{\left|A_{(\frac{1}{2},1]}\right|}{3}$.}}
\end{flushleft}
\medskip

In Subcase 3.1, we applied Theorem \ref{freimain} to the sumset $A_{(\frac{1}{2},1]}^{1(3)}+A_{(\frac{1}{2},1]}^{2(3)}$ and proved the desired bound $|A|\leqslant \max\left(\left\lceil\frac{n}{3}\right\rceil, \left(\frac{1}{3}-\delta\right)n+C\right)$ under both possible conclusions of this theorem. We now apply Theorem \ref{freimain} to the sumset $A_{(\frac{1}{2},1]}^{0(3)}+A_{(\frac{1}{2},1]}^{0(3)}$ instead and we employ a very similar proof strategy as in Subcase 3.1, except that the argument here can be simplified in various places. First suppose that conclusion (1) in Theorem \ref{freimain} holds, so $\left|A_{(\frac{1}{2},1]}^{0(3)}+A_{(\frac{1}{2},1]}^{0(3)}\right|\geqslant 3\left|A_{(\frac{1}{2},1]}^{0(3)}\right|-3$ and the proof is very easy in this case because we immediately obtain many multiples of $3$ in $A_{(\frac{1}{2},1]}+A_{(\frac{1}{2},1]}$.\footnote{In Subcase 3.1, conclusion (1) in Theorem \ref{freimain} gave us that $\left|A_{(\frac{1}{2},1]}^{1(3)}+A_{(\frac{1}{2},1]}^{2(3)}\right|\geqslant\left|A_{(\frac{1}{2},1]}^{1(3)}\right|+\left|A_{(\frac{1}{2},1]}^{2(3)}\right|+\min\left(\left|A_{(\frac{1}{2},1]}^{1(3)}\right|,\left|A_{(\frac{1}{2},1]}^{2(3)}\right|\right)-3$. This provides a weaker lower bound on the number of multiples of $3$ in $A_{(\frac{1}{2},1]}+A_{(\frac{1}{2},1]}$ because $\min\left(\left|A_{(\frac{1}{2},1]}^{1(3)}\right|,\left|A_{(\frac{1}{2},1]}^{2(3)}\right|\right)$ can be small even if $\left|A_{(\frac{1}{2},1]}^{1(3)}\right|+\left|A_{(\frac{1}{2},1]}^{2(3)}\right|\geqslant\frac{2\left|A_{(\frac{1}{2},1]}\right|}{3}$.} Indeed, we then have that \begin{align*} \left|\left(A_{(\frac{1}{2},1]}+A_{(\frac{1}{2},1]}\right)\cap 3\cdot \mathbf{N}\right|&\geqslant \left|A_{(\frac{1}{2},1]}^{0(3)}+A_{(\frac{1}{2},1]}^{0(3)}\right|\\
&\geqslant3\left|A_{(\frac{1}{2},1]}^{0(3)}\right|-3\\
&\geqslant \left|A_{(\frac{1}{2},1]}\right|-3
\end{align*} by the assumption of Subcase 3.2. By the definition \eqref{A''definition} of $A'''$, we deduce the lower bound $\left|A'''\right|\geqslant \left|A_{(\frac{1}{2},1]}\right|-3$. Plugging this bound in in \eqref{dyadicembed} then gives
\begin{align*}
\left\lceil\frac{n}{3}\right\rceil &\geqslant \left|A'''\right|+|B_1|\\
&\geqslant \left|A_{(\frac{1}{2},1]}\right|-3+|B_1|\\
&= \left|A_{(\frac{1}{2},1]}\right|-3+\left|A_{[\frac{2}{3}]}\right|\\
&= |A|-3+\left|A_{(\frac{1}{2},\frac{2}{3}]}\right|
\end{align*} where we used that $|B_1|=\left|A_{[\frac{2}{3}]}\right|$ by \eqref{Adecompo}. So we can plug in the lower bound \eqref{A_2genelowe} on $\left|A_{(\frac{1}{2},\frac{2}{3}]}\right|$ in this inequality and deduce the desired result.

\bigskip

For the remainder of the proof in Subcase 3.2, we may now assume that (2) in Theorem \ref{freimain} holds so that $A_{(\frac{1}{2},1]}^{0(3)}+A_{(\frac{1}{2},1]}^{0(3)}$ contains an arithmetic progression $Q$ of size $2\left|A_{(\frac{1}{2},1]}^{0(3)}\right|-1$ and common difference $d = \gcd_*\left(A_{(\frac{1}{2},1]}^{0(3)}\right)$. The proof in this case is essentially the same as the proof that we used in Subcase 3.1 under the assumption that the sumset $A_{(\frac{1}{2},1]}^{1(3)}+A_{(\frac{1}{2},1]}^{2(3)}$ satisfied conclusion (2) in Theorem \ref{freimain}. Hence, for the proofs of cases 3.2.1, 3.2.2 and 3.2.3 we simply refer back to the the corresponding sections 3.1.1, 3.1.2 and 3.1.3 of Subcase 3.1 while pointing out the minor adaptations of the argument that are required. First we show that $d=3,6$ or $9$. Since $A_{(\frac{1}{2},1]}^{0(3)}$ contains only multiples of 3 by definition, we see that $3|d$. By \eqref{inductionbound}, we may assume that $\left|A_{(\frac{1}{2},1]}\right|>  \left(\frac{1}{3}-\delta\right)\frac{n}{2}$. By the assumption of Subcase 3.2, this implies that \begin{equation}
\left|A_{(\frac{1}{2},1]}^{0(3)}\right|\geqslant \frac{\left|A_{(\frac{1}{2},\frac{2}{3}]}\right|}{3}\geqslant \left(\frac{1}{3}-\delta\right)\frac{n}{6}.
\label{A_u,0bound}
\end{equation} As $A_{(\frac{1}{2},1]}^{0(3)}\subset\left(\frac{n}{2},n\right]$ lies in a progression with common difference $d$, we have that  $\left|A_{(\frac{1}{2},1]}^{0(3)}\right|\leqslant  \left\lceil\frac{n}{2d}\right\rceil$.  Comparing these upper and lower bounds on $\left|A_{(\frac{1}{2},1]}^{0(3)}\right|$ shows that $d=\gcd_*\left(A_{(\frac{1}{2},1]}^{0(3)}\right) = 3, 6$ or $9$, and the remaining part of the proof of Subcase 3.2 consists of dealing with these three cases separately.
\medskip
\begin{flushleft}
\textbf{{3.2.1: Let $\gcd_*\left(A_{(\frac{1}{2},1]}^{0(3)}\right) = 9$.}}
\end{flushleft}
\medskip

In this case, $A_{(\frac{1}{2},1]}^{0(3)}\subset\left(\frac{n}{2},n\right]$ lies in a progression with common difference $d=9$, so suppose that every number in $A_{(\frac{1}{2},1]}^{0(3)}$ is congruent to $a\!\!\mod 9$ for one of $a=0,3$ or $6$. It cannot be the case that $a=0$ as by \eqref{A_u,0bound}, $A_{(\frac{1}{2},1]}^{0(3)}$ would then contain $\left(\frac{1}{3}-\delta\right)\frac{n}{6}$ multiples of 9 in $\left(\frac{n}{2},n\right]$, but it is easy to see that such a set does not have property P for $n$ sufficiently large. Indeed, after dividing by $9$ we obtain the set $\frac{1}{9}\!\cdot\!A_{(\frac{1}{2},1]}^{0(3)}\subset\left(\frac{m}{2},m\right]$ where $m=\left\lfloor \frac{n}{9}\right\rfloor$ and it is enough to show that this set does not have property P. Let $s= \min \frac{1}{9}\!\cdot\!A_{(\frac{1}{2},1]}^{0(3)}$ and as $\left|\frac{1}{9}\!\cdot\!A_{(\frac{1}{2},1]}^{0(3)}\right|\geqslant \left(\frac{1}{3}-\delta\right)\frac{n}{6}$, we have that $\frac{m}{2}<s\leqslant m-\left(\frac{1}{3}-\delta\right)\frac{n}{6}+1< \frac{m}{2}+\frac{m}{100}$ for $\delta>0$ sufficiently small as $m=\left\lfloor \frac{n}{9}\right\rfloor$. The set $\frac{1}{9}\!\cdot\!A_{(\frac{1}{2},1]}^{0(3)}$ also contains at least $\left(\frac{1}{3}-\delta\right)\frac{n}{6}-1>\frac{m}{4}+\frac{m}{200}>\frac{s}{2}$ numbers in $(s,m]\subset(s,2s]$ and hence there there exist $x,y\in \frac{1}{9}\!\cdot\!A_{(\frac{1}{2},1]}^{0(3)}$ with $x,y>s$ and $x+y\equiv 0 \!\! \mod s$ showing that $\frac{1}{9}\!\cdot\!A_{(\frac{1}{2},1]}^{0(3)}$ does not have property P, a contradiction.
So $a=3$ or $a=6$. Then $A_{(\frac{1}{2},1]}^{0(3)}+A_{(\frac{1}{2},1]}^{0(3)}$ contains at least $2\left|A_{(\frac{1}{2},1]}^{0(3)}\right|-1\geqslant \left(\frac{1}{3}-\delta\right)\frac{n}{3}-1$ numbers in $(n,2n]$ which are $6\!\!\mod 9$ if $a=3$, or $3\!\!\mod 9$ if $a=6$. So $A_{(\frac{1}{2},1]}^{0(3)}+A_{(\frac{1}{2},1]}^{0(3)}$ contains all except at most $\left\lceil\frac{n}{9}\right\rceil -\left(\frac{1}{3}-\delta\right)\frac{n}{3}-1\leqslant \frac{\delta n}{3}+2$ of the numbers in $(n,2n]$ which are congruent to $3$ or $6$ modulo $9$. From the definition \eqref{A''definition}, $A'''$ therefore contains all except at most $\frac{\delta n}{3}+2$ of the numbers in $\left(\frac{n}{3},\frac{2n}{3}\right]$ which are $2\!\!\mod 3$ if $a=3$, or $1\!\!\mod 3$ if $a=6$. This is precisely what we deduced in case 3.1.1 when $a_1+a_2\equiv 3,6\!\!\mod 9$, and the argument that we used there works here too. The only modification required is that here we deduce the inequality $\left|A_{(\frac{1}{2},1]}\right|\leqslant3\left\lceil\frac{n}{18}\right\rceil$ by using that $\left|A_{(\frac{1}{2},1]}\right|\leqslant 3\left|A_{(\frac{1}{2},1]}^{0(3)}\right|\leqslant 3\left\lceil\frac{n}{18}\right\rceil$ by the assumption of Subcase 3.2 and as $A_{(\frac{1}{2},1]}$ lies in a progression with common difference $9$ by the assumption of case 3.2.1. \qed

\medskip
\begin{flushleft}
\textbf{{3.2.2: Let $\gcd_*\left(A_{(\frac{1}{2},1]}^{0(3)}\right) = 6$.}}
\end{flushleft}
\medskip

As $\gcd_*\left(A_{(\frac{1}{2},1]}^{0(3)}\right) = 6$, we have that $A_{(\frac{1}{2},1]}^{0(3)}\subset \left(\frac{n}{2},n\right]$ lies in a progression with common difference 6, so either $A_{(\frac{1}{2},1]}^{0(3)}\subset \left(\frac{n}{2},n\right]\cap(6\!\cdot\!\mathbf{N})$ or $A_{(\frac{1}{2},1]}^{0(3)}\subset\left(\frac{n}{2},n\right]\cap(3+6\!\cdot\!\mathbf{N})$. In both cases, we see that the sumset $A_{(\frac{1}{2},1]}^{0(3)}+A_{(\frac{1}{2},1]}^{0(3)}$ consists of multiples of $6$ only, and hence the progression $Q\subset A_{(\frac{1}{2},1]}^{0(3)}+A_{(\frac{1}{2},1]}^{0(3)}$ that we obtained from the conclusion (2) in Theorem \ref{freimain} is a progression with common difference 6 consisting of multiples of 6. We also have that $|Q|\geqslant 2\left|A_{(\frac{1}{2},1]}^{0(3)}\right|-1\geqslant \frac{2\left|A_{(\frac{1}{2},1]}\right|}{3}-1$ by the assumption of Subcase 3.2. The existence of such a progression $Q\subset A_{(\frac{1}{2},1]}+A_{(\frac{1}{2},1]}$ of size $|Q|\geqslant \frac{2\left|A_{(\frac{1}{2},1]}\right|}{3}-1$ with common difference 6 consisting of multiples of 6 is precisely what we needed to make the argument that we used in the corresponding case 3.1.2 work, so the rest of the proof that $|A|\leqslant \max\left(\left\lceil\frac{n}{3}\right\rceil, \left(\frac{1}{3}-\delta\right)n+C\right)$ in case 3.2.2 is identical to that argument.
\qed
\medskip
\begin{flushleft}
\textbf{{3.2.3: Let $\gcd_*\left(A_{(\frac{1}{2},1]}^{0(3)}\right) = 3$.}}
\end{flushleft}
\medskip

If $\gcd_*\left(A_{(\frac{1}{2},1]}^{0(3)}\right) = 3$, then the progression that we obtain from conclusion (2) in Theorem \ref{freimain} is an arithmetic progression $Q\subset A_{(\frac{1}{2},\frac{2}{3}]}+A_{(\frac{1}{2},\frac{2}{3}]}$ with common difference 3 consisting of multiples of 3 in $(n,2n]$ and having size $|Q|\geqslant 2\left|A_{(\frac{1}{2},1]}^{0(3)}\right|-1\geqslant\frac{2\left|A_{(\frac{1}{2},1]}\right|}{3}-1$. This is precisely the information that we used in case 3.1.3 to prove that $|A|\leqslant\max\left(\left\lceil\frac{n}{3}\right\rceil, \left(\frac{1}{3}-\delta\right)n+C\right)$ and the proof here is exactly the same. \qed
\medskip

Hence, we have proved the desired bound $|A|\leqslant\max\left(\left\lceil\frac{n}{3}\right\rceil, \left(\frac{1}{3}-\delta\right)n+C\right)$ in each of the three cases 3.2.1, 3.2.2 and 3.2.3 thus finishing the proof of Subcase 3.2. This completes the proof of Theorem \ref{theo5}.
\vfill
\pagebreak

\bibliographystyle{plain}
\bibliography{referP}
\bigskip

\noindent
{\sc Mathematical Institute, Andrew Wiles Building, University of Oxford, Radcliffe
Observatory Quarter, Woodstock Road, Oxford, OX2 6GG, UK.}\newline
\href{mailto:benjamin.bedert@magd.ox.ac.uk}{\small benjamin.bedert@magd.ox.ac.uk}
\end{document}